\documentclass[10pt]{amsart}
  \usepackage{amsmath,amscd,latexsym,verbatim,amssymb}
 \usepackage{times} 
\usepackage[all]{xy}

 \newcommand{\resp}{{\it resp.} }
\newcommand{\cf}{{\it cf.} }
\newcommand{\ie}{{\it i.e.} }
\newcommand{\eg}{{\it e.g.} }
\newcommand{\loccit}{{\it loc. cit.} }

\renewcommand{\qed}{\hfill$\Box$\medskip}
\newcommand{\e}{\frac{1}{p^\infty}}
\newcommand{\f}{-\frac{1}{p^\infty}}
\newcommand{\s}{\sm o}
 \renewcommand{\ss}{\sm o \sm o}

\newcommand{\sA}{\mathcal{A}}
\newcommand{\sB}{\mathcal{B}}
\newcommand{\sC}{\mathcal{C}}

\newcommand{\sG}{\mathcal{G}}
\newcommand{\sH}{\mathcal{H}}
\newcommand{\sI}{\mathcal{I}}

\newcommand{\sL}{\mathcal{L}}
\newcommand{\sM}{\mathcal{M}}
\newcommand{\sK}{\mathcal{K}}

\newcommand{\Q}{\mathbf{Q}}
\newcommand{\Z}{\mathbf{Z}}
\newcommand{\F}{\mathbf{F}}
\newcommand{\N}{\mathbf{N}}
\newcommand{\R}{\mathbf{R}}
 \newcommand{\inj}{\hookrightarrow}

\newcommand{\surj}{\rightarrow\!\!\!\!\!\rightarrow}

\newcommand{\car}{\operatorname{car}}
\renewcommand{\Im}{\operatorname{Im}}
\renewcommand{\ker}{\rm{ker}\,}

\newcommand{\Coker}{{\rm{Coker}}\,}
\newcommand{\Spec}{\rm{Spec}\,}

\renewcommand{\epsilon}{\varepsilon}
 
\font\sm=cmr10 at9pt

\swapnumbers

\newtheorem{thm}{Th\'eor\`eme}[subsection]
\newtheorem{lemma}[thm]{Lemme}

\newtheorem{prop}[thm]{Proposition}

\newtheorem{cor}[thm]{Corollaire}
 
\newtheorem{sorite}[thm]{Sorite}
\theoremstyle{definition}
\newtheorem{defn}[thm]{Definition}

\newtheorem{qn}[thm]{Question}

\numberwithin{equation}{section}

 at13pt
\font\sm=cmr9 at6pt

\begin{document}

 \title[Lemme d'Abhyankar perfectoide]{Le lemme d'Abhyankar perfectoide}

\author{Yves
Andr\'e}

\address{Institut de Math\'ematiques de Jussieu\\  4 place Jussieu, 75005
Paris\\France.}
\email{yves.andre@imj-prg.fr}
   \date{23 ao\^ut 2016}

 \begin{sloppypar}
 
   \bigskip 
   
   \bigskip 
    
  \medskip \begin{abstract} 
    Nous \'etendons le th\'eor\`eme de presque-puret\'e de Faltings-Scholze-Kedlaya-Liu sur les extensions \'etales finies d'alg\`ebres perfecto\"{\i}des au cas des extensions ramifi\'ees, sans 
  restriction sur le discriminant. Le point cl\'e est une version perfecto\"{\i}de du th\'eor\`eme d'extension de Riemann. 
    Au pr\'ealable, nous revenons sur les aspects cat\'egoriques des alg\`ebres de Banach uniformes et des alg\`ebres perfecto\"{\i}des. 
   \bigskip
   
\noindent{\sm{ABSTRACT.}}  We extend Faltings's ``almost purity theorem" on finite etale extensions of perfectoid algebras (as generalized by Scholze and Kedlaya-Liu) to the ramified case, without 
 restriction on the discriminant. The key point is a perfectoid version of Riemann's extension theorem.   
  Categorical aspects of uniform Banach algebras and perfectoid algebras are revisited beforehand.
  \end{abstract}

     \bigskip 
        \bigskip  
   \maketitle
  \let\languagename\relax

\tableofcontents

 \newpage

  \bigskip    \section*{Introduction.} 
      
    \bigskip
    
    \bigskip  \subsection{} La th\'eorie des extensions presque \'etales de G. Faltings trouve sa source dans les travaux de J. Tate sur les extensions profond\'ement ramifi\'ees de corps locaux, et son aboutissement dans l'\'etude de la cohomologie \'etale $p$-adique des vari\'et\'es alg\'ebriques ou analytiques sur un corps $p$-adique (\cf \eg \cite{Fa}\cite{O}). Elle utilise le langage de la {``presque-alg\`ebre"} introduit par Faltings \`a ce propos, dont le principe consiste,  dans le cadre $(\mathfrak V, \mathfrak m)$ constitu\'e d'un anneau et d'un id\'eal idempotent, \`a {``n\'egliger"} les $\mathfrak V$-modules annul\'es par $\mathfrak m\,$ \cite{GR1}.
      
      Cette th\'eorie a \'et\'e renouvel\'ee  
      par l'approche perfecto\"{i}de de P. Scholze. Un corps $\sK$ complet pour une valuation $p$-adique non discr\`ete, d'anneau de valuation $\sK^{\s}$ et d'id\'eal de valuation $\sK^{\ss}$, est dit {\it perfecto\"{\i}de} si l'\'el\'evation \`a la puissance $p$ est surjective sur $\sK^{\s}/p\sK^{\s}$. 
         Une $\sK$-alg\`ebre de Banach $\sA$ est dite {\it perfecto\"{\i}de} si elle est uniforme (\ie si la sous-$\sK^{\s}$-alg\`ebre $\sA^{\s}$ des \'el\'ements dont les puissances sont born\'ees est born\'ee), et si l'\'el\'evation \`a la puissance $p$ est surjective sur $\sA^{\s}/p\sA^{\s}$.  
               
      \subsection{} Le th\'eor\`eme de {``presque-puret\'e"} de Faltings, sous la forme g\'en\'erale que lui ont donn\'ee Scholze d'une part \cite{S1}, et K. Kedlaya et R. Liu d'autre part \cite{KL}, affirme que {\it si $\sA$ est une $\sK$-alg\`ebre perfecto\"{\i}de et $\sB$ une $\sA$-alg\`ebre \'etale finie, alors $\sB$ est perfecto\"{\i}de et $\sB^{\s}$ est presque \'etale finie sur $\sA^{\s}$ dans le cadre $(\sK^{\s}, \sK^{\ss})$}, de sorte que $tr_{\sB/\sA}$ envoie $\sB^{\s}$ presque surjectivement dans $\sA^{\s}$ si $\sB^{\s}$ est un $\sA^{\s}$-module fid\`ele.

       \subsection{}\label{obj} L'objectif de l'article est de g\'en\'eraliser ce th\'eor\`eme au {\it cas ramifi\'e}, en s'inspirant du lemme d'Abhyankar. Rappelons que dans la situation d'un anneau local r\'egulier d'in\'egale caract\'eristique $(0,p)$ et d'une extension finie plate, ramifi\'ee le long d'un diviseur \`a croisements normaux d\'efini par une \'equation $pg=0$, d'indices de ramification premiers \`a $p$, ce lemme assure qu'on peut {``rendre l'extension \'etale"} sans inverser $pg$, par l'adjonction de racines de $pg$ d'ordre divisible par tous les indices de ramification, suivie du passage \`a la fermeture int\'egrale.        
        
      \smallskip Notre r\'esultat principal peut \^etre vu comme un analogue perfecto\"{\i}de du lemme d'Abhyankar, sans hypoth\`ese de croisements normaux ni condition sur les indices de ramification, au prix de remplacer {``\'etale"} par {``presque \'etale"} (dans un cadre qui, contrairement \`a l'usage, n'est pas celui d'un anneau de valuation et de son id\'eal maximal).

            \begin{thm}\label{T1}  Soit $\sA$ une alg\`ebre perfecto\"{\i}de sur un corps perfecto\"{\i}de $\sK$ de caract\'eristique r\'esiduelle $p$.  On suppose que $\sA$ contient une suite de racines $p^m$-\`emes compatibles d'un \'el\'ement $g\in \sA^{\s} $  
      non diviseur de z\'ero, ce qui permet de voir $\sA^{\s}$ comme $\sK^{\s}[T^{\e}]$-alg\`ebre en envoyant $  T^{\frac{1}{p^i}}$ sur $  g^{\frac{1}{p^i}}$. On se place dans le cadre $\,(\sK^{\s}[T^{\e}], T^{\e}\sK^{\ss}[T^{\e}])$.
      
    \smallskip  Soit $\sB'$ une $\sA[\frac{1}{g}]$-alg\`ebre \'etale finie.
      \begin{enumerate} \item Il existe une plus grande alg\`ebre perfecto\"{\i}de $\,\sB\,$ comprise entre $\,\sA\,$ et $\,\sB'$, telle que l'inclusion $\,\sA\inj \sB\,$ soit continue. On a $\,\sB[\frac{1}{g}]= \sB'$, et $\,\sB^{\s}\,$ est contenue dans la fermeture int\'egrale de $\,g^{\f}\sA^{\s}\, $ dans $\,\sB'$ et lui est presque isomorphe,   
     \item Pour tout $m\in \N$, $\sB^{\s}/p^m\,$ est presque \'etale finie sur $\,\sA^{\s}/p^m\,$ (et $\sB^{\s}[\frac{1}{g}]$ est presque \'etale finie sur $\,\sA^{\s}[\frac{1}{g}]$ dans le cadre $(\sK^{\s}, \sK^{\ss}))$. 
    \item Si $\sB'$ est un $\sA[\frac{1}{g}]$-module fid\`ele, le morphisme trace $\,tr_{\sB'/\sA[\frac{1}{g}]}\,$ envoie $\,\sB^{\s}\,$ presque surjectivement dans $\,g^{\f}\sA^{\s}$.
       \end{enumerate} \end{thm}

 
 \medskip  Ce r\'esultat pr\'esente un double aspect heuristique: pour toute extension finie d'une alg\`ebre perfecto\"{\i}de, $i)$ {\it la ramification g\'eom\'etrique est presque supprim\'ee par adjonction des racines $p^\infty$-i\`emes du discriminant, compl\'etion et passage \`a une fermeture 
  int\'egrale} (en analogie avec le lemme d'Abhyankar), et $ii)$ {\it il en est alors essentiellement de m\^eme de la ramification arithm\'etique} (ce qui \'etend le th\'eor\`eme de presque-puret\'e de Faltings {\it et al.} qui traite du cas g\'eom\'etriquement non ramifi\'e).   
  
\smallskip En ce qui concerne le point $(2)$, on peut naturellement se demander si $\sB^{\s} \,$ est presque \'etale finie sur $\,\sA^{\s} $. C'est plausible, mais nous ne l'avons pas d\'emontr\'e, le probl\`eme \'etant celui de la presque finitude de $\sB^{\s} \,$ sur $\sA^{\s} \,$, \cf rem. \ref{fr}.

       \subsection{} Nous d\'emontrerons ce th\'eor\`eme \`a partir du cas \'etale, appliqu\'e aux {``localisations affino\"{\i}des"}   $\sA\{ \frac{\lambda}{g}\},\; \lambda \in \sK^{\ss}\setminus 0\,$ (c'est-\`a-dire, en termes g\'eom\'etriques, aux compl\'ementaires de voisinages tubulaires du lieu des z\'eros de $\,g\,$ dans le spectre analytique de $\sA$). 
       Le passage crucial \`a la limite $\lambda \to 0$ repose sur des r\'esultats du type {``th\'eor\`eme d'extension de Riemann"}: {\it pour $(\sA, g)$ comme ci-dessus,  $\,\displaystyle  \lim_{j\to \infty}\,  {\sA}\{ \frac{p^j}{g} \}^{\s}    $ est la fermeture compl\`etement int\'egrale de $\sA^{\s}$ dans $\sA[\frac{1}{g}]$,  qui n'est autre que $\,g^{\f}\sA^{\s}  \,$} (th. \ref{T5}).   
             
 Inversement, {\it \'etant donn\'e un syst\`eme projectif d'alg\`ebres perfecto\"{\i}des $(\sA^{j})$ qui se d\'eduisent les unes des autres par localisation ($\,\displaystyle \sA^{i}\{ \frac{p^j}{g}\}   \cong \sA^{j}$)  et contiennent les racines $g^{\frac{1}{p^m}}$,  l'alg\`ebre de Banach uniforme $\,\displaystyle \sA :=  (\lim_{j \to \infty}\,  {\sA}^{j\,\s}) [\frac{1}{p}]$ est presque perfecto\"{\i}de} (\ie l'\'el\'evation \`a la puissance $p$ est presque surjective sur $\sA^{\s}/p\sA^{\s}$) (th. \ref{T6}).

   \subsection{}   Il nous a paru n\'ecessaire de faire pr\'ec\'eder la preuve de ces r\'esultats d'un examen approfondi, assorti de maint exemple, des techniques de construction et des aspects cat\'egoriques des alg\`ebres de Banach uniformes et des alg\`ebres perfecto\"{\i}des. Ces deux cat\'egories admettent limites et colimites, et le plongement de celle-ci dans celle-l\`a admet un adjoint \`a droite. Les colimites s'y calculent donc de m\^eme mani\`ere, mais pas les limites (\cf \S \ref{lp}). C'est le n\oe ud du probl\`eme dans les passages \`a la limite ci-dessus. 
   
   Nous ferons grand usage de techniques galoisiennes (\cf \ref{eg}, \ref{eeap}, \ref{tp}), qui permettent entre autre de simplifier la preuve du th\'eor\`eme de {``presque-puret\'e"}. 

     \subsection{} Le th\'eor\`eme \ref{T1} fournit un moyen aussi \'economique que g\'en\'eral pour {\it associer \`a toute $\sK$-alg\`ebre affine ou affino\"{\i}de r\'eduite une alg\`ebre (presque) perfecto\"{\i}de \`a l'aide d'une normalisation de Noether, en adjoignant les racines $p^\infty$-i\`emes des coordonn\'ees et d'un discriminant, compl\'etant et prenant une fermeture int\'egrale}. Dans une certaine mesure, ce moyen lib\`ere la construction d'alg\`ebres perfecto\"{\i}des du {``carcan torique"}
   habituel (\cf \ref{nN}). 
   
   Dans un article suivant, nous appliquerons le lemme d'Abhyankar perfecto\"{\i}de \`a l'\'etude de la conjecture du facteur direct de M. Hochster en alg\`ebre commutative.

      \bigskip
          
    \bigskip
\begin{small} {\it Remerciements.}  Je suis tr\`es reconnaissant \`a O. Gabber de m'avoir envoy\'e toute une liste d'erreurs, corrections, simplifications et suggestions (les \'eventuelles erreurs restantes sont \'evidemment de ma seule responsabilit\'e). Je remercie aussi B. Bhatt et P. Scholze de leur int\'er\^et  pour ce travail.  
  
  Je remercie le Laboratoire Fibonacci et la Scuola Normale Superiore de Pise o\`u j'ai effectu\'e une partie de ce travail dans les meilleures conditions. \end{small}

 \newpage
  \section{Pr\'eliminaires de {presque-alg\`ebre}.}\label{pa} Nous ferons appel \`a quelques r\'esultats de cette th\'eorie initi\'ee par G. Faltings et d\'evelopp\'ee dans l'ouvrage fondamental de O. Gabber et L. Ramero \cite{GR1}. 
      
        \subsection{Cadre.} Un {\it cadre}\footnote{``Basic setting" dans  \loccit} 
          pour la presque-alg\`ebre consiste en la donn\'ee d'un anneau commutatif unitaire ${\mathfrak V}$ et d'un id\'eal idempotent $\mathfrak m = \mathfrak m^2$. 
          
          Comme dans \cite{GR1}, nous supposerons que 
          {\it $\tilde{\mathfrak m} := \mathfrak m \otimes_{\mathfrak V} \mathfrak m$ est plat sur $\mathfrak V$}. Cette hypoth\`ese est stable par changement de base \cite[rem. 2.1.4]{GR1}, et plus faible que la platitude de $\mathfrak m $ (qui entra\^{\i}ne $\mathfrak m   \cong \tilde{\mathfrak m} $). 
         La platitude de $\mathfrak m$ est acquise si 
         $\mathfrak m$ est de la forme $\pi^{\e}{\mathfrak V}$, o\`u $(\pi^{\frac{1}{p^i}})$ est une suite compatible de racines $p^i$-i\`emes d'un \'el\'ement $\pi $ non-diviseur de z\'ero de l'anneau $\mathfrak V$ (cas qui suffirait dans la suite de l'article).

       \medskip\subsection{${\mathfrak V}^a$-modules.}\label{Mod}  La cat\'egorie des ${\mathfrak V}^a$-modules (ou $({\mathfrak V}, \mathfrak m)^a$-modules, s'il y a lieu de pr\'eciser) est la cat\'egorie des ${\mathfrak V}$-modules localis\'ee par la sous-cat\'egorie de Serre des modules de $\mathfrak m$-torsion: un ${\mathfrak V}$-module est {\it presque nul} s'il est annul\'e par $\mathfrak m$.   
       
        On a $Hom_{{\mathfrak V}^a}(M,N)= Hom_{\mathfrak V}({\tilde{\mathfrak m}} \otimes_{\mathfrak V} M, N),\,$  \cf \cite[\S 2.2.2]{GR1}. 
      
      On dit qu'un morphisme de ${\mathfrak V}$-modules est presque injectif (\resp presque surjectif) si son noyau (\resp conoyau) est annul\'e par $\mathfrak m$; cela revient \`a dire qu'il est un monomorphisme (\resp \'epimorphisme) de ${\mathfrak V}^a$-modules.
      
      La cat\'egorie des ${\mathfrak V}^a$-modules est ab\'elienne. En particulier, un morphisme de ${\mathfrak V}$-modules est \`a la fois presque injectif et presque surjectif si et seulement s'il est un presque-isomorphisme. 
         Le foncteur de localisation $M\mapsto M^a$, qui est l'identit\'e sur les objets, admet\footnote{l'existence et les propri\'et\'es de $(\;)_\ast$ font partie de la th\'eorie g\'en\'erale de la localisation \cite{Ga}; en revanche, l'existence de $(\;)_!$ est sp\'ecifique \`a la presque-alg\`ebre.} 
   \begin{itemize} \item   un adjoint \`a droite: $N\mapsto N_\ast := {\rm{Hom}}_{{\mathfrak V}^a}({\mathfrak V}^a, N) \,$ (module des {\it presque-\'el\'ements}),
      \item un adjoint \`a gauche $N_! :=  \tilde{\mathfrak m} \otimes_{\mathfrak V} N_\ast $.
      \end{itemize}
 En particulier, il commute aux limites et aux colimites. On a en outre  $(N_\ast)^a\cong N$ et  $ (M^a)_\ast \cong  {\rm{Hom}}_{{\mathfrak V}}(\tilde{\mathfrak m}, M)$ (qu'on \'ecrira parfois abusivement $M_\ast$ pour all\'eger).  
  Le foncteur $(\,)_\ast$ n'est pas exact \`a droite, mais $N \to N'$ est un \'epimorphisme si et seulement si $N_\ast \to N'_\ast$ est presque surjectif.

\medskip\subsection{Lemmes de Mittag-Leffler et de Nakayama.}\label{MLN}  Le lemme de Mittag-Leffler vaut dans ce contexte: {\it si $(N^{n})$ un syst\`eme projectif de ${\mathfrak V}^a$-modules dont les morphismes de transition sont des \'epimorphismes, alors $\lim \, N^{n}\to N^{0}$ est un \'epimorphisme}. Cela se d\'eduit du cas des ${\mathfrak V}$-modules en appliquant successivement les foncteurs exacts \`a droite $(\,)_!$ et $(\,)^a$ (dont le compos\'e est l'identit\'e de ${\mathfrak V}^a\hbox{-}{\bf{Mod}}$).

 Au demeurant, la cat\'egorie ab\'elienne ${\mathfrak V}^a\hbox{-}{\bf{Mod}}$ est bicompl\`ete (\ie admet limites et colimites) et admet $\mathfrak V$ comme g\'en\'erateur, les produits d'\'epimorphismes y sont des \'epimorphismes, donc les $\lim^i\, $ se comportent comme dans la cat\'egorie des groupes ab\'eliens \cite{Ro}; en particulier, elles s'annulent pour $i>1$ dans le cas d'un syst\`eme index\'e par un ensemble ordonn\'e d\'enombrable. La localisation commute aux $\lim^1\, $, qui se calculent comme conoyau du t\'elescope habituel.
 
  \smallskip La version du lemme de Nakayama pour les modules complets vaut aussi dans ce contexte (\cite[lem. 5.3.3]{GR1}):  {\it soit $\mathfrak I$ un id\'eal de $\mathfrak V$ et soit $f : M \to N $ un morphisme entre ${\mathfrak V}^a$-modules ${\mathfrak I}$-adiquement complets; alors $f$ est un \'epimorphisme s'il l'est mod. ${\mathfrak I}$}.

En effet, on a un diagramme commutatif \`a fl\`eches horizontales \'epimorphiques 
\[ \begin{CD} gr_{\mathfrak I} {\mathfrak V} \otimes_{{\mathfrak V}/{\mathfrak I}} M/{\mathfrak I}M@>  >>  gr_{\mathfrak I} M \\@V{1\otimes \bar f}VV @VV{gr_{\mathfrak I} f}V     \\\ gr_{\mathfrak I} {\mathfrak V} \otimes_{{\mathfrak V}/{\mathfrak I}} N/{\mathfrak I}N @>  >>  gr_{\mathfrak I} N , \end{CD}\] 
   donc l'hypoth\`ese entra\^ine que $gr_{\mathfrak I} f$ est un \'epimorphisme. Ceci implique que $M/{\mathfrak I}^nM \stackrel{f_n}{\to} N/{\mathfrak I}^n N$ est un \'epimorphisme, de m\^eme que ${\ker} f_{n+1} \to {\ker} f_n$ par le lemme du serpent. Passant \`a la limite, $f  $ est donc un \'epimorphisme d'apr\`es Mittag-Leffler.

 \subsection{${\mathfrak V}^a$-alg\`ebres.}\label{Alg} Les ${\mathfrak V}^a$-alg\`ebres sont les mono\"{\i}des dans ${\mathfrak V}^a\hbox{-}{\bf{Mod}}$. Le foncteur $(\; )^a$ envoie ${\mathfrak V}\hbox{-}{\bf{Alg}}$ vers ${\mathfrak V}^a\hbox{-}{\bf{Alg}}$ et admet (la restriction de) $(\;)_\ast$ comme adjoint \`a droite. Il admet aussi un adjoint \`a gauche not\'e $(\;)_{!!}$
 \cite[\S 2.2.25]{GR1}.
 
\smallskip La cat\'egorie  ${\mathfrak V}^a\hbox{-}{\bf{Alg}}$ admet des sommes amalgam\'ees,  
et $(\mathfrak B\otimes_{\mathfrak A} \mathfrak C)^a\cong \mathfrak B^a\otimes_{\mathfrak A^a} \mathfrak C^a$. 
  
  Si $\mathfrak A$ est une ${\mathfrak V}^a$-alg\`ebre, on d\'efinit des cat\'egories $\mathfrak A\hbox{-}{\bf{Mod}}$ et $\mathfrak A\hbox{-}{\bf{Alg}}$  \cite[2.2.12]{GR1}. Le foncteur $(\; )^a$ envoie ${\mathfrak A}_\ast\hbox{-}{\bf{Alg}}$ vers ${\mathfrak A}\hbox{-}{\bf{Alg}}$ et admet $(\;)_\ast$ comme adjoint \`a droite.  
    
\smallskip Rappelons que la cat\'egorie des $\sK$-alg\`ebres (associatives commutatives unif\`eres) admet des (petites) limites et colimites, et que le foncteur module sous-jacent pr\'eserve les limites, ainsi que les colimites filtrantes. Via $(\;)^a$, il en est de m\^eme pour les presque alg\`ebres. Une d\'ecomposition en produit fini $\mathfrak B = \prod {\mathfrak B}_i$ correspond \`a un syst\`eme complet d'idempotents orthogonaux $ e_i \in {\mathfrak B}_\ast$.

    La cat\'egorie $\mathfrak A\hbox{-}{\bf{Mod}}$ admet un produit tensoriel $\otimes_{\mathfrak A}$,  
  qui induit le coproduit de $\mathfrak A\hbox{-}{\bf{Alg}}$. 
  
  \smallskip Soit $\phi:  \mathfrak A  {\to} \mathfrak B$ un morphisme dans ${\mathfrak V}^a\hbox{-}{\bf{Alg}}$. C'est un monomorphisme si et seulement si $\phi_\ast$ est presque injectif; on dit alors que $\mathfrak B$ est une {\it extension} de $\mathfrak A$.

   \subsubsection{Exemple} La presque-alg\`ebre s'introduit naturellement dans le contexte suivant. Soient $\sK$ un corps complet pour une valuation non discr\`ete,  $\mathfrak V = \sK^{\s}$ son anneau de valuation,  $\mathfrak m = \sK^{\ss}$ son id\'eal de valuation, $\varpi$ un \'el\'ement non nul de $\mathfrak m$, et $\mathfrak A$ une $\mathfrak V$-alg\`ebre sans $\varpi$-torsion ni \'el\'ement infiniment $\varpi$-divisible. On peut alors munir la $\sK$-alg\`ebre $\mathfrak A[\frac{1}{\varpi}]$ d'une norme canonique, dont la boule unit\'e, en g\'en\'eral distincte de $\mathfrak A$ mais presque isomorphe \`a $\mathfrak A$, est $(\mathfrak A)^a_\ast$, \cf sorite \ref{s1} (2b).
  
    \subsection{Recadrage}\label{rec} Au rebours de la pratique en presque-alg\`ebre, dans le pr\'esent travail,
      le cadre ne sera pas fix\'e une fois pour toutes; au contraire, nous utiliserons de mani\`ere essentielle des changements de cadre $({\mathfrak V}, \mathfrak m) \to ({\mathfrak V}', \mathfrak m')$ (homomorphismes d'anneaux envoyant $\mathfrak m$ dans  $\mathfrak m'$). Le foncteur exact de restriction des scalaires $ {\mathfrak V}' \hbox{-}{\bf{Mod}} \to   {\mathfrak V}\hbox{-}{\bf{Mod}}$ induit un  foncteur exact 

\centerline{ $ ({\mathfrak V}', \mathfrak m')^a\hbox{-}{\bf{Mod}} \to  ({\mathfrak V}, \mathfrak m)^a\hbox{-}{\bf{Mod}}; $ }
si ${\mathfrak V}= {\mathfrak V}'= {\mathfrak m}' $, c'est le foncteur $(\;)^a$.  On a de m\^eme un foncteur $({\mathfrak V}', \mathfrak m')^a\hbox{-}{\bf{Alg}} \to  ({\mathfrak V}, \mathfrak m)^a\hbox{-}{\bf{Alg}}$, qui est l'identit\'e sur les anneaux sous-jacents. Si $\mathfrak A'$ correspond \`a $\mathfrak A$ par ce dernier foncteur, on a des foncteurs de {\it recadrage}

\centerline{$\mathfrak A'\hbox{-}{\bf{Mod}} \to  \mathfrak A\hbox{-}{\bf{Mod}},\;\mathfrak A'\hbox{-}{\bf{Alg}} \to \mathfrak A\hbox{-}{\bf{Alg}} $} 
\noindent qui sont l'identit\'e sur les objets (vus comme modules, \resp alg\`ebres, sur l'anneau sous-jacent \`a $\mathfrak A$); si ${\mathfrak V}= {\mathfrak V}'= {\mathfrak m}' $, il s'agit encore du foncteur $(\;)^a$. Ces foncteurs sont des isomorphismes si $\mathfrak m$ engendre $\mathfrak m'$, puisque la presque-nullit\'e d'un module sur l'anneau sous-jacent \`a $\mathfrak A$ n'est pas modifi\'ee.
On a une transformation naturelle de {\it recadrage} 
\[ {\rm{Hom}}_{{\mathfrak V}' }(\tilde{\mathfrak m}' , N)\to  {\rm{Hom}}_{{\mathfrak V} }(\tilde{\mathfrak m} , N)\]
 entre foncteurs des presque-\'el\'ements et des presque-\'el\'ements du recadr\'e, respectivement.

  Heuristiquement, le sens de {``presque"} devient plus grossier par recadrage (passage de $\mathfrak m'$ \`a $\mathfrak m$), mais ne change pas si l'image de $\mathfrak m$ engendre $\mathfrak m'$.

 \subsection{Platitude.} Un $\mathfrak A$-module $M$ est dit {\it plat}  si l'endofoncteur
 {$ - \otimes_{\mathfrak A} M$ de   $\mathfrak A\hbox{-}{\bf{Mod}}$}
  est exact. 
 
  On dit que $\phi$ est {\it (fid\`element) plat} si  
 {$  \mathfrak B \otimes_{\mathfrak A} -:  \mathfrak A\hbox{-}{\bf{Mod}}\to \mathfrak B\hbox{-}{\bf{Mod}}$}
 est (fid\`ele) exact; $\phi$ est plat (\resp fid\`element plat) si et seulement si $\mathfrak B$ est un $\mathfrak A$-module plat (\resp $\phi$ est un monomorphisme et $\mathfrak B/\mathfrak A$ est un $\mathfrak A$-module plat \cite[3.1.2 (vi)]{GR1}) -  c'est le cas en particulier si $\phi_\ast$ est plat (\resp fid\`element plat). En outre, $\phi$ est fid\`element plat si et seulement si $\phi_{!!}$ l'est \cite[3.1.3 ii]{GR1}. Toute colimite filtrante d'extensions (fid\`element) plates est (fid\`element) plate.

  \subsection{$\mathfrak A$-Modules projectifs finis.} Passons \`a de plus subtiles propri\'et\'es qui, contrairement aux pr\'ec\'edentes, ne sont pas purement cat\'egoriques.  
  
 \subsubsection{} Soit $\mathfrak A$ une $\mathfrak V^a$-alg\`ebre. Un 
    $\mathfrak A$-module $P$ est {\it projectif fini}\footnote{Nous adoptons les simplifications terminologiques de Scholze \cite[4.10, 4.13]{S1}; ce n'est d'ailleurs pas l\`a la d\'efinition de \loccit mais c'en est une caract\'erisation \cf \cite[lemma 2.4.15]{GR1}.} si pour tout $\eta\in \mathfrak m$, il existe $n\in \N$ et des morphismes $P \to \mathfrak A^n\to P$ dont le compos\'e est $\eta 1_P$ (cette notion ne d\'epend que de la classe d'isomorphisme de $P$ dans $\mathfrak A\hbox{-}{\bf{Mod}}$). 
        
    Il revient au m\^eme (\cite[2.3.10 (i), 2.4.15]{GR1}) de dire que 
    
    $(a)$ $P$ est {\it presque projectif}, \ie pour tout module $N$, et tout $i>0$, $\mathfrak m{\rm{Ext^i}}\,(P, N)=0 $,
     
    $(b)$ $P$ est {\it presque de type fini}, \ie pour tout $\eta\in \mathfrak m$, il existe un morphisme $\mathfrak A^{n(\eta)}\to P$ de conoyau annul\'e par $\eta$.
   
 \subsubsection{Remarques}\label{r1} $(1)$ Tout module $P$ presque projectif est plat  \cite[2.4.18]{GR1}, et fid\`element plat si $P$ est un $\mathfrak A$-module fid\`ele (\ie $\mathfrak A\to  End(P)^a$ est un monomorphisme)\footnote{Le crit\`ere \'equivalent donn\'e dans \cite[2.4.28]{GR1} est que le morphisme {``dual"} $End(P)^a \to \mathfrak A$ soit un \'epimorphisme.}. 
 
\smallskip \noindent $(2)$ Si $P$ est projectif fini, il en de m\^eme de ses puissances ext\'erieures. On dispose d'un morphisme trace $tr_{P/A}: ({\rm{End}}\, P)^a \to A$ \cite[4.1.1]{GR1}.   

  \smallskip \noindent $(3)$ Soit $\mathfrak A'$ une $\mathfrak A$-alg\`ebre fid\`element plate.  Alors $\mathfrak P$ est un $\mathfrak A$-module projectif fini si et seulement si  $P\otimes_{\mathfrak A} \mathfrak A'$ est un $\mathfrak A'$-module projectif fini  \cite[3.2.26 iii]{GR1}. 
  
  \begin{lemma}\label{L1} Soit $\mathfrak I$ un id\'eal de $\mathfrak A$ pour lequel $\mathfrak A$ est $\mathfrak I$-adiquement compl\`ete. Alors tout module projectif fini $P$ est $\mathfrak I$-adiquement complet.  \end{lemma}
  
  \begin{proof}\footnote{voir aussi \cite[5.3.5]{GR1}.} Il s'agit de voir que $P\to \hat P:= \lim P/\mathfrak I^n$ est un isomorphisme. Pour tout $\eta\in \mathfrak m$, il existe $n\in \N$ et un diagramme 
    \[   \begin{CD}  P@>   >>   \mathfrak A^n   @>  >>  P    \\       @V VV   @V  VV @VV  V     \\\   \hat{P}@>   >>    \hat{\mathfrak A}^n   @>  >>   \hat{P}   \end{CD}\] o\`u les deux fl\`eches compos\'ees horizontales sont la multiplication par $\eta$ et la fl\`eche verticale du milieu est un isomorphisme. Chassant, on trouve que le noyau et le conoyau de $P\to \hat P$ sont annul\'es par $\eta$.  
   \end{proof}
 
  \subsubsection{} Un $\mathfrak A$-module projectif fini $P$ est {\it de rang (constant) $r$} si \begin{small} $\displaystyle  \bigwedge^{r+1}$ \end{small}  $P =0$ et \begin{small} $\displaystyle  \bigwedge^{r }$ \end{small}  $P$ est un $\mathfrak A$-module inversible  \cite[4.3.9]{GR1}. D'apr\`es \cite[4.4.24]{GR1}, un tel module est localement libre de rang $r$ pour la topologie fpqc (nous n'utiliserons pas ce r\'esultat).

 \subsection{$\mathfrak A$-alg\`ebres \'etales finies.}  Une $\mathfrak A$-alg\`ebre $\mathfrak B$ est {\it  \'etale finie} (\resp {\it \'etale finie de rang $r$}) si 
  $\,(a)$ $\mathfrak B$ est un $\mathfrak A$-module 
  projectif fini  (\resp projectif fini de rang $r$),
 
 $(b)$ $\mathfrak B$ est {\it non ramifi\'ee}, \ie on a une d\'ecomposition de $\mathfrak A$-alg\`ebre
  $\;\mathfrak B\otimes_{\mathfrak A} \mathfrak B \cong \mathfrak B \times \mathfrak C\,$
  dans laquelle la premi\`ere projection (dans le membre de droite) s'identifie au morphisme de multiplication $\mu_{\mathfrak B}: \mathfrak B\otimes_{\mathfrak A} \mathfrak B\to \mathfrak B$. 
  
  Sous $(a)$, la condition $(b)$ \'equivaut \`a la platitude de $\mu_{\mathfrak B}$ \cite[3.1.2 vii, 3.1.9]{GR1}.

 \subsubsection{Remarques}\label{r1'} $(1)$ Toute extension \'etale finie $\mathfrak A\inj \mathfrak B$ est fid\`element plate.
 
\smallskip \noindent $(2)$ Si $\mathfrak B$ est \'etale finie sur $\mathfrak A$, on dispose d'une trace $Tr_{\mathfrak B/\mathfrak A}: \, \mathfrak B\to \mathfrak A $ (morphisme de $\mathfrak A$-modules) d\'efinie comme trace de l'endomorphisme de multiplication par $b$. Il commute au changement de base, et la composition avec la multiplication de $\mathfrak B$ induit un isomorphisme de $\mathfrak B$ sur son dual $\mathfrak B^\vee$ \cite[4.1.14]{GR1} (d'o\`u aussi un isomorphisme entre $\mathfrak B_\ast$ et son $\mathfrak A_\ast$-dual). Si $\mathfrak B$ est une extension \'etale finie, c'est un \'epimorphisme \cite[4.1.7, 4.1.8, 4.1.11]{GR1}.  

 \smallskip \noindent $(3)$ Soient $\mathfrak A', \mathfrak B$ deux $\mathfrak A$-alg\`ebres. On suppose $\mathfrak A \inj \mathfrak A'$ fid\`element plat. Alors $\mathfrak B$ est  \'etale finie (resp \'etale finie de rang $r$) sur $\mathfrak A$ si et seulement si $\mathfrak B\otimes_{\mathfrak A} \mathfrak A'$ l'est sur $\mathfrak A'$ \cite[2.4.18, 3.2.26 (ii)]{GR1}.

\subsubsection{}\label{eqrem} L'{``{\it \'equivalence remarquable}"} de Grothendieck vaut dans ce contexte, du moins si  $\mathfrak m$ est de la forme $\pi^{\e}{\mathfrak V}$ et si ${\mathfrak A}$ soit $\pi$-adiquement compl\`ete: la r\'eduction modulo $\pi$ induit alors une \'equivalence entre $\mathfrak A$-alg\`ebres \'etales finies et $\mathfrak A/\pi$-alg\`ebres \'etales finies \cite[th. 5.3.27]{GR1}.

 \subsection{Extensions galoisiennes.}\label{eg} 

\subsubsection{} Soient $\mathfrak A \inj \mathfrak B \inj \mathfrak C$ des extensions, et soit $X$ un ensemble de $\mathfrak A$-homomorphismes de $\mathfrak B$ vers $\mathfrak C$. On a alors un morphisme canonique de $\mathfrak C$-alg\`ebres

\smallskip \centerline{  $\mathfrak B \otimes_{\mathfrak A} \mathfrak C \to $  \begin{small} $\displaystyle  \prod_{\chi  \in X}$ \end{small}  $\mathfrak C ,  \;\;  \;b\otimes c \mapsto (\chi(b)c)_{\chi \in X}$ .}

 En particulier, si $G$ un groupe fini de $\mathfrak A$-automorphismes de $\mathfrak B$, on a un morphisme canonique de $\mathfrak B$-alg\`ebres (pour l'action \`a droite de $\mathfrak B$ sur le membre de gauche) 

\smallskip \centerline{ $\mathfrak B \otimes_{\mathfrak A} \mathfrak B \to $  \begin{small} $\displaystyle  \prod_{\gamma  \in G}$ \end{small}  $\mathfrak B ,  \;\;  \;b\otimes b' \mapsto (\gamma(b)b')_{\gamma\in G}$.} 
  
  \noindent Si $n$ est l'ordre de $G$, et si l'on num\'erote les \'el\'ements de $G$ pour identifier le groupe de ses permutations \`a  $\mathfrak S_n$, le produit en couronne\footnote{\ie le produit $\mathfrak S_n \times G^n$ muni de la multiplication 
  $(\sigma, (\gamma_1, \ldots, \gamma_n)) (\sigma', (\gamma_1', \ldots, \gamma_n'))= (\sigma\sigma', (\gamma_{\sigma'(1)} \gamma'_1, \ldots, \gamma_{\sigma'(n)} \gamma'_n))$.} $\,\mathfrak S_n \wr G \,$ agit par $\mathfrak A$-automorphismes de  
  {{\begin{small} $  \prod $ \end{small}  $\mathfrak B\,$}} 
  via 
$\;(\sigma, (\gamma_1, \ldots, \gamma_n))(b_1, \ldots, b_n) = (\gamma_{\sigma^{-1}(1)}(b_{\sigma^{-1}(1)}),\ldots ,\gamma_{\sigma^{-1}(n)}(b_{\sigma^{-1}(n)})). $

 On note $\mathfrak B$ la $\mathfrak A$-alg\`ebre des $G$-invariants de $\mathfrak B\,$ (en tant que $\mathfrak A$-module, c'est le noyau de $\mathfrak B \stackrel{(\cdots, \gamma-1, \cdots)}{\to}$\begin{small} $\displaystyle  \prod_{  G}$ \end{small}  $\mathfrak B$).

\subsubsection{}  On dit que l'extension $\mathfrak A\inj \mathfrak B$ est {\it galoisienne} de groupe $G$ si {\it $\mathfrak B^G= \mathfrak A$ et si le morphisme canonique $\mathfrak B \otimes_{\mathfrak A} \mathfrak B \to $  \begin{small} $\displaystyle  \prod_{  G}$ \end{small}  $\mathfrak B$ est un isomorphisme} (de $ \mathfrak B$-modules \`a droite). Comme $(\;)$ commute aux limites, il revient au m\^eme de requ\'erir que $(\mathfrak B_\ast)^G= \mathfrak A_\ast$ et que $(\mathfrak B \otimes_{\mathfrak A} \mathfrak B)_\ast \to $  \begin{small} $\displaystyle  \prod_{  G}$ \end{small}  $\mathfrak B_\ast$ soit un isomorphisme\footnote{cela n'entra\^{\i}ne ni que $\mathfrak B_\ast$ soit galoisienne sur $\mathfrak A_\ast$ (car $(\,)_\ast$ ne commute pas au produit tensoriel), ni que $\mathfrak B_{!!}$ soit galoisienne sur $\mathfrak A_{!!}$ (car $(\,)_{!!}$ ne commute pas aux produits finis).}. 
 Le groupe $ G\times G\ $ d'automorphismes de $\, \mathfrak B \otimes_{\mathfrak A} \mathfrak B  \cong $  \begin{small} $\displaystyle  \prod_{  G}$ \end{small}  $\mathfrak B  $ s'identifie alors \`a un sous-groupe de $ \mathfrak S_n \wr G $ par
\begin{equation}\label{eqg} (\gamma , 1) \mapsto (r_{\gamma^{-1}}, (1, \ldots, 1)), \;\; ( 1, \gamma) \mapsto  (\ell_\gamma\,, (\gamma, \ldots, \gamma)), \end{equation} 
o\`u $\ell, r$ d\'esignent les translations \`a gauche et \`a droite respectivement.

\begin{prop}\label{P2}
\begin{enumerate}
\item Toute extension galoisienne de groupe $G$ est \'etale finie, de rang l'ordre de $G$. La trace est donn\'ee par la somme des conjugu\'es.
\item  Soit $\mathfrak A'$ une $\mathfrak A$-alg\`ebre. Si $\mathfrak A \inj \mathfrak B$ est une extension galoisienne de groupe $G$, il en est de m\^eme de $\mathfrak A' \inj \mathfrak B\otimes_{\mathfrak A}\mathfrak A'$, et r\'eciproquement si $\mathfrak A'$ est fid\`element plate.  
\item Si $\mathfrak C$ est une extension interm\'ediaire telle que $\mathfrak C\inj \mathfrak B$ est galoisienne de groupe $H< G$, le morphisme canonique de $\mathfrak B$-alg\`ebres $\mathfrak C \otimes_{\mathfrak A} \mathfrak B \to $  \begin{small} $\displaystyle  \prod_{  G/H}$ \end{small}  $\mathfrak B  $ est un isomorphisme, et $\mathfrak A \inj \mathfrak C$ est \'etale finie, de rang $\vert G/H\vert$. 
 \end{enumerate}
\end{prop}  

\begin{proof} $(1)$ Le point est que $\mathfrak B$ est un $\mathfrak A$-module projectif fini, comme nous allons le voir (il s'agit d'un d\'ecalque de l'argument classique  en th\'eorie de Galois des anneaux).
 Notons $t_{\mathfrak B/\mathfrak A}: \mathfrak B_\ast\to \mathfrak B_\ast^G = \mathfrak A_\ast$ 
la somme des conjugu\'es. Voyons $\mathfrak B$ comme $\mathfrak B\otimes_{\mathfrak A} \mathfrak B$-alg\`ebre via le premier facteur de la d\'ecomposition (autrement dit, via le morphisme de multiplication).  Comme $\mathfrak B_\ast\otimes_{\mathfrak A_\ast} \mathfrak B_\ast\to (\mathfrak B\otimes_{\mathfrak A} \mathfrak B)_\ast$ est un presque-isomorphisme, on obtient, pour tout $\eta\in \mathfrak m$, un \'el\'ement $e_\eta \in \mathfrak B_\ast\otimes_{\mathfrak A_\ast} \mathfrak B_\ast$ tel que $e_\eta^2= \eta e_\eta$, qui annule le noyau de la multiplication  $\mathfrak B_\ast\otimes_{\mathfrak A_\ast} \mathfrak B_\ast\to \mathfrak B_\ast$ et que celle-ci envoie sur $ \eta 1_{\mathfrak B_\ast}$.  En d\'eveloppant $e_\eta = $  \begin{small}$\displaystyle \sum_{i=1}^{i=n(\eta)}\,$\end{small}   $ b_i\otimes b'_i$, on obtient $\sum \gamma(b_i) b'_i = 0$ si $\gamma\neq 1_G, $ et $\; \sum  b_i  b'_i = \eta 1_{\mathfrak B}$, d'o\`u $ \eta b = \sum t_{\mathfrak B/\mathfrak A}(bb_i)b'_i$ pour tout $b\in \mathfrak B_\ast $; en d'autres termes, le compos\'e 
\[\mathfrak B_\ast \stackrel{b\mapsto (t_{\mathfrak B/\mathfrak A}(bb_i))}{\to} \mathfrak A_\ast^{n(\eta)} \stackrel{(a_i) \mapsto \sum a_ib'_i}{\to} \mathfrak B_\ast\]
 est $ \eta \,id_{\mathfrak B_\ast}$, ce qui montre que $\mathfrak B$ est projectif fini.  
  Par descente fid\`element plate de $\mathfrak B$ \`a $\mathfrak A$, on voit aussi que $\mathfrak B$ est de rang $\vert G\vert$, et que
 $Tr_{\mathfrak B/\mathfrak A} = t_{\mathfrak B/\mathfrak A}^a$.  

  \smallskip \noindent $(2)$ Si $\mathfrak B$ est extension galoisienne de $\mathfrak A$ de groupe $G$, on a vu qu'elle est (fid\`element) plate, donc $\mathfrak B' :=  \mathfrak B\otimes_{\mathfrak A}\mathfrak A'$ est une extension de $\mathfrak A'$. On a $\mathfrak B' \otimes_{\mathfrak A'}\mathfrak B' \cong (\mathfrak B \otimes_{\mathfrak A}\mathfrak B)\otimes_{\mathfrak A} \mathfrak A' \cong $ \begin{small} $\displaystyle  \prod_{  G }$ \end{small}  $\mathfrak B'  $, d'o\`u aussi   $\mathfrak B'\cong (\mathfrak B' \otimes_{\mathfrak A'}\mathfrak B' )^{G\times 1} \cong  $ \begin{small} $\displaystyle\prod_{  G }$ \end{small}  $\mathfrak B'  \cong  \mathfrak B'^G\otimes_{\mathfrak A'}\mathfrak B' $  et   $(\mathfrak B')^G =   \mathfrak A'$ par platitude fid\`ele de $\mathfrak B'$ sur $\mathfrak A'$. On obtient la r\'eciproque par descente fid\`element plate.  
    
   \smallskip \noindent $(3)$ Apr\`es tensorisation $\otimes_{\mathfrak C} \mathfrak B$, ce morphisme canonique est  le compos\'e des isomorphismes de $\mathfrak C$-alg\`ebres 
    
  \centerline{$\mathfrak B \otimes_{\mathfrak C} (\mathfrak C\otimes_{\mathfrak A} \mathfrak B)\;  \;\stackrel{\sim}{\to}  \;  \;\mathfrak B \otimes_{\mathfrak A}\mathfrak B\; \stackrel{\sim}{\to}\;$  \begin{small} $\displaystyle  \prod_{  G }$ \end{small}  $\mathfrak B \;\stackrel{\sim}{\to}  \;   $  \begin{small} $\displaystyle  \prod_{  G/H }$ \end{small}  $(\mathfrak B \otimes_{\mathfrak C}\mathfrak B)  \;\stackrel{\sim}{\to}  \;  \mathfrak B \otimes_{\mathfrak C} $  \begin{small} $\displaystyle  \prod_{  G/H }$ \end{small}  $\mathfrak B,$ } 
  \noindent donc est un isomorphisme par descente fid\`element plate de $\mathfrak B $ \`a $\mathfrak C$ (\cf rem. \ref{r1'} $(3)$). En outre, $\mathfrak C$ est \'etale finie de rang $\vert G/H\vert$ sur $\mathfrak A$ par descente fid\`element plate de $\mathfrak B $ \`a $\mathfrak A$.\end{proof}
  
   \subsubsection{} Nous ferons usage de la construction suivante, standard dans le cas connexe (\cf \eg \cite[5.3.9]{Sz}) mais valable en g\'en\'eral\footnote{comme nous l'a signal\'e O. Gabber.}. 
   
 \begin{lemma}\label{L0} Soit $R\inj S$ une extension \'etale finie d'anneaux, de rang constant $r$. Il existe une extension galoisienne $R\inj T$ de groupe $\mathfrak S_r$ se factorisant par $S$, telle que $S\inj T$ soit une extension galoisienne de groupe $\mathfrak S_{r-1}$.
 \end{lemma} 

   Posons $X= {\Spec}\, R, \, Y= {\Spec}\, S, \, Z =$ compl\'ementaire dans $Y\times_X Y \times_X \cdots \times_X Y$ ($r$ facteurs) des diagonales partielles. Ces derni\`eres sont ouvertes et ferm\'ees dans le produit fibr\'e puisque $Y$ est \'etale sur $X$, donc il en est de m\^eme de $Z$ qui est par suite \'etale fini sur $X$ (et aussi sur $Y$ via la premi\`ere projection). Par ailleurs, $\mathfrak S_r$ agit par permutation des facteurs, et on v\'erifie sur les fibres que le morphisme de rev\^etements \'etales $\mathfrak S_r \times Z \to Z\times_X Z$ (\resp $\mathfrak S_{r-1} \times Z \to Z\times_Y Z$) est un isomorphisme.  \qed
   
   Nous aurons aussi besoin du lemme suivant
   
         \begin{lemma}\label{L3} Soit $G$ un groupe fini d'automorphismes d'un anneau $S$, et soit $R:= S^G $ l'anneau des invariants. Soit $R\subset S'\subset S$ une extension interm\'ediaire stable sous $G$, et galoisienne sur $R$ de groupe $G$. Alors $S= S'$.
         \end{lemma}     

  Les idempotents $e_g$ donnant lieu \`a la d\'ecomposition $\;  S'\otimes_R S' \cong  $ 
  {\begin{small}$\displaystyle{\prod_{\gamma \in G}}$\end{small}} 
  $\, S'$ d\'ecomposent aussi 
 $\displaystyle S \otimes_R S' \cong $  {\begin{small}$\displaystyle{\prod_{\gamma \in G}}$\end{small}} $\, e_\gamma(S \otimes_R S' ) \cong $    {\begin{small}$\displaystyle{\prod_{\gamma \in G}}$\end{small}} $\, S''$ (puisque que les $e_\gamma$ sont permut\'es par $G$). Compte tenu de ce que $S'$ est fid\`element plat sur $R$ (\cf prop. \ref{P2}), on a 
 $ S'' \cong  (S \otimes_R S')^{G\times 1}  =  S^G \otimes_{R} S' = S'$, d'o\`u $S'\otimes_R S'= S \otimes_R S' $,  puis $S= S'$.  
\qed

  \newpage
  \bigskip  \section{  La cat\'egorie bicompl\`ete des alg\`ebres de Banach uniformes.} 
    
 \medskip Dans cet article, le r\^ole principal est jou\'e par les alg\`ebres commutatives sur un corps complet non-archim\'edien, compl\`etes pour une norme multiplicative pour les puissances. Nous passons ici en revue leurs propri\'et\'es g\'en\'erales et esquissons, dans ce contexte, un dictionnaire entre le langage de l'analyse fonctionnelle et celui de l'alg\`ebre commutative.

  \subsection{Alg\`ebres de Banach} Nous renvoyons \`a \cite{BGR} comme r\'ef\'erence g\'en\'erale.
   
    \subsubsection{}\label{algn} Soit $\sK$ un corps complet pour une valeur absolue non archim\'edienne non triviale, de corps r\'esiduel $k$.   On note ${\sK}^{\s}$ son anneau (\resp  ${\sK}^{\ss}$ son id\'eal) de valuation, de sorte que $k= {\sK}^{\s}/{\sK}^{\ss}$.
    
    Dans cet article, toutes les $\sK$-alg\`ebres sont commutatives associatives unif\`eres (l'alg\`ebre nulle n'est pas exclue). Si $\sA$ est une $\sK$-alg\`ebre, une {\it semi-norme} $  \sA \stackrel{\vert\;\vert}{\to} \R\,$ de $\sK$-alg\`ebre v\'erifie $\vert a+b\vert\leq \max (\vert a\vert, \vert b\vert), \,\vert ab\vert \leq \vert a\vert \vert b\vert $ avec \'egalit\'e si $a\in \sK\,$, et $\vert 1_\sA\vert = 1$ si $\sA\neq 0$. C'est une {\it norme} si elle ne s'annule pas sur $\sA\setminus 0$. 
    On note $ \sA_{\vert\,\vert \leq r}$, ou plus simplement $ \sA_{\leq r}$, le sous-groupe additif de $\sA$ form\'e des  $a\in \sA \mid\, \vert a\vert \leq r\}$; de m\^eme  $ \sA_{< r}:= \{a\in \sA \mid\, \vert a\vert < r\}$. La boule unit\'e $ \sA_{\leq 1} $ est une sous-$\sK^{\s}$-alg\`ebre ouverte.  
    
    Toute $\sK$-alg\`ebre $\sA$ munie d'une norme multiplicative est int\`egre.
      
   \smallskip  La cat\'egorie des {\it $\sK$-alg\`ebres norm\'ees} a pour morphismes les {\it homomorphismes continus} d'alg\`ebres: ce sont les homomorphismes $\sA \stackrel{\phi}{\to} \sB$ born\'es sur $\sA_{\leq 1}$, \ie ceux pour lesquels $\vert\vert \phi\vert\vert := \sup_{a\in \sA \setminus 0}\, \frac{\vert \phi(a)\vert}{\vert a\vert} < \infty$. On dit que $\phi$ est {\it isom\'etrique}, ou que $\sA$ est une {\it sous-alg\`ebre norm\'ee} de $\sB$, si $\vert\phi(a)\vert= \vert a \vert$ pour tout $a\in \sA$.  
   
   Cette cat\'egorie admet des sommes amalgam\'ees: le s\'epar\'e de $\sB\otimes_\sA\, \sC$ pour la semi-norme produit tensoriel.
     
 \smallskip  La cat\'egorie $\,\sK\hbox{-}{\bf Ban}$ des {\it $\sK$-alg\`ebres de Banach} est la sous-cat\'egorie pleine form\'ee des alg\`ebres norm\'ees compl\`etes.
  Une $\sK$-alg\`ebre de Banach $\sA$ \'etant fix\'ee, les $\sK$-alg\`ebres de Banach $\sB$ munies d'un morphisme $\sA \stackrel{\phi}{\to} \sB$ sont appel\'ees {\it $\sA$-alg\`ebres de Banach}. Elles forment de mani\`ere \'evidente une cat\'egorie $\,\sA\hbox{-}{\bf Ban}$.
      
    \subsubsection{}\label{pt}   La cat\'egorie $\,\sK\hbox{-}{\bf Ban}$ admet un objet initial $\sK$ et un objet final $0$, des sommes amalgam\'ees $\sB\hat\otimes_\sA\, \sC$ ({\it produit tensoriel compl\'et\'e} \cite[3.1.1, prop. 2]{BGR}), et donc des colimites finies.      Le foncteur $ - \hat\otimes_\sK \sA$ est adjoint \`a gauche du foncteur oubli $\,\sA\hbox{-}{\bf Ban}\to \,\sK\hbox{-}{\bf Ban}$. 

Elle admet des produits fibr\'es $\sB \times_\sA\, \sC$, et donc des limites finies.
    
    Si $I$ est un id\'eal ferm\'e de $\sB$ et $J$ un id\'eal ferm\'e de $\sC$, et si $K$ est la fermeture de l'image de $I\otimes \sC + \sB \otimes J$ dans $\sB\hat\otimes_\sA\, \sC$, alors le morphisme canonique \begin{equation}\label{es} (\sB\hat\otimes_\sA\, \sC)/K\to (\sB/I)\hat\otimes_\sA\, (\sC/J)\end{equation} est une isom\'etrie (\cf \cite[2.1.8 prop. 6]{BGR} et sa preuve). 
    
    Par ailleurs, si $\sC$ est {\it projectif fini} sur $\sA$ et si sa norme induit la topologie canonique de $\sA$-module projectif fini, alors $\sB\otimes_\sA \sC$ est complet (cela suit de \cite[lemma 2.2.12]{KL}). Ce n'est plus vrai en g\'en\'eral si $\sC$ est fini non projectif. 
      
 \subsubsection{Remarques}\label{r.1} $(1)$ Si $\sC'$ est une sous-$\sA$-alg\`ebre norm\'ee de $\sC$, $\sB \hat\otimes_\sA\, \sC'$ n'est pas n\'ecessairement une sous-alg\`ebre norm\'ee de $\sB \hat\otimes_\sA\, \sC$.
 
\smallskip\noindent $(2)$ Supposons que $\sC$ admette une base orthogonale $ c_1, \ldots, c_n, \ldots$ sur $\sA$. Tout \'el\'ement $d\in \sB \hat\otimes_\sA\, \sC$ s'\'ecrit alors de mani\`ere unique $d = \sum b_j \otimes c_j$, et on a $\vert d\vert = \max \vert b_j\vert \vert c_j\vert$. En effet, pour toute autre expression $d = \sum b'_i \otimes c'_i$, on aura $c'_i = \sum a_{ij}c_j$, d'o\`u  $\vert c'_i \vert = \max \vert a_{ij}\vert \vert c_j\vert$, et $b_j = \sum a_{ij} b'_i$, d'o\`u $\max \vert b_j\vert \vert c_j\vert \leq \max  \vert a_{ij}\vert \vert c_j \vert \vert b'_i\vert \leq \max \vert c'_i \vert \vert b'_i\vert$. 
 On en d\'eduit que $\sB \stackrel{b\mapsto b\otimes 1}{\to} \sB \hat\otimes_\sA\, \sC$ est isom\'etrique.         
   
     \subsubsection{}\label{mb}   Si $\sA$ est une $\sK$-alg\`ebre norm\'ee, on note  $\sA^{\s}$ l'ensemble des {\it \'el\'ements dont les puissances restent born\'ees}. C'est une sous-$\sK^{\s}$-alg\`ebre ouverte de $\sA$ qui contient $\sA_{\leq 1}$, et qui ne d\'epend que de la classe d'\'equivalence de la norme. Sa formation est fonctorielle et commute \`a la compl\'etion.
      On a $(\sA \times \sB)^{\s}= \sA^{\s}\times \sB^{\s}$ (le produit \'etant muni de la norme supremum).  
     
     On note $\sA^{\ss}$ l'ensemble des {\it \'el\'ements topologiquement nilpotents}. C'est un id\'eal ouvert de $\sA^{\s}$,  et $\sA^{\ss}= \sqrt{\sA^{\ss}}$.      Si $\sA$ est compl\`ete, $\sA^{\ss}$ est contenu dans le radical de Jacobson de $\sA^{\s}$, car pour tout $(a,b) \in \sA^{\ss} \times \sA^{\s},\; 1-ab$ est inversible d'inverse $\sum a^n b^n$. En particulier, un \'el\'ement de $  \sA^{\s}$ est inversible si et seulement si son image dans $\sA^{\s}/\sA^{\ss}$ est inversible. Les id\'eaux maximaux de $\sA$ sont ferm\'es \cite[1.2.4/5]{BGR}.
     
     La r\`egle $\sA \mapsto \sA^{\s}/\sA^{\ss}$ d\'efinit un foncteur des $\sK$-alg\`ebres norm\'ees vers les $k$-alg\`ebres {\it r\'eduites} (\ie sans \'element nilpotent non nul).  
          
   \subsubsection{}  Pour toute $\sK$-alg\`ebre de Banach $\sA$, on note  $\sA\langle T\rangle$ le compl\'et\'e de l'alg\`ebre des polyn\^omes $\sA[T]$ pour la norme supremum des coefficients ({\it norme de Gauss}, qui est multiplicative si celle de $\sA$ l'est). Alors $\sA\langle T\rangle^{\s}$ est l'alg\`ebre des s\'eries convergentes \`a coefficients dans $\sA^{\s}$. Le couple $(\sA\langle T\rangle, T)$ est universel parmi les couples $(\sB , b\in \sB^{\s})$ o\`u $\sB$ est une $\sA$-alg\`ebre de Banach  (\cf \cite[1.4.3, cor. 2]{BGR}).       
     
     \subsubsection{Remarque}\label{r.3} Pour traiter des questions d'isom\'etrie, il est commode de se placer dans la sous-cat\'egorie non pleine de $\sK$-$\bf{Ban}$ ayant m\^eme objets et pour morphismes ceux de norme $\leq 1$, o\`u $\hat\otimes$ est encore un coproduit. Dans cette cat\'egorie, le couple $(\sA\langle T\rangle, T)$ est universel parmi les couples $(\sB , b\in \sB_{\leq 1})$.
     
   Si $\sA\to \sB$ est un morphisme dans cette cat\'egorie, le morphisme canonique $\sB \hat\otimes_\sA \sA\langle T\rangle \to \sB\langle T\rangle $ est isom\'etrique: l'inverse est donn\'e par la propri\'et\'e universelle de $\sB\langle T\rangle $. En combinant ceci \`a \eqref{es}, on trouve que pour tout id\'eal ferm\'e $J$ de $\sA\langle T\rangle$ le morphisme canonique 
   \begin{equation}\label{et}  \sB\langle T\rangle / J\sB\langle T\rangle^- \to \sB\hat\otimes_\sA (\sA\langle T\rangle/J)  \end{equation}
est isom\'etrique, $J\sB\langle T\rangle^-$ d\'esignant la fermeture de l'id\'eal $J\sB\langle T\rangle$.       
  
     \subsection{Normes spectrales} 
  
     \subsubsection{}\label{nsp} La {\it semi-norme spectrale}\footnote{Comme dans le cas archim\'edien, cette semi-norme a en effet une interpr\'etation {``spectrale"}, du moins dans le cas complet: c'est le supremum des semi-normes multiplicatives born\'ees, qui sont les points du spectre de Berkovich $\mathcal M(\sA)$.} associ\'ee \`a une norme $\vert\;\vert$ de $\sK$-alg\`ebre $\sA$ est d\'efinie par 
     \begin{equation}\label{e0}\,\displaystyle\vert a \vert_{sp} :=\lim_m\, \vert a^m\vert ^{\frac{1}{m}}\,\end{equation} (ce n'est pas n\'ecessairement une norme: on peut avoir $\vert a\vert_{sp} = 0$ pour $a$ non nul).
 On a $ \vert \; \vert_{sp} \leq \vert \;\vert $. Si $  \vert \; \vert_{sp} = \vert \;\vert $, c'est-\`a-dire si $\vert\;\vert$ est multiplicative pour les puissances, on dit que $\vert \;\vert $ est une {\it norme spectrale}, ou que $\sA$ est une {\it $\sK$-alg\`ebre norm\'ee spectrale}.
   Dans ce cas,  $\sA^{\s}$ co\"{\i}ncide avec $\sA_{\leq 1}$, et $\sA^{\ss}$ avec $\sA_{< 1}$.  
   
   Une norme de $\sK$-alg\`ebre $\sA$ est multiplicative si et seulement si elle est spectrale et $\sA^{\s}/\sA^{\ss}$ est int\`egre \cite[1.5.3]{BGR}\footnote{si $\vert \sK \vert \neq \vert \sA \vert$, il existe de telles alg\`ebres qui ne sont pas des corps valu\'es et telles que $\sA^{\s}/\sA^{\ss}$ soit un corps \cite[1.7.1]{BGR}.}. 
    
\subsubsection{}\label{ns}  Une $\sK$-alg\`ebre norm\'ee $\sA$ est dite {\it uniforme} si $\sA^{\s}$ est born\'e. Il revient au m\^eme de dire que la norme est \'equivalente \`a sa semi-norme spectrale associ\'ee. Il est clair, en effet, que cette derni\`ere condition entra\^{\i}ne que $\sA^{\s}$ est born\'e, puisque $\sA^{\s}$ est contenue dans la boule unit\'e de la semi-norme spectrale; pour la r\'eciproque, on observe que $\, \vert a\vert \leq \vert a\vert_{sp}\sup_{b\in \sA^{\s}}\,\vert b\vert $.

\smallskip Toute alg\`ebre norm\'ee uniforme $\sA$ est r\'eduite. Si $ \vert \sK\vert$ est dense dans $\R_+$, on a $ \sA^{\ss}=  \sK^{\ss}\sA^{\s}= ( \sA^{\ss})^2 $.

 \subsubsection{Exemples}\label{E1}  
$(1)$ Toute $\sK$-alg\`ebre affino\"{\i}de (classique, \ie quotient de $\sK\langle T_1, \ldots, T_n\rangle$) qui est r\'eduite est uniforme \cite[6.2.1]{BGR}. En fait, dans ce cas, la structure alg\'ebrique d\'etermine la topologie et aussi la norme spectrale \cite[6.1.3, prop. 2]{BGR}. 

\smallskip \noindent $(2)$ Si $\sA$ est uniforme (\resp spectrale), il en est de m\^eme de $\sA\langle T\rangle$: $\sA\langle T\rangle^{\s}= \sA^{\s}\langle T\rangle$ est born\'e (\resp est la boule unit\'e).

\subsubsection{} Un homomorphisme $\phi: \sA\to \sB$ entre $\sK$-alg\`ebres norm\'ees uniformes est continu si et seulement si $\phi(\sA^{\s})\subset \sB^{\s}$ (et dans le cas o\`u les normes sont spectrales, si et seulement si $\vert\vert \phi \vert\vert \leq 1$).
 Il en d\'ecoule que la norme spectrale d'une alg\`ebre norm\'ee uniforme est l'unique norme spectrale compatible avec la topologie.
 
On en d\'eduit qu'une $\sK$-alg\`ebre $\sA$ admet au plus une norme spectrale pour laquelle elle est compl\`ete: en effet, si $\vert\;\vert_1, \vert\; \vert_2$ sont deux telles normes, on peut supposer, en rempla\c cant la premi\`ere par le supremum des deux, que le morphisme identique $(\sA, \vert\;\vert_1)\to (\sA, \vert\;\vert_2)$ est continu. Par le th\'eor\`eme de l'image ouverte de Banach \cite[I.3.3, th. 1]{B1}, $\vert\;\vert_1$ est \'equivalente \`a $  \vert\; \vert_2$, donc ce morphisme est bicontinu et l'on conclut de ce qui pr\'ec\`ede que $\vert\;\vert_1=  \vert\; \vert_2$.

  \smallskip  La compl\'etion d'une alg\`ebre norm\'ee uniforme est une alg\`ebre de Banach uniforme. Il suit du th\'eor\`eme de l'image ouverte qu'une alg\`ebre de Banach est uniforme si et seulement si elle est compl\`ete pour la semi-norme spectrale (qui est alors une norme).   

 \subsubsection{}\label{abu} On note $\,\sK\hbox{-}{\bf uBan}$ (\resp $\,\sA\hbox{-}{\bf uBan}$)  la sous-cat\'egorie pleine de $\,\sK\hbox{-}{\bf Ban}$ (\resp \resp $\,\sA\hbox{-}{\bf Ban}$) form\'ee des alg\`ebres de Banach uniformes. Le plongement de $\,\sK\hbox{-}{\bf uBan}$ dans $\,\sK\hbox{-}{\bf Ban}$
  admet un adjoint \`a gauche ({\it uniformisation}), donn\'e par le compl\'et\'e s\'epar\'e relativement \`a la semi-norme spectrale associ\'ee; par d\'efinition, l'uniformis\'ee $\sA^u$ de $\sA$ est donc spectrale. Le morphisme $\sA \to \sA^u$ (unit\'e d'ajonction) est un isomorphisme si et seulement si $\sA$ est uniforme.
      
     La cat\'egorie $\,\sK\hbox{-}{\bf uBan}$ admet un objet initial $\sK$ et un objet final $0$, des sommes amalgam\'ees $\sB\hat\otimes^u_\sA\, \sC$ (uniformisation du produit tensoriel compl\'et\'e), et donc des colimites finies. En tant qu'adjoint \`a gauche, l'uniformisation pr\'eserve les sommes amalgam\'ees, \ie transforme $\hat\otimes$ en $\hat\otimes^u$.  Si $\sA$ est de Banach uniforme, le foncteur $ - \hat\otimes^u_\sK \sA$ est adjoint \`a gauche du foncteur oubli $\,\sA\hbox{-}{\bf uBan}\to \,\sK\hbox{-}{\bf uBan}$\footnote{Signalons que $\sB\hat\otimes_\sA\, \sC$ peut ne pas \^etre r\'eduite m\^eme si $\sA, \sB, \sC$ sont des corps de car. $0$, \cf \cite{FM}.}. 
     
      Si $\sA\to \sB$ est un morphisme dans $\,\sK\hbox{-}{\bf uBan}$ et $J$ un id\'eal ferm\'e de $\sA\langle T\rangle$,  le morphisme canonique 
   \begin{equation}\label{eu}  (\sB\langle T\rangle / J\sB\langle T\rangle)^u \to \sB\hat\otimes^u_\sA (\sA\langle T\rangle/J)  \end{equation}
est isom\'etrique (on peut munir $\sA$ et $\sB$ de leur norme spectrale, ainsi que les deux membres de la formule \eqref{et}).

     \subsubsection{}\label{i} Tout produit fibr\'e (fini) d'alg\`ebres de Banach uniformes, muni de la norme supremum, est uniforme; donc la cat\'egorie des $\sK$-alg\`ebres de Banach uniformes admet des limites finies.  
          
     \smallskip  Tout facteur d'une alg\`ebre de Banach uniforme est uniforme (en fait, dans une alg\`ebre norm\'ee $\sA$, tout idempotent non nul $e$ est de norme $\geq 1$, et de norme $1$ si $\sA$ est spectrale);
      plus pr\'ecis\'ement, toute d\'ecomposition de $\sA= \sB\times \sC$ en tant qu'alg\`ebre provient d'une d\'ecomposition dans la cat\'egorie des alg\`ebres norm\'ees uniformes.   

      \subsubsection{Remarque}\label{r.2} Si $\sC$ admet une base orthogonale sur $\sA$, et si $\sB$ est spectrale, alors $\sB \stackrel{b\mapsto b\otimes 1}{\to} \sB \hat\otimes^u_\sA\, \sC$ est isom\'etrique. En effet, $\vert b\otimes 1\vert_{sp} = \lim_m\, \vert b^m\otimes 1\vert ^{\frac{1}{m}}= \lim_m\, \vert b^m \vert ^{\frac{1}{m}}= \vert b\vert$ d'apr\`es la rem. \ref{r.1} $2)$.

 \subsubsection{Exemple.}\label{E0}  Soit $\sL$ une extension finie galoisienne de $\sK$ de groupe $G$, munie de l'unique valeur absolue qui prolonge celle de $\sK$ (en particulier, $G$ agit isom\'etriquement). Alors $\sL\otimes_\sK \sL$ munie de la (semi-)norme produit tensoriel est une $\sK$-alg\`ebre de Banach uniforme de dimension finie.  
  L'homomorphisme canonique $\sL\otimes_\sK \sL \to $ {\begin{small}$\displaystyle{\prod_{\gamma \in G}}$\end{small}}$\sL$  est continu et bijectif, donc un isomorphisme d'alg\`ebres de Banach, de sorte que
 {$\sL\hat\otimes^u_\sK \sL \cong $ {\begin{small}$\displaystyle{\prod_{\gamma \in G}}$\end{small}}$\sL$}. 
 L'alg\`ebre de Banach $\sL\otimes_\sK \sL  $ est spectrale si et seulement si  $\sL\otimes_\sK \sL \to $ {\begin{small}$\displaystyle{\prod_{\gamma \in G}}$\end{small}}$\sL$ est isom\'etrique, ce qui \'equivaut \`a dire que l'extension $\sL^{{\s}a}$ de $ \sK^{{\s}a}$ est galoisienne dans le cadre  $(\sK^{\s}, \sK^{{\ss}})$; dans le cas de valuation discr\`ete, ceci a lieu si et seulement si l'extension de corps valu\'es $\sL/\sK$ est non ramifi\'ee. 
 
 Le cas d'une extension infinie, beaucoup plus d\'elicat, est \'etudi\'e en d\'etail dans \cite{FM}.
 
\bigskip{\it Dans toute la suite, on suppose que le corps r\'esiduel $k$ est de caract\'eristique $p> 0$.}

\medskip 
    \subsection{Dictionnaire.}  
    
    \subsubsection{}\label{ferm} Commen\c cons par rappeler quelques op\'erateurs de fermeture. Soit $R\inj S$ une extension d'anneaux commutatifs unitaires. 
    
    Un \'el\'ement $s\in S$ est dit entier (\resp presque entier\footnote{on verra dans le sorite \ref{S2} que cette terminologie, loin d'entrer en conflit avec celle de la presque-alg\`ebre, consonne.})
     sur $R$ si ses puissances engendrent (\resp sont contenues dans) un sous-$R$-module de type fini de $S$.  On note $R^+_S$ (\resp $R^\ast_S$) leur ensemble. C'est un sous-anneau de $S$, la {\it fermeture int\'egrale} (\resp {\it compl\`ement int\'egrale}) de $R$ dans $S$. On dit que $R$ est int\'egralement ferm\'e (\resp compl\`etement int\'egralement ferm\'e) dans $S$ si $R= R^+_S$ (\resp $ R= R^\ast_S$) (\cf \eg \cite{GH}) . 
   Alors que $R^+_S$ est int\'egralement ferm\'e dans $S$, $R^\ast_S$ n'est pas n\'ecessairement compl\`etement int\'egralement ferm\'e dans $S$ (cf. rem. \ref{r2} $(2)$ ci-dessous). On a bien s\^ur $R^+_S \subset R^\ast_S$ (avec \'egalit\'e si $R$ est noeth\'erien). 
    
    \smallskip Soit $p$ un nombre premier. Un \'el\'ement $s\in S$ est $p$-radiciel sur $R$ s'il existe $n\in \N$ tel que $s^{p^n}\in R$. On dit que $R$ est $p$-radiciellement ferm\'e dans $S$ si tout \'el\'ement $p$-radiciel est dans $R$; autrement dit, si $(s\in S, \, s^p\in R) \Rightarrow s\in R$. La {\it fermeture $p$-radicielle} de $R$ dans $S$, not\'ee $R^\dagger_S$, est le plus petit anneau interm\'ediaire $p$-radiciellement ferm\'e dans $S$. C'est la r\'eunion croissante des sous-anneaux $R_i$ d\'efinis inductivement par $R_0= R, \, R_{i+1} =$ sous-anneau de $S$ engendr\'e par les \'el\'ements $p$-radiciels de $R_i$ (\cf \cite{ADR}). On a bien s\^ur $R^\dagger_S \subset R^+_S$, et $R$ est $p$-radiciellement ferm\'e dans $S$ si et seulement si $R= R^\dagger_S$.
    
    Lorsque $R$ est de car. $p$, ou bien lorsque $S= R[\frac{1}{p}]$, l'ensemble des \'el\'ements $p$-radiciels de $S$ sur $R$ est un sous-anneau de $S$ \cite{Ro}, qui co\"{\i}ncide donc avec $R^\dagger_S$.  
    
   \smallskip   \'Etant donn\'e un carr\'e commutatif     \[  \xymatrix @-1pc { R  \ar[d]    \,  \ar@{^{(}->}[r] & S  \ar[d]      \\   R'    \,  \ar@{^{(}->}[r] & S'    }  \] 
 on a des homomorphismes naturels $R^\dagger_S \to R'^\dagger_{S'},\, R^+_S \to R'^+_{S'},\, R^\ast_S \to R'^\ast_{S'}$ (pour le premier, on construit pas \`a pas $R_i \to R'_i$).
    
   La fermeture int\'egrale de $R$ dans $S$ commute \`a la localisation \cite[V, \S 1, n. 5, prop. 16]{B2}, et il en est de m\^eme de la fermeture $p$-radicielle comme on le v\'erifie imm\'ediatement. Ce n'est pas vrai en revanche pour la fermeture compl\`etement int\'egrale \cite[V, \S 1, ex. 12]{B2}. 
    
    \subsubsection{} Voici un fragment de dictionnaire entre le langage de l'analyse fonctionnelle et celui de l'alg\`ebre commutative.

  Fixons $\varpi\in {\sK}^{\ss}\setminus 0$, et notons $\Gamma \subset  \R$ le groupe de valuation correspondant: $\vert \varpi \vert^\Gamma = \vert \sK^\times \vert$. Pour tout $s\in \Gamma$, soit $\varpi_s$ un \'el\'ement de $\sK^\times$ de norme $\vert \varpi \vert^s$ (les consid\'erations qui suivent ne d\'ependent pas de tels choix). 
   Un compl\'et\'e-s\'epar\'e ($\varpi$-adique ou selon une semi-norme) sera affubl\'e d'un chapeau.
  
   \begin{sorite}\label{s1}  \begin{enumerate}  

 \item Soit $\sA$ une $\sK$-alg\`ebre norm\'ee telle que $\vert \sK\vert$ soit dense dans $ \vert \sA\vert$. Pour tout $a\in \sA$, on a \begin{equation}\;\vert a\vert = \vert \varpi\vert^r,\,\;{\rm{\it avec}}\;\;  r:= \sup \{s\in \Gamma \mid \varpi_{-s} a \in \sA_{\leq 1}\}.\end{equation}
 En particulier, la topologie de $\sA_{\leq 1}$ est la topologie $\varpi$-adique, et $\sA$ est de Banach si et seulement si $\sA_{\leq 1}$ est $\varpi$-adiquement compl\`ete.

 \item   R\'eciproquement, soit $\,{\mathfrak{A}}\,$ une ${\sK}^{\s}$-alg\`ebre plate (\ie sans $\varpi$-torsion), et posons 
\[  \sA := {\mathfrak{A}}[\frac{1}{\varpi}] .\]
  Pour tout $a\in \sA  $, posons 
\begin{equation}\,{}^{{}^{{\mathfrak{A}}}}\vert a\vert  := \vert \varpi\vert^r,\,\;{\rm{\it avec}}\;\;  r:= \sup \{s\in \Gamma \mid \varpi_{-s} a \in {\mathfrak{A}} \} \in \R \cup \{+\infty\}.\end{equation}
   \begin{enumerate}
 \item Cela d\'efinit une semi-norme sur $\sA$, qui est une norme si et seulement si aucun \'el\'ement de ${\mathfrak{A}}\setminus 0$ n'est infiniment $\varpi$-divisible; $\vert \sK\vert$ est dense dans $ {}^{{}^{{\mathfrak{A}}}}\vert \sA\vert $.    
 
 En outre $\hat{\mathfrak A}$ est aussi une $\sK^{\s}$-alg\`ebre plate,  $\,\hat{\sA}  =  \hat{\mathfrak A}[\frac{1}{\varpi}],$ et  $\hat\sA_{\leq 1} =\widehat{\sA_{\leq 1}}$.  
 
\item 
 $\;\;\;\;\;\;\;\; \displaystyle{\sA_{< 1}= \sK^{\ss}{\mathfrak{A}}}$,

 $\;\;\;\;\;\;\;\; \displaystyle{\sA_{\leq 1}= {\mathfrak{A}}}\;\;$ si $\Gamma$ est discret,

$\;\;\;\;\;\;\;\; \displaystyle{\sA_{\leq 1}=  {\mathfrak{A}}_\ast}\;$ en g\'en\'eral\footnote{Notation emprunt\'ee \`a la presque-alg\`ebre lorsque $\Gamma$ est dense dans $\R$, auquel cas $\sK^{\ss}$ est un id\'eal idempotent de $\sK^{\s}$, \cf \S \ref{pa}; sinon, $\mathfrak{A}_\ast= \mathfrak{A}$.}, avec 
 $\;{\mathfrak{A}}_\ast :=   {\rm{Hom}}_{\sK^{\s}}(\sK^{\ss}, {\mathfrak{A}}) = \displaystyle\bigcap_{s\in \Gamma_{>0}}\,\varpi_{-s} {\mathfrak{A}}.$ 
  
\item Un homomorphisme $\phi: \sA= {\mathfrak{A}}[\frac{1}{\varpi}] \to \sA'= {\mathfrak{A}}'[\frac{1}{\varpi}]$ est continu si $\phi({\mathfrak{A}})\subset {\mathfrak{A}}'$, et est isom\'etrique si en outre $\phi^{-1}{\mathfrak{A}}'= {\mathfrak{A}}$. Plus pr\'ecis\'ement, un homomorphisme continu $\phi$ est isom\'etrique si et seulement si $\phi(\sA_{\leq 1})\subset \sA'_{\leq 1}$ et l'homomorphisme induit $\sA_{\leq 1}/\varpi \to \sA'_{\leq 1}/\varpi$ est injectif.

\item Si $  {\mathfrak{A}}\subset {\mathfrak{A}}'$ est une sous-alg\`ebre telle que le quotient ${\mathfrak A}'/\mathfrak A$ soit sans $\varpi$-torsion,  la restriction de $ \,{}^{{}^{{\mathfrak{A}'}}}\vert  \;\vert $ \`a $\mathfrak A$ est $\,{}^{{}^{{\mathfrak{A}}}}\vert \;\vert $.  

Si $\mathfrak I$ est un id\'eal de $\mathfrak A$ tel que ${\mathfrak A}/\mathfrak I$ soit sans $\varpi$-torsion, $ \,{}^{{}^{{\mathfrak{A}}/\mathfrak I}}\vert  \;\vert $ est la semi-norme quotient de $\,{}^{{}^{{\mathfrak{A}}}}\vert \;\vert $. R\'eciproquement, pour tout id\'eal $\sI$ de $\sA$, alors $(\sA/\sI)_{\leq 1}  = (\sA_{\leq 1}/\sI_{\leq 1})_\ast$.  
 
\item Soient $\,{\mathfrak{B}},{\mathfrak{C}} \,$ deux ${\mathfrak{C}}$-alg\`ebres sans $\varpi$-torsion. Alors la semi-norme associ\'ee \`a $(\mathfrak B\otimes_{\mathfrak A}\mathfrak C)/( \varpi^\infty$-torsion$)$ est la semi-norme produit tensoriel (o\`u $ \varpi^\infty$-torsion d\'esigne la torsion $\varpi$-primaire). R\'eciproquement, si $\sB, \sC$ sont deux $\sA$-alg\`ebres norm\'ees, on a 

 $\;\;\;\;\;\;\;\; \displaystyle  {(\sB \otimes_\sA\, \sC)}_{\leq 1}  =  (\sB_{\leq 1}\otimes_{\sA_{\leq 1}} \sC_{\leq 1})/( \varpi^\infty$-tors$)\,$ si $\Gamma$ est discret,

$\;\;\;\;\;\;\;\; \displaystyle {(\sB \otimes_\sA\, \sC)}_{\leq 1}  =  ((\sB_{\leq 1}\otimes_{\sA_{\leq 1}} \sC_{\leq 1})/( \varpi^\infty$-tors$))_\ast  \;$ sinon.

De m\^eme pour les produits tensoriels compl\'et\'es.  \end{enumerate}

\item Le passage au quotient modulo $\sA^{\ss}$ (\resp \`a la compl\'etion) induit une bijection entre les sous-$\sK^{\s}$-alg\`ebres $\mathfrak B$ ouvertes [compl\`etement] int\'egralement ferm\'ees dans $\sA^{\s}$ et les sous-$k$-alg\`ebres [compl\`etement] int\'egralement ferm\'ees de la $k$-alg\`ebre r\'eduite $\sA^{\s}/\sA^{\ss}$ (\resp les sous-$\sK^{\s}$-alg\`ebres ouvertes int\'egralement ferm\'ees dans $\hat\sA^{\s}$).
  
\newpage \medskip  Supposons  $\vert \sK \vert$ dense dans $\vert \sA \vert\,$ (ce qui est le cas si $\Gamma$ est dense dans $\R$). 
 
 \item Pour $\mathfrak A$  comme ci-dessus et $\sA = \mathfrak A[\frac{1}{\varpi}]$, et avec les notations de \ref{ferm}:
 \[ \mathfrak A_\sA^\dagger := \{a\in \sA\mid \exists n, \, a^{p^n}\in \mathfrak A\}, \;\;  \mathfrak A_\sA^+ , \;\; \mathfrak A_\sA^\ast :=   \{a\in \sA\mid  \exists m \,\forall n ,\, \varpi^m a^n \in \mathfrak A\},\]
on a 
\begin{equation}\label{1eq}  \mathfrak A_\sA^\dagger \subset \mathfrak A_\sA^+ \subset \mathfrak A_\sA^\ast =   ( \mathfrak A_\ast)_\sA^\ast = \sA^{\s}\end{equation}  et 
 $   (\mathfrak A_\sA^\dagger)_\ast =  (\sA^{\s})_\ast $ est la boule unit\'e de la semi-norme spectrale associ\'ee \`a $\vert\,\vert$.  
 
  En particulier, si $\sA_{\leq 1}$ est noeth\'erienne, sa fermeture int\'egrale dans $\sA$ est \'egale \`a $\sA^{\s}$.

\item  Les conditions suivantes sont \'equivalentes: \begin{enumerate}

  \item  $\vert\;\vert$ est spectrale, 
  
\item   $ {\sA_{\leq 1}}=\sA^{\s}$,
  
  \item $ {\sA_{\leq 1}}= ({\sA_{\leq 1}})_\sA^\ast \;\;$  (${\sA_{\leq 1}}$ est compl\`etement int\'egralement ferm\'e dans $\sA$),
  
\item $ {\sA_{\leq 1}}= ({\sA_{\leq 1}})_\sA^+ \;\;$    ($ {\sA_{\leq 1}}$ est int\'egralement ferm\'e dans $\sA$),

\item  $ {\sA_{\leq 1}}= ({\sA_{\leq 1}})_\sA^\dagger \;\;$  (${\sA_{\leq 1}}$ est $p$-radiciellement ferm\'e dans $\sA$), 

  \item (si $\vert \varpi \vert \geq \vert p\vert$ et $\frac{1}{p}\in \Gamma $) l'homomorphisme ${\sA_{\leq 1}}/\varpi_{\frac{1}{p}} \stackrel{x\mapsto x^p}{\to}\, {\sA_{\leq 1}}/\varpi  \,$ est {injectif}.
 \end{enumerate}  
\item Si $\sA$ est uniforme, on a 
\begin{equation}\label{eq1}     \sA^{\s} = (\sA^{\s})_\ast 
= (\sA_{\leq 1})^\ast_\sA. \end{equation} 
 et 
 \begin{equation}\label{eq1'}     \widehat{\sA^{\s}} = (\hat\sA)^{\s}. \end{equation} 
   \end{enumerate}
  \end{sorite}
    
    \begin{proof} $(1)$ L'\'egalit\'e $\;\vert a\vert = \vert \varpi\vert^r$ est claire si $\vert a\vert \in \vert\sK^\times\vert$; elle vaut donc pour tout $a\in \sA$ compte tenu de la densit\'e de $\vert\sK\vert$ dans $ \vert \sA\vert$.
        
 \smallskip \noindent $(2)$ $(a)$ Que ${}^{{}^{B}}\vert\;\vert$ soit une semi-norme non-archim\'edienne d\'ecoule directement de ce que ${\mathfrak{A}}$ est une $\sK^{\s}$-alg\`ebre. C'est une norme si ${\mathfrak{A}}$ est $\varpi$-adiquement s\'epar\'ee. 
 
 Par platitude de $\mathfrak A$,  $  \mathfrak A/\varpi^{n-1}\stackrel{\cdot \varpi}{\to} \mathfrak A/\varpi^n$ est injectif pour tout $n$, d'o\`u, en passant \`a la limite, le fait que 
 $\hat{\mathfrak A}$ soit sans $\varpi$-torsion. 
 
 Les \'egalit\'es $\,\hat{\sA}  =   \hat{\mathfrak A}[\frac{1}{\varpi}],\;  {\hat\sA}_{\leq 1} =\widehat{\sA_{\leq 1}}$, de m\^eme que les points   $(a)$ $(b)$ et la premi\`ere assertion de $(c)$ r\'esultent de la d\'efinition de la semi-norme (en notant que  $\sup \{  s\in \Gamma_{>0}\} > 0$ si $\Gamma$ est discret, $=0$ sinon). On a  ${}^{{}^{{\mathfrak{A}}}}\vert\;\vert =  {}^{{}^{\sA_{\leq 1}}}\vert\;\vert$. 
 
 Compl\'etons la preuve de $(c)$. Puisque $\vert \sK\vert $ est dense dans $\sA$, $\phi$ est isom\'etrique si et seulement si $\sA_{\leq 1} \stackrel{\phi}{\to} \sA'_{\leq 1} \cap \phi(\sA)  $ est un isomorphisme. Comme l'injectivit\'e de $\sA_{\leq 1}/\varpi\to \sA'_{\leq 1}/\varpi$ implique aussi celle de $\phi$, on peut supposer $\phi$ injectif (et le noter comme une inclusion pour all\'eger). Il suffit donc de s'assurer que $\sA'_{\leq 1}\cap \frac{1}{\varpi} \sA_{\leq 1}  = \sA_{\leq 1} \Rightarrow   \sA'_{\leq 1}\cap \sA =  \sA_{\leq 1}   $ (la r\'eciproque \'etant banale). Or pour tout $n\geq 1$,  $\sA'_{\leq 1}\cap \frac{1}{\varpi} \sA = \sA_{\leq 1} \Rightarrow  \varpi^n\sA'_{\leq 1}\cap \frac{1}{\varpi} \sA_{\leq 1}  = \varpi^n\sA'_{\leq 1}\cap   \sA_{\leq 1} $, d'o\`u aussi par r\'ecurrence   $\sA'_{\leq 1}\cap \frac{1}{\varpi^n} \sA = \sA_{\leq 1} $.
 
 La premi\`ere assertion de $(d)$ r\'esulte de $(c)$ puisqu'on a  ${\mathfrak A}' \cap {\mathfrak A}[\frac{1}{\varpi}]= \mathfrak A $. Pour la seconde, en notant $\bar a $ l'image d'un \'el\'ement $a\in \sA$ modulo ${\mathfrak I}[\frac{1}{\varpi}]$, on observe que  ${\mathfrak A}  \cap {\mathfrak I}[\frac{1}{\varpi}]= \mathfrak I $   de sorte que $  \sup \{s\in \Gamma \mid \varpi_{-s}  \bar a \in  \mathfrak{A}/\mathfrak I  \}  = \sup_{i\in {\mathfrak I}[\frac{1}{\varpi}]} \sup \{s\in \Gamma \mid \varpi_{-s}  a + i  \in  \mathfrak{A}  \}  $. Pour l'assertion r\'eciproque, on note que $\sA_{\leq 1}/ \sI_{\sI}$ est sans $\varpi$-torsion et on applique le point $(b)$.
 
Pour $(e)$, consid\'erons une somme $\sum b_n  \otimes c_n$ avec $\vert b_n \vert\vert c_n\vert \leq 1$. Dans le cas discret (\resp non discret), on la r\'e\'ecrit comme $\sum \varpi_{r_n} b_n  \otimes \varpi^{-1}_{r_n} c_n$ avec ${r_n}$ \'egal \`a (\resp inf\'erieur \`a, mais arbitrairement proche de) la valuation de $y_n$.

       \smallskip\noindent     $(3)$ Le cas de la compl\'etion se d\'eduit du cas de la r\'eduction modulo $\sA^{\ss}$. Celle-ci induit une bijection entre sous $\sK^{\s}$-alg\`ebres ouvertes $\mathfrak B$ de $\sA^{\s}$ et sous-$k$-alg\`ebres $\bar{\mathfrak B}$ de $\sA^{\s}/\sA^{\ss}$. Le point est de montrer que $\mathfrak B$ et $\bar{\mathfrak B}$ sont simultan\'ement int\'egralement ferm\'ees. Cela d\'ecoule imm\'ediatement de l'observation suivante:
            
      \smallskip      {\it  soient $S$ un anneau, $I$ un id\'eal, $R$ un sous-anneau de $S$ contenant $I$, $s$ un \'el\'ement de $S$ et $\bar s$ son image dans $S/I$.  Alors $s$ est entier sur $R$ si et seulement si $\bar s$ est entier sur $R/I$ .}
            
        \smallskip    En effet, la r\'eduction modulo $I$ induit une bijection entre sous-$R$-modules $M$ de type fini de $S$ contenant $R$ et sous-$R/I$-modules $\bar M$ de type fini de $S/I$ contenant $R/I$.  Que $s$ soit  entier sur $R$ signifie que ses puissances  engendrent un tel module $M,$
              et de m\^eme modulo $I$ (le m\^eme argument vaut en rempla\c cant {``int\'egralement ferm\'e"} par {``compl\`etement int\'egralement ferm\'e"}).   
              
                \smallskip \noindent     $(4)$ Rappelons que $\mathfrak A_\sA^\ast $ est  l'ensemble des \'el\'ements de $\sA$ dont les puissances sont contenues dans un sous-$\mathfrak A$-module de type fini. Puisque \^etre contenu dans un sous-$\mathfrak A$-module de type fini de $\sA$, c'est \^etre born\'e, on a $\mathfrak A_\sA^\ast   = \sA^{\s}$.

  La semi-norme $\vert \; \vert^\dagger $  associ\'ee \`a $\mathfrak A_\sA^\dagger$ est donn\'ee par 
   $\;\vert a\vert^\dagger   = \vert \varpi\vert^{r^\dagger}\,\;{\rm{avec}}\;\;  r^\dagger := \sup \{s\in \Gamma  \mid \exists n, \, \varpi_{-sp^n} a^{p^n} \in  \sA_{\leq 1} \}$. D'o\`u 
   $ \;\vert a\vert^\dagger  = {\rm{inf}}\, \vert a^{p^n}\vert^{\frac{1}{p^n}} $, et par le lemme de Fekete, 
   $ \;\vert a\vert^\dagger  = {\rm{lim}}\, \vert a^{p^n}\vert^{\frac{1}{p^n}}   = \vert a\vert_{sp}$.
   
  La formule $ \;\vert a\vert^\dagger  = {\rm{inf}}\, \vert a^{p^n}\vert^{\frac{1}{p^n}} $ montre aussi que $(\mathfrak A_\sA^\dagger)_\ast$ contient la boule unit\'e de $\sA$ pour la semi-norme {spectrale} $\vert\,\vert^\dagger$. Par d\'efinition de $\sA^{\s}$, et compte tenu de $(2b)$, on a donc $\sA^{\s}  \subset (\mathfrak A_\sA^\dagger)_\ast$, et en derni\`ere instance  $\sA^{\s}_\ast  = (\mathfrak A_\sA^\dagger)_\ast$.

          \smallskip\noindent     $(5)$  Les \'equivalences $(a)  \Leftrightarrow (b)    \Leftrightarrow (c) \Leftrightarrow (d) \Leftrightarrow (e)  $ r\'esultent de $(2b)$ et $(4)$,         
            et si $\vert \varpi \vert \geq \vert p\vert$ et $\frac{1}{p}\in \Gamma $, on a  $(e)\Leftrightarrow (f)$. 
                    
                    \smallskip\noindent     $(6)$ On peut supposer $\vert\; \vert$ spectrale. Alors $(2b)$ et $(5b)$ entra\^{\i}nent que   $(\sA^{\s})_\ast = \sA^{\s}$. L'autre \'egalit\'e suit de \eqref{1eq}, de m\^eme que  $\widehat{\sA^{\s}} = (\hat\sA)^{\s}$ via le point $(2a)$ (avec $\vert\;\vert_{sp}$).
                       \end{proof}

  \subsubsection{Remarques}\label{r2}  $(1)$ $\sA^{\s}/\varpi_{ {\frac{1}{p}}}\stackrel{x\mapsto x^p}{\to}\, \sA^{\s}/\varpi \,$ est toujours injectif 
   m\^eme si $\sA$ n'est pas uniforme.

   \smallskip\noindent $(2)$  Au point $(3)$, il se peut que les inclusions $\mathfrak A_\sA^+ \subset \mathfrak A_\sA^\ast = \sA^{\s}\subset  \sA_{\vert\,\vert_{sp} \leq 1}$ soient strictes, comme le montre l'exemple $\mathfrak A = \sK^{\s}[\varpi T, \ldots, \varpi^{i+1} T^{p^i}, \ldots] \subset  \sK^{\s}[T]$: on a $\vert T^m \vert = \vert \varpi\vert ^{-1-[\log_p m]},\;  \sA = \sK[T],\; \sA^{\s} = \mathfrak A_\sA^\ast = \sK^{\s} + \sK^{\ss}T\sK^{\s}[T]\, $ et $ \, \sA_\sA^{\s\ast}    =  \mathfrak A_\sA^{\ast\ast} =  \sK^{\s}[T]$, tandis que $  \sA_\sA^{\sm o+}=   \sA^{\s}$. 
     On voit l\`a un exemple de fermeture compl\`etement int\'egrale $\mathfrak A_\sA^\ast$ qui n'est pas compl\`etement int\'egralement ferm\'ee, et un exemple o\`u $\sA^{\s}$ n'est pas la boule unit\'e de la norme spectrale.
            
   \subsubsection{} Tirons-en les premiers fruits, en supposant $\vert\sK \vert$ dense dans $\R_+$.
    
\begin{cor}\label{C1}  Le foncteur $\sB \mapsto \sB^{\s}$ induit une \'equivalence de la cat\'egorie des $\sK$-alg\`ebres de Banach uniformes  vers celle des $\sK^{\s}$-alg\`ebres ${\mathfrak{B}}$ v\'erifiant les  propri\'et\'es suivantes: 
\begin{enumerate} \item $\mathfrak B$ est plate sur $\sK^{\s}$,
\item $\mathfrak B$ est $\varpi$-adiquement compl\`ete,
\item $\mathfrak B$ est  [$p$-radiciellement/int\'egralement/compl\`etement int\'egralement] ferm\'ee dans ${\mathfrak{B}}[\frac{1}{\varpi}]$, 
\item ${\mathfrak{B}}= {\rm{Hom}}_{\sK^{\s}}(\sK^{\ss}, \mathfrak B)$.\end{enumerate} 
Un quasi-inverse est donn\'e par ${\mathfrak{B}}\mapsto ({\mathfrak{B}}[\frac{1}{\varpi}], \,{}^{{}^{{\mathfrak{B}}}}\vert \;\vert )$ (\`a valeurs dans les alg\`ebres spectrales).
Idem pour les $\sA$-alg\`ebres de Banach uniformes (\resp $\sA^{\s}$-alg\`ebres v\'erifiant $(1)$ \`a $(4)$). \qed\end{cor}

Sous $(4)$, les trois interpr\'etations de $(3)$ \'equivalent \`a: ${}^{\mathfrak B}\vert\;\vert$ est spectrale (points $(2)(b)$ et $(5)$ du sorite).

\begin{cor}\label{C1'}  Soit $\sB$ une $\sK$-alg\`ebre de Banach uniforme. Il y a bijection entre les idempotents de $\sB$, ceux de $\sB^{\s}$, ceux de $\sB^{\s}/\varpi$, et ceux de $\sB^{\s}/\sB^{\ss}$.
\end{cor} 
 
On peut supposer $\sB$ spectrale. Pour la premi\`ere bijection, voir \ref{i}. Pour la seconde, le point est que $\sB^{\s}$ est $\varpi$-adiquement compl\`ete; pour la troisi\`eme, que le noyau de $\sB^{\s}/\varpi\to \sB^{\s}/\sB^{\ss}$ est un nil-id\'eal. \qed

 \begin{lemma}\label{L4} Soient $\sA$ une $\sK$-alg\`ebre de Banach uniforme et $\mathfrak m$ un id\'eal idempotent de $  \sA^{\s}$ tel que $\sK^{\ss}\mathfrak m= \mathfrak m$.   Consid\'erons l'endofoncteur des $\sA^{\s}$-alg\`ebres donn\'e par les presque-\'el\'ements du localis\'e, dans le cadre $(\sA^{\s}, \mathfrak m)$\footnote{On passe du cadre initial  $(\sK^{\s}, \sK^{\ss})$ \`a $(\sA^{\s}, \mathfrak m)$ via $(\sA^{\s}, \sK^{\ss}\sA^{\s}= \sA^{\ss})$, le sens de {``presque"} \'etant le m\^eme dans $(\sK^{\s}, \sK^{\ss})$ et dans $(\sA^{\s}, \sK^{\ss}\sA^{\s} )$ (\cf \S \ref{rec}). }: 
\begin{equation}\mathfrak B \mapsto (\mathfrak B^a)_\ast =  {\rm{Hom}}_{\sA^{\s a}}(\sA^{\s a}, \mathfrak B^a) = {\rm{Hom}}_{\sA^{\s}}(\tilde{\mathfrak m},\mathfrak B) .\end{equation} 
 Ce foncteur respecte chacune des conditions $(1)$ \`a $(4)$ du corollaire pr\'ec\'edent.
 En particulier, le foncteur des presque-\'el\'ements est adjoint \`a droite de $(\,)^a$ sur la cat\'egorie des $\sA^{\s}$-alg\`ebres v\'erifiant $(1)$ \`a $(4)$, donc commute aux limites. 
\end{lemma} 

\begin{proof} $(1)$: si $\mathfrak B$ est sans $\varpi$-torsion, il en est de m\^eme de ${\rm{Hom}}_{\sA^{\s}}(\tilde{\mathfrak m}, \mathfrak B)$.
  
   \noindent    $(2)$ est plus d\'elicat car le foncteur $\mathfrak B \mapsto (\mathfrak B^a)_\ast $, qui commute aux limites, ne commute pas au quotient par $\varpi^n$ en g\'en\'eral. On a toutefois un morphisme canonique $u: \widehat{(\mathfrak B^a)_\ast} \to (\hat{\mathfrak B}^a)_\ast  $ obtenu en prenant la limite des morphismes $ {\rm{Hom}}_{\sA^{\s}}(\tilde{\mathfrak m},\mathfrak B)/ \varpi^n \to   {\rm{Hom}}_{\sA^{\s}}(\tilde{\mathfrak m},\mathfrak B/ \varpi^n)$. Ces derniers sont injectifs (le noyau de $ {\rm{Hom}}_{\sA^{\s}}(\tilde{\mathfrak m},\mathfrak B)  \to   {\rm{Hom}}_{\sA^{\s}}(\tilde{\mathfrak m},\mathfrak B/ \varpi^n)$ est l'image de 
 $  {\rm{Hom}}_{\sA^{\s}}(\tilde{\mathfrak m}, \varpi^n \sA^{\s} \otimes_{\sA^{\s}}\mathfrak B) = \varpi^n {\rm{Hom}}_{\sA^{\s}}(\tilde{\mathfrak m},\mathfrak B)  \to   {\rm{Hom}}_{\sA^{\s}}(\tilde{\mathfrak m},\mathfrak B )$), donc $u$ est injectif. Du triangle commutatif 
       \[  \xymatrix @-1pc { {\rm{Hom}}_{\sA^{\s}}(\tilde{\mathfrak m},\mathfrak B)  \ar[d]    \,  \ar@{^{}->}[r] & \lim {\rm{Hom}}_{\sA^{\s}}(\tilde{\mathfrak m},\mathfrak B)/ \varpi^n  \ar[d]^{u}      \\    {\rm{Hom}}_{\sA^{\s}}(\tilde{\mathfrak m}, \lim \mathfrak B/ \varpi^n)    \, = & \lim  {\rm{Hom}}_{\sA^{\s}}(\tilde{\mathfrak m},\mathfrak B/ \varpi^n)   }  \] on d\'eduit que 
 si $\mathfrak B$ est $\varpi$-adiquement compl\`ete, $u$ est inverse \`a gauche du  morphisme de compl\'etion $(\mathfrak B^a)_\ast  \mapsto \widehat{(\mathfrak B^a)_\ast}$. C'est donc un isomorphisme.
      
    \noindent    $(3)$:  soit $f\in {\rm{Hom}}_{\sA^{\s}}(\tilde{\mathfrak m}, \mathfrak B)[\frac{1}{\varpi}]$ tel que $\varpi f, f^p\in  {\rm{Hom}}_{\sA^{\s}}(\tilde{\mathfrak m}, \mathfrak B) $. Alors pour tout $\eta \in \tilde{\mathfrak m}$, on a $\varpi f(\eta), f(\eta)^p\in \mathfrak B$. 
    Par hypoth\`ese ceci implique $f(\eta)\in \mathfrak B$. On conclut que $f\in {\rm{Hom}}_{\sA^{\s}}(\tilde{\mathfrak m}, \mathfrak B) $.
      
  \noindent    $(4)$:  ${\rm{Hom}}_{\sK^{\s}}(\sK^{\ss}, {\rm{Hom}}_{\sA^{\s}}(\tilde{\mathfrak m}, \mathfrak B)) = {\rm{Hom}}_{\sA^{\s}}(\sK^{\ss}\otimes_{\sK^{\s}} \sA^{\s}, {\rm{Hom}}_{\sA^{\s}}(\tilde{\mathfrak m}, \mathfrak B)) =
   {\rm{Hom}}_{\sA^{\s}}(\sK^{\ss}\otimes_{\sK^{\s}} \tilde{\mathfrak m},  \mathfrak B)  =  {\rm{Hom}}_{\sA^{\s}}(\tilde{\mathfrak m}, \mathfrak B)$.
   \end{proof}

\subsubsection{Remarques}\label{234} $(1)$  En revanche, l'adjoint \`a gauche $(\;)_{!!}$ de $(\,)^a$ ne respecte pas les conditions $(2)$ \`a $(4)$. Mais on peut le modifier en un adjoint \`a gauche $(\;)_{\hat{!!}}$ de $(\; )^a$ sur la cat\'egorie des $\sA^{\s}$-alg\`ebres $\mathfrak B$ v\'erifiant $(1)$ \`a $(4)$. Il suffit de 

 $a)$ quotienter $\mathfrak B_{!!}$ par la torsion $\varpi$-primaire (pour assurer $(1)$),  
 
 $b)$ compl\'eter pour assurer $(2)$, 
 
 $c)$ passer \`a la fermeture [$p$-radicielle/int\'egrale] dans le $\frac{1}{\varpi}$-localis\'e (ce qui assure $(3)$), 
 
 $d)$ prendre les presque-\'el\'ements dans le cadre $(\sK^{\s}, \sK^{{\ss}})$ (ce qui assure $(4)$),
 
\noindent compte tenu de ce que chacune de ces op\'erations est un r\'eflecteur (d'apr\`es le sorite \ref{s1} $(4)$, elles ne se d\'etruisent pas l'une l'autre - les deux derni\`eres reviennent \`a passer \`a l'anneau des \'el\'ements de puissances born\'ees dans le $\frac{1}{\varpi}$-localis\'e).

\smallskip\noindent $(2)$  De m\^eme, le coproduit de la cat\'egorie des $\sA^{\s}$-alg\`ebres v\'erifiant $(1)$ \`a $(4)$ n'est pas $\otimes$ (qui ne respecte pas $(2)$ ni $(3)$ ni $(4)$) mais le bifoncteur donn\'e par le compl\'et\'e $\mathfrak B \hat\otimes^u_{\mathfrak A} \mathfrak C$ de la fermeture compl\`etement int\'egrale de  $(\mathfrak B  \otimes_{\mathfrak A} \mathfrak C )/( \varpi^\infty$-tors$)\,$  dans $(\mathfrak B  \otimes_{\mathfrak A} \mathfrak C)[\frac{1}{\varpi}]$.

\subsubsection{}\label{recc}  L'extension de la presque-alg\`ebre au contexte topologique pose probl\`eme:
 la cat\'egorie des modules topologiques et  homomorphismes continus n'est pas ab\'elienne. Pour les alg\`ebres de Banach uniformes, on peut n\'eanmoins en esquisser un avatar \`a partir des r\'esultats pr\'ec\'edents:  la  cat\'egorie $\sA^{\hat a}\hbox{-}{\bf uBan}$ des {\it $\sA^{\hat a}$-alg\`ebres de Banach uniformes} a pour objets les $\sK$-alg\`ebres de Banach uniformes, not\'ees $\sB^{\hat a}$, et les morphismes $\sB^{\hat a}\to \sC^{\hat a}$ sont les morphismes $\sB^{\s a} \to \sC^{\s a}\,$ (dans le cadre $( \sA^{\s}, \mathfrak m)$). 

Si le cadre est $(\sK^{\s}, \sK^{{\ss}})$, cette cat\'egorie est bien entendu isomorphe \`a celle des $\sK$-alg\`ebres de Banach uniformes.
 
  On a un foncteur \'evident $(\;)^{\hat a}:  \sA \hbox{-}{\bf uBan}\to \sA^{\hat a}\hbox{-}{\bf uBan}$  qui est l'identit\'e sur les objets. D'apr\`es le lemme pr\'ec\'edent,  $(\;)^{\hat a}$ admet un adjoint \`a droite $(\;)_{\hat\ast}$ dont l'unit\'e d'adjonction est 
    \begin{equation}\label{e3}  \sB \to (\sB^{\hat a})_{\hat\ast} =  (\sB^{{\s} a})_\ast[\frac{1}{\varpi}] = {\rm{Hom}}_{\sA^{\s}}(\tilde{\mathfrak m},\sB^{\s})[\frac{1}{\varpi}]  = {\rm{Hom}}_\sA^{cont}(\tilde{\mathfrak m}\otimes\sA, \sB). \end{equation} 
   Ce morphisme est injectif si $\mathfrak m$ agit fid\`element sur $\sB$. 
  On a $((\sB^{\hat a})_{\hat\ast} )^{\s} =  {\rm{Hom}}_{\sA^{\s}}(\tilde{\mathfrak m}, \sB^{\s})$. 
  
  D'autre part, $(\;)^{\hat a}$ admet un adoint \`a gauche encore not\'e $(\;)_{\hat{!!}}$.
   
  \smallskip  Comme on le verra ci-dessous (\S \ref{E4'}), l'inclusion dans l'alg\`ebre des presque-\'el\'ements 
    \[(\sB^{\hat a})_{\hat\ast} \subset  (\sB^a)_\ast :=  {\rm{Hom}}_{\sA^{\s} }(\tilde{\mathfrak m} , \sB ) =  {\rm{Hom}}_{\sA }(\tilde{\mathfrak m}\otimes \sA, \sB )\]
     n'est pas une \'egalit\'e en g\'en\'eral.  On a toutefois:
    
    \begin{lemma}\label{L5} 
    Soit $\sB \stackrel{\phi}{\to} \sC$ un morphisme de $\sA$-alg\`ebres de Banach uniformes, et soit $\sB^{\s} \stackrel{\phi^{\s}}{\to} \sC^{\s}$ sa restriction. 
     Alors, dans le cadre $( \sA^{\s}, \mathfrak m)$,  $\phi$ est un presque-isomorphisme si et seulement si $\phi^{\s}$ l'est (c'est-\`a-dire si $\phi^{\hat a}$ est un isomorphisme).      \end{lemma}  
    
        \begin{proof} Il est clair que si $(\phi^{\s})^a$ est un isomorphisme, il en est de m\^eme de $\phi^{\hat a}$, et que si $\phi^{\hat a}$ est un monomorphisme, il en est de m\^eme de $(\phi^{\s})^a$. Prouvons que si $\phi^{\hat a} = (\phi^{\s})^a[\frac{1}{\varpi}]$ est un isomorphisme, $(\phi^{\s}) $ est presque surjectif. 
      
    Si $\phi^{\hat a}= \phi^{\s a}$ est un isomorphisme, il en est de m\^eme de $(\phi^a)_{\hat{!!}} = (\phi^{\s a})_{\hat{!!}}[\frac{1}{\varpi}]$ (\cf rem. \ref{234}). 
     On a donc un diagramme commutatif 
           \[  \xymatrix @-1pc {  (\sB^{\s a})_{!!}   \ar[d]    \,  \ar@{^{}->}[r] &  (\sB^{\s a})_{\hat{!!}}   \ar[d]  \,  \ar@{^{}->}[r] &   \sB^{\s}   \ar[d]^{\phi^{\s}} \\     (\sC^{\s a})_{!!}  \,  \ar@{^{}->}[r] &   (\sC^{\s a})_{\hat{!!}}  \,  \ar@{^{}->}[r] &   \sC^{\s}        }  \]
   dans lequel la fl\`eche verticale du milieu est un isomorphisme, et le compos\'e des fl\`eches du bas est un presque-isomorphisme. Cela implique que $\phi^{\s}$ est presque surjectif.    \end{proof}

    \smallskip 
    \subsection{Extensions enti\`eres d'alg\`ebres norm\'ees uniformes.}
     
       \subsubsection{} Soit $\sA$ une $\sK$-alg\`ebre norm\'ee uniforme. Soit $\sB$ une extension de $\sA$, munie d'une norme prolongeant celle de $\sA$. On a $  \sA\cap \sB^{\s} = \sA^{\s}$.
       
    \begin{lemma}\label{L6}  
    \begin{enumerate}
    \item  La fermeture compl\`etement int\'egrale de $\sA^{\s}$ dans $\sB$ est contenue dans $\sB^{\s}$.  En particulier, $\sA^{\s}$ est compl\`etement int\'egralement ferm\'e dans $\sA$.
    
    \smallskip Supposons qu'un groupe fini $\,G\,$ agisse sur $\sB$ et que $\sB^G=\sA$. 
    
     \item Supposons que $G$ pr\'eserve la norme (ce qui est le cas en particulier si celle-ci est spectrale compl\`ete). Alors $ \sB^{\sG}  = \sA^{\s}$, $\sB$ est entier sur $\sA$, $\sB^{\s}$ est entier sur $\sA^{\s}$, et la fermeture int\'egrale de $\sA^{\s}$ dans $\sB$ est  $\sB^{\s}$.
    
    \item Supposons en outre $\sB $ soit uniforme. 
     Alors $\vert\;\vert_{sp}$ est l'unique norme spectrale $G$-invariante sur $\sB$ \'etendant la norme spectrale de $\sA$. 
       \item Supposons de surcro\^{\i}t que la norme de $\sB$ soit spectrale, et soit $g$ un \'element de $\sA$. Si la multiplication par $g$ dans $\sA$ est injective (\resp isom\'etrique), il en est de m\^eme de la multiplication par $g$ dans $\sB$.  
    \end{enumerate}  
    \end{lemma}   
 
\begin{proof}  $(1)$ Soit $b\in \sB$ tel qu'il existe un sous-$\sA^{\s}$-module de type fini de $\sB$ qui contient toutes les puissances de $b$. Comme $\sA^{\s}$ est born\'e et $\sA\to \sB$ est continu, ce sous-module est born\'e, donc $b\in \sB^{\s}$. 
 
\smallskip\noindent  $(2)$ Par hypoth\`ese, $G$ pr\'eserve $\sB^{\s}$. On a $\sB^{\sG}  = \sA\cap \sB^{\s} = \sA^{\s}$. Par ailleurs, tout anneau muni d'une action d'un groupe fini est entier sur l'anneau des invariants \cite[V, \S 1, n. 9, prop. 22]{B2}. En particulier, $\sB$ est entier sur $\sA$ et $\sB^{\s}$ est entier sur $\sA^{\s}$. Compte tenu du point $(1)$, cela entra\^{\i}ne que $\sB^{\s}$ est la fermeture int\'egrale de $\sA^{\s}$ dans $\sB$.  

Notons par ailleurs que si la norme est spectrale compl\`ete, elle est unique avec cette propri\'et\'e (\cf \ref{ns}), donc pr\'eserv\'ee par $G$.
 
 \smallskip\noindent  $(3)$ $\vert\;\vert_{sp}$ est une norme spectrale $G$-invariante \'etendant la norme spectrale de $\sA$. Soit $\vert\;\vert'$ une autre telle norme.   Soit $b\in \sB$ et consid\'erons le polyn\^ome 
 
  {\begin{small}$\displaystyle{\prod_{\gamma \in G}}$\end{small}  $\,  (T- \gamma(b))= T^d + $
   {\begin{small}$\displaystyle{\sum_{i=1}^{i=d} }$\end{small}} 
    $  a_i T^{d-i} \, \in \sA[T].$ }
  D'apr\`es \cite[3.1.2, prop.1]{BGR}, on a  
  \begin{equation}\label{e4}  \max_\gamma \, \vert \gamma(b)\vert'  = \max \vert a_i\vert^{\frac{1}{i}}, \end{equation}  d'o\`u $\max_\gamma \, \vert \gamma(b)\vert'  =  \max_\gamma \, \vert \gamma(b)\vert_{sp}   $, et on conclut que $\vert\;\vert'= \vert\;\vert_{sp}$ par $G$-invariance.
    
   \smallskip\noindent  $(4)$ d\'ecoule de la m\^eme formule: si $b\in \sB$ est remplac\'e par $ab$ ($a\in \sA$), les $a_i$ le sont par $a^i a_i $, de sorte que $\vert ab\vert =\max \vert a^i a_i\vert^{\frac{1}{i}}$.
 \end{proof}

   \subsubsection{Remarque}\label{r3}  Sans l'hypoth\`ese d'invariance de la norme dans $(2)$, $\sB^{\s}$ peut ne pas \^etre entier sur $\sA^{\s}$. Sous l'hypoth\`ese $(2)$, $\sB $ n'est pas n\'ecessairement fini sur $\sA $, et m\^eme s'il l'est, $\sB^{\s}$ n'est pas n\'ecessairement fini sur $\sA^{\s}$, comme le montrent les exemples ci-dessous.    
    
       \subsubsection{Exemples prophylactiques}
    $(1)$ \label{E2} Soit $V$ un anneau de valuation discr\`ete complet de car. $\neq 2$, d'uniformisante $\varpi$. Soient $A = V[[T]]$ et $B$ la fermeture int\'egrale de $ A$ dans une extension finie galoisienne de groupe $G$ du corps de fractions $Q( A)$ de $ A$; $B$ est alors une extension finie de $A$.  
      Munissons $\sA = A[\frac{1}{\varpi}]$, \resp $\sB = B[\frac{1}{\varpi}]$, de la norme $\,{}^{{}^{A}}\vert \;\vert $ (norme multiplicative $\varpi$-adique), \resp $\,{}^{{}^{B}}\vert \;\vert $.  Ce sont des alg\`ebres de Banach uniformes, on a $\sA^{\s} = A, \, \sB^{\s}= B$, et $G$ agit par isom\'etries.

    Soit par exemple $ B=  A[U]/ (U^2- TU +\varpi)$, de sorte que $\sB = \sA[\sqrt{T^2-4\varpi}]$. Le polyn\^ome $U^2- TU +\varpi$ est irr\'eductible sur $Q( A)$, mais se factorise dans $\widehat{Q(A)}$, et m\^eme dans $\widehat{A[\frac{1}{T}]}$: les racines sont $\; \displaystyle u= 
    \frac{\varpi}{T} +     \sum_1^\infty\,     \frac{(1\cdot 3\cdot 5\cdots (2n-1)) 2^{n}\varpi^{n+1}}{(n+1)!\, T^{2n+1}}\;$ et $\, T-u$. Ainsi $B$ admet deux normes multiplicatives induisant celle de $A$, pour lesquelles $U$, \resp $T-U$, est une unit\'e. L'alg\`ebre $ \sB'  := B[\frac{1}{\varpi}]$ munie de cette derni\`ere norme (non $G$-invariante ni compl\`ete) contient   $\frac{U}{\varpi} = \frac{1}{T-U}$ qui est de norme $1$ mais n'est pas entier sur $A$.  La topologie de $(\sB')^{\s}$ est la topologie $\varpi$-adique, tandis que celle induite sur $B$ n'est pas $\varpi$-adique, mais $U$-adique (moins fine).

    \medskip  \noindent  $(2)$ \label{E3} 
    Supposons $\sK$ parfait de caract\'eristique $2$.  Soit $\sA =   \widehat{ \sK^{\s}[[T^{{\,\frac{1}{2^\infty}}} ]]}\otimes_{\sK^{\s}} \sK $, et soit $\sB $ la cl\^oture radicielle de l'extension quadratique $\sA[U]/(U^2-TU-1)$. Enti\`ere sur $\sA$, elle est contenue dans $\frac{1}{T}(\sA + \sA U)$. On calcule
   \[ U^{\frac{1}{2}}= T^{-\frac{1}{2}}(U-1),\, 
\;\ldots\;,  U^{\frac{1}{2^i}}= T^{-\frac{1}{2}-\frac{1}{4}\cdots -\frac{1}{2^i} }(U - \alpha_i),\;\; \alpha_i\in \sA, \]
 d'o\`u 
 $\;{\displaystyle{ \sB /\sA = \bigcup_i \, T^{-\frac{1}{2}-\frac{1}{4}\cdots -\frac{1}{2^i} } U \sA =  \frac{U}{T} \;T^{\frac{1}{2^\infty}}\sA}},$ qui n'est pas de type fini sur $\sA$. Donc 
 $\sB$ ne l'est pas non plus. Par ailleurs, une r\'etraction $\sA$-lin\'eaire de $\sA \to  \sB$ enverrait $U^{\frac{1}{2^i}}$ sur $1+ a_i\in \sA$ avec $a_i= T^{\frac{1}{2^i}}(1+a_{i-1})$. Or  $ a_0= T^{\frac{1}{2}}(1+a_1)= T^{\frac{1}{2}}+ \cdots   + T^{\frac{1}{2^{i-1}}} + T^{\frac{1}{2^{i }}}(1+a_i)$  et pour  $i\to \infty$, la s\'erie ne converge pas dans $\sA$.
      Une telle r\'etraction n'existe donc pas.

      \smallskip \noindent   $(3)$\label{E4}   M\^eme si $\sB$ est finie sur $\sA$ et si $ \sA\to \sB$ est scind\'ee,  
      $\sB^{\s}$ n'est pas n\'ecessairement finie sur   $\sA^{\s}$, et l'application $\sA^{\s}$-lin\'eaire $\sA^{\s}\to \sB^{\s}$ n'est pas n\'ecessairement scind\'ee, comme l'illustre un autre
 avatar du paradoxe de Z\'enon: prenons pour $ \sA $ le corps $\sK$ pr\'ec\'edent lui-m\^eme, et pour $\sB$  l'extension quadratique $\sK[U]/(U^2-\varpi U- 1) $, o\`u $\varpi \in \sK^{\sm oo}\setminus 0$. Alors $\sB^{\s}$ est la fermeture int\'egrale de $\sK^{\s}$ dans $\sB$, et c'est aussi la cl\^oture radicielle de $\sK^{{\s}} [U]$,  
  mais le $\sK^{\s} $-module
   ${\displaystyle{ \sB^{\s}/\sK^{\s} =   \frac{U}{\varpi} \;\varpi^{\frac{1}{2^\infty}}\sK^{\s} } }$
    n'est pas de type fini,
  et il n'y a pas de r\'etraction $\sK^{\s}$-lin\'eaire de $\sK^{\s} \to  \sB^{\s}$.     
    Pour un exemple diadique de m\^eme farine, voir \cite[6.4.1]{BGR}.

         \subsubsection{} Passons aux extensions \'etales finies d'alg\`ebres de Banach uniformes.
        
        \begin{lemma}\label{L6'} Soit $\sA$ une $\sK$-alg\`ebre de Banach uniforme. L'oubli de la norme induit un foncteur pleinement fid\`ele

      \smallskip  \centerline{$\{\sA$-alg\`ebres de Banach uniformes \'etales finies$\} \to  \{\sA$-alg\`ebres \'etales finies$\}$.}
        \end{lemma}
    
   \begin{proof} La fid\'elit\'e est claire. Pour la pl\'enitude, observant qu'un quotient $\sB'$ d'une $\sA$-alg\`ebre \'etale finie $\sB$ est facteur direct ($\Spec\,\sB'$ est ouvert et ferm\'e dans $\Spec \,\sB$), on se ram\`ene au cas des isomorphismes, qui sont isom\'etriques en vertu de l'unicit\'e de la norme spectrale compl\`ete. \end{proof}
   
   \subsubsection{Remarque.} C'est en fait une \'equivalence de cat\'egories, voir \cite[prop. 2.8.16 $(b)$]{KL} (compte tenu de ce que $\sB$ est une $\sA$-alg\`ebre finie projective en tant que module, elle est munie d'une classe d'\'equivalence canonique de normes compl\`etes; le point d\'elicat est de v\'erifier qu'elles sont uniformes). Nous n'aurons pas besoin de ce fait.

    \smallskip
      \subsection{Monomorphismes (et recadrage).} 
     \subsubsection{} Les monomorphismes de $\sK\hbox{-}\bf uBan$ sont les homomorphismes continus injectifs. Il est clair que tout morphisme injectif $\phi: \sA\to \sB$
      est un monomorphisme.  Pour la r\'eciproque, on peut supposer $\sA$ non nulle; alors $\sK\oplus \ker \phi \subset \sA$ est une sous-alg\`ebre de Banach de $\sA$. On conclut en consid\'erant les deux morphismes $\sK\oplus  \ker \phi \to \sA$ qui envoient $\ker \phi$ sur lui-m\^eme et sur $0$ respectivement.
     
     Cela d\'ecoule aussi formellement du fait que $\sK\hbox{-}\bf uBan$ admet des {\it objets libres}:  $S\mapsto \sK\langle T_s\rangle_{s\in S}$ est adjoint \`a gauche du foncteur oubli $\sK\hbox{-}{\bf uBan}\to {\bf Ens}$.

         \subsubsection{}\label{E6} Supposons $\vert \sK\vert$ dense dans $\R_+$. Les recadrages du type consid\'er\'e au \S \ref{recc} donnent parfois naissance \`a des monomorphismes d'alg\`ebres de Banach uniformes, \cf \eqref{e3}. Examinons en d\'etail le cas, fondamental dans la suite, du passage du cadre $(\sK^{\s}, \sK^{\ss})$ au cadre 
         \[  (\mathfrak V = \sK^{\s}\langle T^{\e}\rangle := \widehat{\sK^{\s}[T^{\e}]} , \;\; \mathfrak m =   T^{\e}{\mathfrak V}^{\ss}  = (  T)^{\e}\sK^{\ss}{\mathfrak V}^{\ss}  ) .\]
      On a $ \sK\langle T^{\e}\rangle^{\s} = \sK^{\s}\langle T^{\e}\rangle$, et $\mathfrak m $ est plat donc $\mathfrak m= \tilde{\mathfrak m}$.  
           
           Si $\sB$ est une $\sK$-alg\`ebre de Banach uniforme, la donn\'ee d'un morphisme $ \sK\langle T^{\e} \rangle^{\s}\to \sB$ \'equivaut \`a celle d'une suite $(g^{\frac{1}{p^m}})$ compatible de racines $p^m$-i\`emes d'un \'el\'ement $g\in \sB^{\s}$ ($g^{\frac{1}{p^m}}$ sera l'image de $T^{\frac{1}{p^m}}$).
         
   \begin{lemma}\label{L7}   Soient $\sB$ une $\sK\langle T^{\e}\rangle$-alg\`ebre de Banach uniforme, et $g$ l'image de $T$. \begin{enumerate} \item L'homomorphisme canonique $(\sB^{\s a})_\ast \to \sB^{\s}[\frac{1}{g}]$ est injectif, et son image est $\varpi$-adiquement compl\`ete et \'egale \`a 
          \begin{equation}\label{e5}   g^{\f}  \sB^{\s} := \cap\,   g^{-\frac{1}{p^m}}\sB^{\s}.\end{equation}
          \item    L'homomorphisme $\sB^{\s} \to   (\sB^{\s a})_\ast$ est injectif si et seulement si $g$ n'est pas diviseur de z\'ero dans $\sB$.            
   \item   L'homomorphisme 
$ (\sB^{\s}_\ast)/\varpi  \to  (\sB^{\s}/\varpi)_\ast  $  est injectif.
   \end{enumerate}    \end{lemma} 
    \begin{proof} $(1)$ On a $\displaystyle{ \mathfrak m = \varinjlim_{\cdot T^{  \frac{1}{p^m} - \frac{1}{p^{m+1}} } }    \mathfrak V},\, $ d'o\`u
    \begin{equation}\label{e6} \displaystyle{(\sB^{\s a})_\ast =  {\rm{Hom}}_{\mathfrak V}( {\mathfrak m},\sB^{\s})= \varprojlim_{\cdot T^{  \frac{1}{p^m} - \frac{1}{p^{m+1}} } }   \sB^{\s}  }.\end{equation} 
      En particulier, $(\sB^{\s a})_\ast$ est r\'eduite tout comme $\sB^{\s}$. Ceci entra\^ine que la multiplication par $g$ dans $(\sB^{\s a})_\ast$ est presque injective: si $gb=0$, alors $(g^{\frac{1}{p^n}} b)^{p^n}= 0$, donc  $g^{\frac{1}{p^n}} b =0$. Enfin, comme $(\sB^{\s a})_\ast$ est sans $\mathfrak m$-torsion, on conclut que $(\sB^{\s a})_\ast \to \sB^{\s}[\frac{1}{g}]$ est injectif, et la formule  \eqref{e6} montre que 
     \begin{equation}\label{e6'} (\sB^{\s a})_\ast \cong g^{\f}  \sB^{\s} \end{equation}
      qui est $\varpi$-adiquement compl\`ete d'apr\`es le lemme \ref{L4}. 
      $(2)$ en d\'ecoule. 
      
      \smallskip \noindent $(3)$  Comme 
      $  \sB^{\s}$ est sans $\varpi$-torsion, on a une suite exacte 

\noindent $   0\to \varprojlim_{\cdot g^{  \frac{1}{p^m} - \frac{1}{p^{m+1}} }}  \sB^{\s}\stackrel{\cdot \varpi}{\to}  \varprojlim_{\cdot g^{  \frac{1}{p^m} - \frac{1}{p^{m+1}} }}  \sB^{\s} \to  \varprojlim_{\cdot g^{  \frac{1}{p^m} - \frac{1}{p^{m+1}} }}(\sB^{\s}/\varpi) \, . $ \end{proof}

Pour comparer $  (g^{\f}  \sB^{\s})[\frac{1}{\varpi}] $ \`a  $g^{\f}  \sB$ (qui la contient), d\'efinissons, pour tout \'el\'ement $b\in g^{\f}  \sB$, la quantit\'e
    \begin{equation}\label{e7}  \,\displaystyle \vert b \vert_{g^{\f}  \sB}   := \sup_m \,\vert g^{\frac{1}{p^m}}b\vert_{\sB } \in \R\cup \{+\infty\} .\end{equation} 
 Elle v\'erifie les propri\'et\'es d'une norme except\'e qu'elle peut prendre la valeur $+\infty$. 
L'ensemble des $b\in g^{\f}  \sB$ pour lesquels $\vert b \vert_{g^{\f}  \sB} < \infty $ forme une alg\`ebre de Banach uniforme (spectrale si $\sB$ l'est). L'alg\`ebre spectrale associ\'ee n'est autre que $\sB^{\hat a}_{\hat \ast}$ (\cf \ref{recc}): le sous-anneau $(g^{\f}  \sB)^{\s} $ form\'e des \'el\'ements dont les puissances sont born\'ees eu \'egard \`a $\vert \; \vert_{g^{\f}  \sB} $ v\'erifie
 \begin{equation}\label{e8}   (g^{\f}  \sB)^{\s} = g^{\f}  \sB^{\s}  ,\end{equation}  
  car $\;\, b\in (g^{\f}\sB)^{\s} \Leftrightarrow  \sup_{\ell, m} \vert (g^{\frac{1}{p^m}}b)^\ell\vert <\infty      \Leftrightarrow  b \in  g^{\f}\sB^{\s}$.   
  
  \begin{lemma}\label{L8} Si $g$ est inversible dans $\sB$, alors $\sB^{\s} =  g^{\f}\sB^{\s} $.
  \end{lemma} 
  
  \begin{proof} Par uniformit\'e, il existe alors une suite $\lambda_m \in \sK^{\s}$, de norme tendant vers $1$, telle que $\lambda_m g^{-\frac{1}{p^m}}\in \sB^{\s}$. Tout $b\in g^{\f}\sB^{\s} $ v\'erifie alors $\lambda_m b\in  \sB^{\s} $, d'o\`u $b\in \sB^{\s}$. 
   \end{proof}
  
  \subsubsection{Remarques}  $(1)$ Si la multiplication par $g$ est isom\'etrique, $\sB^{\s}[\frac{1}{g}] \cap \sB = \sB^{\s}$, les id\'eaux $g^{\frac{1}{p^m}}\sB^{\s}$ de $\sB^{\s}$ sont ferm\'es,  $\sB^{\s} \inj  g^{\f}  \sB^{\s}$ est isom\'etrique, et  $\vert \; \vert_{g^{\f}  \sB} $   est une norme qui fait de $g^{\f}\sB$ une alg\`ebre de Banach uniforme.  On a alors 
  
  \begin{equation}\label{asteq}   g^{\f} \sB =  g^{\f} \sB^{\s}[\frac{1}{p}] = \sB_\ast   \end{equation}
  dans le cadre $( \sK^{\s}[ T^{\e}],\,    (  T)^{\e}\sK^{\ss}\sK^{\s}[ T^{\e}] )$.
    En revanche, si  la multiplication par $g$ n'est pas isom\'etrique et m\^eme si $g$ est non-diviseur de z\'ero, non seulement l'inclusion  $\sB^{\s} \inj  g^{\f}  \sB^{\s}$ n'est pas n\'ecessairement isom\'etrique, mais $g^{\f}  \sB^{\s}[\frac{1}{\varpi}] $ peut \^etre distincte de $g^{\f}  \sB$ comme le montre l'exemple ci-dessous.

 \smallskip \noindent $(2)$ L'inclusion $\sB^{\s} \subset  g^{\f}\sB^{\s} $ est en g\'en\'eral stricte, m\^eme si la multiplication par $g$ est isom\'etrique, comme le montre l'exemple du compl\'et\'e de $ \sK + T_1^{\e} \sK\langle  T_1^{\e}, T_2^{\e}\rangle$ dans $\sK\langle  T_1^{\e}, T_2^{\e}\rangle , \; g= T_1$.

        \subsubsection{Exemple prophylactique}\label{E4'}  Supposons $\sK$ parfait de car. $p$. La $\sK^{\s}$-alg\`ebre parfaite $\sK^{\s} \langle T_1^{\e} , (\varpi^m  T_1^{\frac{1}{p^m}}T_2)^{\e} \rangle_{m\in \N}$ est de la forme $\sB^{\s}$ pour une $\sK\langle  T_1^{\e} \rangle$-alg\`ebre de Banach uniforme $\sB$. On a $\sB =T_1^{\f}\sB^{\s}[\frac{1}{\varpi}]$, distinct de $ T_1^{\f}\sB  = \sK\langle  T_1^{\e}, T_2^{\e}\rangle$.

             \subsubsection{} Supposons $g$ non-diviseur de z\'ero. Rappelons que la fermeture compl\`etement int\'egrale de $\sB^{\s}$ dans $\sB^{\s}[\frac{1}{g}]$ (\resp $\sB[\frac{1}{g}] = \sB^{\s}[\frac{1}{\varpi g}]$) est   
        \begin{equation}\label{e9} \sB^{{\s}\,\ast}_{\sB^{\s}[\frac{1}{g}]} := \{b\in \sB^{\s}[\frac{1}{g}]\mid \exists m \,\forall n ,\, g^m b^n \in \sB^{\s}\}\end{equation}
         \begin{equation}\label{e9'} (\resp \;\;\; \sB^{{\s}\,\ast}_{\sB[\frac{1}{g}]} := \{b\in \sB[\frac{1}{g}]\mid \exists m \,\forall n ,\, (\varpi g)^m b^n \in \sB^{\s}\}),\end{equation}
         et remarquons en passant que si $G$ est un groupe d'automorphismes de $\sB^{\s}$ qui fixe $g$, on a 
              \begin{equation}\label{e10} (\sB^{{\s}\,\ast}_{\sB[\frac{1}{g}]} )^G =  \sB^{{\s}G\, \ast}_{\sB^{G}[\frac{1}{g}]}. \end{equation} 
           
          \begin{sorite}\label{S2} Soient $\sB$ une $\sK\langle T^{\e}\rangle$-alg\`ebre de Banach uniforme (voire une alg\`ebre norm\'ee uniforme), et $g$ l'image de $T$. 
        Les conditions suivantes sont \'equivalentes:
          \begin{enumerate} 
     \item  $ g^{\f}\sB^{\s} =  (g^{\f}\sB^{\s})^{ \dagger}_{\sB^{\s}[\frac{1}{g}]}\,$ ($ g^{\f}\sB^{\s} $ est $p$-radiciellement ferm\'e dans $ \sB^{\s}[\frac{1}{g}]$),
  \item $ g^{\f}\sB^{\s} =  (g^{\f}\sB^{\s})^{ \dagger}_{\sB[\frac{1}{g}]}\,$ ($ g^{\f}\sB^{\s} $ est $p$-radiciellement ferm\'e dans $ \sB[\frac{1}{g}]$),
    \item $g^{\f}\sB^{\s} =  (\sB^{\s})^\ast_{\sB^{\s}[\frac{1}{g}]}$,        
    \item $g^{\f}\sB^{\s} =  (\sB^{\s})^\ast_{\sB[\frac{1}{g}]}$,        
     \item $g^{\f}\sB^{\s} $ est compl\`etement int\'egralement ferm\'e dans $\sB^{\s}[\frac{1}{g}]\,$.  
      \item $g^{\f}\sB^{\s} $ est compl\`etement int\'egralement ferm\'e dans $\sB[\frac{1}{g}]\,$.  
     \end{enumerate}  \end{sorite}  

      \begin{proof} On a  $g^{\f}\sB^{\s} \subset \sB^{{\s}\,^\ast}_{\sB^{\s}[\frac{1}{g}]}  \subset \sB^{{\s}\,^\ast}_{\sB[\frac{1}{g}]}\subset (g^{\f}\sB^{{\s}})^\ast_{\sB[\frac{1}{g}]}$ et $(6)$ implique toutes les autres conditions. Il reste \`a prouver $(1)\Rightarrow (2) \Rightarrow (6)$. 
      
   \smallskip\noindent   $(1)\Rightarrow (2) $: on peut supposer 
       $\sB$ spectrale. Soit $b\in \sB
         [\frac{1}{g}] $ tel que $b^p \in g^{\f}\sB^{\s}$. Soit $m$ tel que $g^m b \in \sB$. On a $(g^m b)^p\in \sB^{\s}$, d'o\`u $g^m b\in \sB^{\s}$.
          Donc $b\in \sB^{\s}[\frac{1}{g}]$ compte tenu du point $(5)$ du sorite \ref{s1}, et $(1)$ implique $b \in g^{\f}\sB^{\s}$. 
         
       \smallskip\noindent   $(2)\Rightarrow (6) $:   Soit $b\in   (g^{\f}\sB^{{\s}})^\ast_{\sB[\frac{1}{g}]}$: il existe $m$ tel que pour tout $n$, $(\varpi g)^m b^{p ^n} \in  g^{\f}\sB^{\s}$, donc $\varpi_{\frac{m}{p^n}} g^{\frac{m}{p^n}} b  \in  g^{\f}\sB^{\s}$ d'apr\`es $(1)$. Ceci valant pour tout $n$, on conclut que $b \in g^{\f}\sB^{\s}$ compte tenu du point $(6)$ du sorite \ref{s1}.        \end{proof}  
       
           \subsubsection{\it Remarque.} Ce dernier argument montre plus g\'en\'eralement que pour anneau $R$  contenant un suite compatible $t^{\frac{1}{p^m}}$ de racines d'un non-diviseur de z\'ero $t$,  $R^\ast_{R[\frac{1}{t}]}  =  t^{\f} R^\dagger_{R[\frac{1}{t}]} $.

     \smallskip\subsection{Epimorphismes (et localisation).}  
      \subsubsection{}\label{id}
    Dans $\sK\hbox{-}\bf uBan$, tout morphisme d'image dense est un \'epimorphisme. 
      Un \'epimorphisme $\phi$ est dit {\it extr\'emal} si dans toute factorisation $\phi= \mu \circ \lambda $ o\`u $\mu$ est un monomorphisme, $\mu$ est un isomorphisme. Dans $\sK\hbox{-}\bf uBan$, tout morphisme $ \sA\stackrel{\phi}{\to} \sB$ admet une factorisation canonique $\sA \stackrel{\psi}{\to} (\sA/ \ker \phi)^u \stackrel{\chi}{\to}  \sB$ et $\psi$ est  \'epimorphisme extr\'emal (il n'est toutefois pas clair que $\chi$ soit un monomorphisme, ni donc que tout \'epimorphisme extr\'emal soit de la forme $\psi$). Tout morphisme $\phi$ surjectif est un \'epimorphisme extr\'emal, car la norme de $\sB$ est alors \'equivalente \`a la semi-norme quotient de la norme de $\sA$ en vertu du th\'eor\`eme de l'image ouverte de Banach.
     
          \subsubsection{}\label{loc}  Passons \`a un autre cas, la {\it localisation uniforme} utilis\'ee en g\'eom\'etrie analytique sur $\sK$. Si $f_1, \ldots, f_n, f \in \sA$ engendrent l'id\'eal unit\'e, on note $(fU-f_i)_i^-$ l'adh\'erence de l'id\'eal engendr\'e par les $fU-f_i$ et on pose
     \begin{equation}   \sA\{\frac{f_1, \ldots, f_n}{f}\} := \sA\langle U_1, \ldots, U_n\rangle/ (fU_i-f_i)_i^- .  \end{equation} 
     C'est la $\sA$-alg\`ebre de Banach universelle $\sB$ pour laquelle l'image de $f$ est inversible et les $f_i/f$ sont contenus dans $\sB^{\s}$, \cf \cite[6.1.4]{BGR}\footnote{Cette r\'ef\'erence se place dans le contexte o\`u $\sA$ est affino\"{\i}de, de sorte que l'id\'eal $(fU-f_i)_i$ est ferm\'e. Quitte \`a substituer $(fU-f_i)_i^-$, les arguments de \loccit valent pour une $\sK$-alg\`ebre de Banach quelconque.}. Le morphisme $\sA \to \sA\{\frac{f_1, \ldots, f_n}{f}\} $ est un \'epimorphisme de $\sK$-alg\`ebres de Banach car $\sA[\frac{1}{f}] \to   \sA\{\frac{f_1, \ldots, f_n}{f}\}$ est d'image dense (noter que $\sA  {\to} \sA[ U_1, \ldots, U_n]/ (fU_i-f_i)_i$ se factorise \`a travers un isomorphisme $\sA[\frac{1}{f}] \stackrel{\sim}{\to} \sA[ U_1, \ldots, U_n]/ (fU_i-f_i)_i$ puisque $f_1, \ldots, f_n, f$ engendrent $\sA$). 
     
 D\'etaillons le cas fondamental pour la suite o\`u $i=1$ et $f_1=1$:
      \begin{equation}   \sA\{\frac{1}{f}\} := \sA\langle U\rangle/ (fU-1)^- .  \end{equation} 
    C'est donc la $\sA$-alg\`ebre de Banach universelle $\sB$ pour laquelle l'image de $f$ est inversible d'inverse contenu dans $\sB^{\s}$.
   On a $  \sA\{\frac{1}{f}\} = 0$ si $f\in \sA^{\ss}\,$ ($fU-1$ est alors inversible). 
    
  \begin{lemma}\label{L10}\cite[prop. 2.3]{M}   Si $\sA$ est uniforme, l'id\'eal $(fU-1)$ de $\sA\langle U\rangle$ est ferm\'e. 
  \end{lemma}  
  L'argument, qui utilise la transform\'ee de Gelfand, est r\'e\'ecrit de mani\`ere peut-\^etre plus transparente dans \cite[2.8.8]{KL}.\qed

   \smallskip     Fixons $\lambda\in \sK^{\s}$ tel que $g:= \lambda f$ soit dans $\sA^{\s}$. Voyons 
           \begin{equation}   \sA\{\frac{\lambda}{g}\} :=  \sA\{\frac{1}{f}\}    \end{equation} 
          comme $\sK\langle T\rangle$-alg\`ebre de Banach ($T\mapsto g$), et notons $\iota_{  \lambda/g}: \sA \to  \sA\{ \frac{\lambda}{g}\}$ le morphisme canonique. On a une isom\'etrie (\cf formule \eqref{et})
      \begin{equation}\label{e11} \sA\{ \frac{\lambda}{g}\} \cong \sA\hat\otimes_{\sK\langle T\rangle} {\sK\langle T\rangle} \{\frac{\lambda}{T}\} , \end{equation} 
 et $ \sA\{ \frac{\lambda}{g}\}=0$ si $g\in \lambda \sA^{\ss}$.  
  Les propri\'et\'es suivantes sont faciles \`a v\'erifier:
   
    - si $\vert \lambda\vert \leq \vert \mu\vert$, on a un morphisme canonique $\sA\{ \frac{\lambda}{g}\rangle \to \sA\{ \frac{\mu}{g}\}$, et le morphisme canonique $\sA\{ \frac{\mu}{g}\} \to \sA\{ \frac{\lambda}{g}\}\{ \frac{\mu}{g}\}$ est un isomorphisme.

   - si $\vert g-h \vert < \vert \lambda\vert $, on a un isomorphisme canonique $\sA\{ \frac{\lambda}{g}\}\cong \sA\{ \frac{\lambda}{h}\}$
   (consid\'erer le  d\'eveloppement 
  $\displaystyle \frac{1}{h}-\frac{1}{g} = \sum_1^\infty \lambda^{-m-1} (g-h)^m (\frac{\lambda}{g})^{m+1}$). 
  
  - $\displaystyle \forall a\in \sA,\; \vert \iota_{  \lambda/g}(a)\vert = \inf_m \vert(\frac{g}{\lambda})^m a\vert_\sA$.
  
  \subsubsection{Exemples}\label{E7} $(1)$ Pour l'alg\`ebre spectrale $\sA= \sK\langle T\rangle$,  on a $\, \sK\langle T, \frac{\lambda}{T}\rangle  := \sA\{ \frac{\lambda}{T}\} = \sA\{ \frac{\lambda}{T}\}^u$: c'est l'alg\`ebre spectrale des fonctions analytiques sur la {``couronne ferm\'ee"} de rayons $(\lambda, 1)$, et  $\sK\langle T, \frac{\lambda}{T}\rangle^{\s} =\sK^{\s}\langle T, \frac{\lambda}{T}\rangle  $ est la sous-alg\`ebre des fonctions born\'ees par $1$. La norme n'est multiplicative que si $\lambda =0$ ou $\vert \lambda \vert = 1\,$ (notons qu'en g\'en\'eral, si $\sA$ est une $\sK$-alg\`ebre de Banach multiplicativement norm\'ee, $  \sA\{ \frac{1}{g}\}= \sA\{ \frac{1}{g}\}^u$ est une sous-alg\`ebre de Banach du compl\'et\'e $ \widehat{Q(\sA)}$ du corps de fractions de $\sA$).
  
\noindent \smallskip $(2)$ Les alg\`ebres de Banach $\sK\langle T, \frac{\lambda}{T}\rangle$ et $\sK\langle T, \frac{\lambda^2}{T^2}\rangle$ sont isomorphes mais non isom\'etriques: la premi\`ere est spectrale, mais pas la seconde (la norme de $ \frac{\lambda}{T}$ est $\vert\lambda\vert^{-1}>1$).     
          
  \medskip  Il peut arriver que $ \sA\{ \frac{\lambda}{g}\}$ ne soit pas uniforme, \cf \cite[\S 3]{M} (ceci n'arrive pas si $\sA$ est affino\"{\i}de r\'eduite munie de sa norme spectrale, \cite[7.2.3, prop. 4]{BGR}).

   Quoi qu'il en soit,  $\sA \to \sA\{\frac{\lambda}{g}\}^u$ est un \'epimorphisme dans $\sK\hbox{-}\bf uBan$, car l'image de $\sA[\frac{1}{g}]$ est dense dans $\sA\{\frac{\lambda}{g}\}^u$. {\it Le foncteur $\sA \to \sA\{\frac{\lambda}{T\cdot 1}\}^u$ est adjoint de l'inclusion dans $\sK\langle T\rangle\hbox{-}\bf uBan$ de la sous-cat\'egorie pleine form\'ee dans alg\`ebres $\sB$ telles que $T.1$ soit inversible et contenu dans $\sB^{\s}$}. 
   
   Si $\vert\sK\vert$ est dense dans $\vert\sA\vert$, on a, avec le sorite \ref{s1} (2)(4) 
   \begin{equation}\label{e13}  ( \sA\{ \frac{\lambda}{g}\}^u)_{\leq 1} = ( \sA\{ \frac{\lambda}{g}\}^u)^{\s} =  \sA\{ \frac{\lambda}{g}\}^{\s}_\ast =  \sA^{\s}\langle\frac{\lambda}{g}\rangle^\ast_{\sA^{\s}\langle\frac{\lambda}{g}\rangle[\frac{1}{\varpi}] } =  \widehat{\sA^{\s}[ \frac{\lambda}{g}]^\ast_{\sA[ \frac{\lambda}{g}]}} ,  
    \end{equation}
 ainsi que
     \begin{equation}\label{e14}   \sA\{ \frac{\lambda}{g}\}^{\s}  =  \widehat{\sA^{\s}[ \frac{\lambda}{g}]^+_{\sA[ \frac{\lambda}{g}]}} \end{equation}
    (\cf \cite[2.5.9 (d)]{KL}; c'est cet anneau qui intervient dans la th\'eorie des espaces de Huber).
    
  \begin{lemma}\label{L11}  On suppose $g\in \sA$ non diviseur de z\'ero. Soit $\iota:  \sA\inj \sA' $ un monomorphisme d'alg\`ebres de Banach tel que $\sA'$ soit contenue dans $\sA[\frac{1}{g}]$ (en tant que $\sA$-alg\`ebre).  Alors le morphisme induit $ \sA\{ \frac{\lambda}{g}\}\to  \sA'\{ \frac{\lambda}{g}\}$ est un isomorphisme d'alg\`ebres de Banach.
    \end{lemma} 
  
  \begin{proof} (d'apr\`es O. Gabber). Pour tout $n\in \N$, soit $V_n$ le sous-espace ferm\'e de $\sA' \times \sA$ form\'e des couples $(a', a = g^n a')$. Via la projection selon la premi\`ere composante, l'espace de Banach $\hat\oplus V_m$ se surjecte sur $\sA'$. D'apr\`es le th\'eor\`eme de l'image ouverte et le th\'eor\`eme de Baire, il existe $n$ tel que le sous-espace ferm\'e $\oplus_{m\leq n} V_m$ se surjecte ouvertement sur $\sA'$. Donc la multiplication par $g^n$ induit une application continue $\sA' \to \sA$, d'o\`u aussi une application continue $\sA'\{ \frac{\lambda}{g}\} \to \sA\{ \frac{\lambda}{g}\}$. Composant avec la multiplication par $g^{-n}$ dans $\sA\{ \frac{\lambda}{g}\}$, on obtient un inverse continu du morphisme canonique $\sA\{ \frac{\lambda}{g}\} \to \sA'\{ \frac{\lambda}{g}\}$.    
     \end{proof}  
     
    \begin{lemma}\label{L11'} On suppose $\sA$ uniforme, et $g$ non-diviseur de z\'ero. Alors $\displaystyle \sA^{\s}\to \lim_{\lambda \to 0}\, \sA\{ \frac{\lambda}{g}\}^{  \sm o}$ est injectif. 
    \end{lemma}
    
   \begin{proof}  En effet, on peut supposer $\vert \;\vert$ spectrale. Pour tout $a\in \sA^{\s}\setminus 0$ et tout $\lambda\in \sK^{\s}\setminus 0$, on a $\vert a\vert_{\sA\{\frac{\lambda}{g}\}} =    \displaystyle \inf_{m}  {\vert (\frac{g}{\lambda})^m a\vert} $. Si $\vert\lambda\vert \leq  \vert ga\vert, \,  \displaystyle \inf_{m}  {\vert (\frac{g}{\lambda})^m a\vert} \geq 1$, donc  l'image de $\,a\,$ dans $\sA\{\frac{\lambda}{g}\}^{\s}$ n'est pas nulle. \end{proof}

  \subsubsection{Remarques}\label{r4} $(1)$ $\iota_{\lambda/g}$ est injectif si et seulement si $ \displaystyle \inf_m \vert(\frac{g}{\lambda})^m a\vert_\sA = 0 \Rightarrow a= 0$.  C'est le cas en particulier si la multiplication par $g$ dans $\sA$ est isom\'etrique. 
   En revanche, la multiplication par $g$ dans $\sA\{ \frac{\lambda}{g}\}^{\s}$ n'est pas isom\'etrique si $\vert \lambda \vert < 1$, mais les morphismes $\sA\{ \frac{\lambda}{g}\}^u\to  \sA\{ \frac{\mu}{g}\}^u$ sont tous injectifs (comme on le voit en prenant $\mu=1$ et en utilisant la densit\'e de $\sA[\frac{1}{g}]$).  
  
 \smallskip\noindent $(2)$    Le compl\'et\'e de $\sK^{\s}[T_1, T_2, \frac{T_2T_1^i}{\varpi^i}]_{i\geq 1}$ est la boule unit\'e $\sA^{\s}$ d'une $\sK$-alg\`ebre de Banach uniforme $\sA$ sans $T_1$-torsion, mais on a  
 $\iota_{\lambda/T_1} (T_2)=0$ si $\vert\lambda \vert > \vert \varpi\vert$, puisque $T_2$ est alors infiniment divisible par $\varpi$ dans $\sK^{\s}[T_1, T_2, \frac{T_2T_1^i}{\varpi^i}, \frac{\lambda}{T_1}]$.

     \smallskip\subsection{Produits (et transform\'ee de Gelfand).}\label{prodG}

      \subsubsection{} Les produits infinis d'alg\`ebres de Banach n'existent pas en g\'en\'eral: si $\sA$ n'est pas uniforme et si $a_n$ est une suite non born\'ee d'\'el\'ements de $\sA^{\s}$, la donn\'ee des morphismes $\sK \stackrel{1\mapsto a_n}{\to} \sA  $ ne peut \'equivaloir \`a celle d'un morphisme de $\sK$ vers une alg\`ebre de Banach.  
      
   En revanche, du fait que les morphismes d'alg\`ebres munis d'une norme spectrale v\'erifient $\vert \vert \phi\vert \vert \leq 1$, 
      la cat\'egorie des $\sK$-alg\`ebres de Banach uniformes a des produits quelconques (le produit vide \'etant l'alg\`ebre nulle): 
    \begin{equation}\label{e15} \prod^u \sA^\alpha:= \{  (a^\alpha) \in \prod \sA^\alpha,  \; \vert (a^\alpha)\vert :=\sup \vert  a^{\alpha}\vert_{sp}  < \infty\}.\end{equation}    
    On a    \begin{equation}\label{e16} (\prod^u \sA^\alpha)^{\s} =  \prod \sA^{\alpha{\s}}. \end{equation}    
     
    \subsubsection{Exemple prophylactique.}\label{E40} On peut se demander si toute $\sK$-alg\`ebre de Banach spectrale $\sA$ se d\'ecompose 
       en produit uniforme $ \sA = \prod^u \, \sA^\alpha$
        de $\sK$-alg\`ebres de Banach spectrales {``connexes"}\footnote{Une alg\`ebre de Banach uniforme est connexe (\ie n'a d'autre idempotent que $0$ et $1$) si et seulement si son spectre de Berkovich $\mathcal M(\sA)$ est connexe, \cf  \cite[cor. 7.4.2]{Be}.}.  
        
      Il n'en est rien. Soit en effet ${\rm{B}}(\sA)$ l'alg\`ebre de Boole des idempotents de $\sA$ (ou de $\sA^{\s}$, c'est la m\^eme chose). Une telle d\'ecomposition, qui \'equivaut \`a $\sA^{\s} = \prod  \, \sA^{\alpha {\s}}$,  a lieu si et seulement si ${\rm{B}}(\sA)$ est complet (toute famille a un supremum) et atomique (tout \'el\'ement non nul est le supremum d'une famille d'\'el\'ements minimaux non nuls).  
   
          Inversement, partons d'une alg\`ebre de Boole ${\rm{B}}$ et construisons la $\sK$-alg\`ebre de Banach uniforme  
      $\,\sA :=( \sK\langle T_b\rangle_{b\in {\rm{B}}}/(T_b T_{b'} - T_{b\wedge b'},\, T_b + T_{\neg b'} -1))^u  \,$  topologiquement engendr\'ee par les idempotents $T_b$. On a ${\rm{B}}(\sA)\cong {\rm{B}}$. Si ${\rm{B}}$ est par exemple une alg\`ebre de Boole libre infinie - jamais compl\`ete ni atomique -, $\sA$ n'admet pas de d\'ecomposition en produit uniforme de facteurs connexes.
          En s'appuyant sur \cite[cor. 9.2.7]{Be} (voir aussi \cite{M2}), on peut d'ailleurs montrer que le spectre de Berkovich $\mathcal M(\sA)$ s'identifie au spectre maximal de ${\rm{B}}$ (qui lui-m\^eme s'identifie \`a l'espace profini des composantes connexes de $\Spec \, \sA$), et $\sA$ \`a l'alg\`ebre des fonctions continues born\'ees \`a valeurs dans $\sK$ sur cet espace.

 \subsubsection{}\label{trG} La {\it transform\'ee de Gelfand} d'une $\sK$-alg\`ebre de Banach uniforme $\sA \neq 0$ est  \begin{equation}\label{e19}\displaystyle \Gamma\sA := \prod^u_{x\in \mathcal M(\sA)} \, \mathcal H{(x)},\end{equation}
  le produit \'etant pris sur les points du spectre de Berkovich, c'est-\`a-dire les semi-normes multiplicatives born\'ees, et $\mathcal H{(x)}$ d\'esignant le compl\'et\'e du corps r\'esiduel de $x$, \cf \cite[cor. 1.3.2]{Be}.
      On a un monomorphisme canonique $\sA \to \Gamma\sA$, qui est isom\'etrique si et seulement si $\sA$ est spectrale \cite[1.3.2]{Be}.  
   Le spectre $\mathcal M(\Gamma \sA)$ est le compactifi\'e de Stone-Cech de $\mathcal M(\sA)$ discr\'etis\'e \cite[1.2.3]{Be}, et l'application $\mathcal M(\Gamma \sA)\to \mathcal M(\sA)$ est la surjection canonique. 
  
  Cette construction d\'efinit un endofoncteur $\Gamma$ de $\sK\hbox{-}{\bf uBan}$: tout morphisme $\sA\stackrel{f}{\to} \sB$ induit un homomorphisme {``diagonal"} continu $\displaystyle\mathcal H{(x)} \to \prod^u_{f_\ast(y)= x} \mathcal H{(y)}$
  et on obtient $\Gamma(f)$ en prenant le produit uniforme lorsque $x$ parcourt $\mathcal M(\sA)$. C'est un foncteur fid\`ele (puisque $\sB \to \Gamma\sB$ est un monomorphisme), non plein. 
 
 On d\'eduit de l\`a qu'un homomorphisme continu $\sA \to \sB$ d'alg\`ebres de Banach spectrales est isom\'etrique si le morphisme associ\'e $\mathcal M(\sB)\to \mathcal M(\sA)$ est surjectif (la r\'eciproque n'est pas vraie comme le montre l'exemple du plongement de $\sK\langle T\rangle$ dans le compl\'et\'e de son corps de fractions).

  \newpage
 \subsection{Limites.}\label{ulim}        
         
   \subsubsection{}  La cat\'egorie des $\sK$-alg\`ebres de Banach uniformes a des \'egalisateurs: les sous-alg\`ebres repr\'esentant ces \'egalisateurs au sens usuel sont ferm\'ees. Comme elle a des produits, elle est {\it compl\`ete}, \ie admet toutes les (petites) limites \cite[V.2 cor. 2]{mL}.
   
  Dans le cas d'un syst\`eme projectif $(\sA^{\alpha})$  de $\sK$-alg\`ebres de Banach uniformes index\'e par un ensemble ordonn\'e $(\{\alpha\}, \leq)$ (non n\'ecessairement filtrant), la limite s'obtient comme  {\it limite {uniforme}}
   d\'efinie par
    \begin{equation}\label{e17} {\rm{{ulim}}}\, \sA^{\alpha} := \{a = (a^{\alpha})\in  { \rm{lim}}\,  \, \sA^{\alpha}, \;  \vert a \vert :=\sup \vert  a^{\alpha}\vert_{sp}   < \infty\}.\end{equation} 
     C'est bien une alg\`ebre de Banach uniforme (et m\^eme spectrale), et on a (compte tenu de ce que $(\;)^{\s}$ est la boule unit\'e pour la norme spectrale)   
      \begin{equation}\label{e18}( {\rm{{ulim}}}\, \sA^{\alpha})^{\s} = { \rm{lim}}\,  \,\sA^{\alpha {\s}}.\end{equation} 
      En particulier, les conditions $(1)$ \`a $(4)$ du cor. \ref{C1} sont pr\'eserv\'ees par passage \`a la limite. 
      
 \subsubsection{Exemples}\label{E8} $(1)$ Soit $(r_i)$ une suite de r\'eels positifs tendant vers $1$ en croissant. Alors les alg\`ebres de fonctions analytiques en la variable $T$ sur les {``disques ferm\'es"} de rayon $r^i$ forment un syst\`eme projectif de $\sK$-alg\`ebres de Banach uniformes, de limite {uniforme} $\sK^{\s}[[T]]\otimes_{\sK^{\s}} \sK$, l'alg\`ebre des fonctions analytiques born\'ees sur le {``disque unit\'e ouvert"}.
    
   \smallskip\noindent  $(2)$ Pour $\vert \varpi\vert < 1$, les alg\`ebres $\sK\langle T, \frac{\varpi^{j}}{T}\rangle $ de fonctions analytiques sur les {``couronnes ferm\'ees"} de rayons $(\vert \varpi\vert^j ,1)$  forment un syst\`eme projectif de $\sK$-alg\`ebres de Banach uniformes, de limite {uniforme} $\sK\langle T\rangle $: c'est le cas le plus simple du th\'eor\`eme d'extension de Riemann non archim\'edien (\cf \S 4).

  \begin{lemma}\label{L13} Soit $(\sA^{\alpha})$ un syst\`eme projectif de $\sK$-alg\`ebres de Banach uniformes, de limite uniforme $\sA$. Alors l'homomorphisme canonique 
  $\sA^{\s}/\varpi \to \lim ( \sA^{\alpha\sm o}/\varpi)$  est injectif. \end{lemma}
  
  \begin{proof} En effet, comme $  \sA^{\alpha\sm o}$ est sans $\varpi$-torsion, on a une suite exacte 
$  0\to \lim \sA^{\alpha\sm o}\stackrel{\cdot \varpi}{\to} \lim \sA^{\alpha\sm o} \to \lim (\sA^{\alpha\sm o}/\varpi)   \,$. \end{proof}

 \begin{lemma}\label{L14}    Soit $(\sA^{\alpha})$ un syst\`eme projectif de $\sK\langle T^{\e}\rangle$-alg\`ebres de Banach uniformes. La limite uniforme est canoniquement une $\sK\langle T^{\e}\rangle$-alg\`ebre de Banach uniforme. Notons $g^{\frac{1}{m}}$ l'image de $T^{\frac{1}{m}}$ dans chaque $\sA^{\alpha}$.
      On a un isomorphisme canonique 
     \[ g^{\f} \lim \sA^{\alpha{\s}} \cong \lim   g^{\f}\sA^{\alpha {\s}}.\]
     \end{lemma}

 \begin{proof} Cela d\'ecoule, dans le cadre $( \sK^{\s}[ T^{\e}],  T^{\e}\sK^{\ss}[ T^{\e}])$, de ce que $( )^a$ et $( )_\ast$ pr\'eservent les limites; alternativement, de l'interversion de limites $  \lim_m \lim_\alpha g^{-\frac{1}{p^m}}\sA^{\alpha\sm o}=  \lim_\alpha \lim_m g^{-\frac{1}{p^m}}\sA^{\alpha\sm o} $.  \end{proof}

 \noindent Ce lemme s'applique notamment au cas du syst\`eme de localisations $(\sA\{ \frac{\varpi^j}{g}\}^u)$  (qui contiennent $\frac{1}{g}$) 
et montre, compte tenu du lemme \ref{L8},  que la limite uniforme $\tilde \sA$ v\'erifie $\tilde \sA^{\s} = g^{\f}\tilde \sA^{\s}$.

\newpage\subsection{Colimites.}\label{{ucolim}}  
             \subsubsection{}  Soit $(\sB_{\alpha})$ un syst\`eme inductif de $\sK$-alg\`ebres de Banach uniformes index\'e par un ensemble $(\{\alpha\}, \leq)$ {\it filtrant}.  On munit $\rm{colim}\, \sB_{\alpha} $ de la semi-norme (spectrale) limite $\vert(b_{\alpha})\vert = { \rm{lim}}_\alpha\, \vert b_{\alpha}\vert_{sp}$ (bien d\'efinie en vertu de la monotonie), et on d\'efinit la {\it colimite {uniforme}} comme \'etant le compl\'et\'e s\'epar\'e $\, \rm{{ucolim}}\, \sB_{\alpha} \,$ de cette colimite semi-norm\'ee. Munie de la norme induite, c'est une alg\`ebre de Banach uniforme (et m\^eme spectrale). 
     
   On v\'erifie imm\'ediatement que la colimite filtrante {uniforme} est une colimite filtrante dans $\sK\hbox{-}\bf uBan$, et on a $(\rm{ucolim}\, \sB_{\alpha} )^{\s}=  \widehat{\rm{colim}\, \sB_{\alpha}^{\s}} $ (avec la notation du sorite \ref{s1}, on a ${}^{\rm{colim}\, \sB_{\alpha}^{\s}}\vert\,b\vert=  { \rm{lim}}_\alpha\, \vert b_{\alpha}\vert_{sp}$).  
    
  Si les normes sont spectrales et les morphismes $ \sB_{\alpha}\to \sB_{\beta}$ isom\'etriques, il en est de m\^eme de $ \sB_{\alpha}\to { \rm{{ucolim}}}\,\sB_{\beta}$.
    
    \subsubsection{}\label{coco} Comme $\sK\hbox{-}\bf uBan$ admet des colimites filtrantes et des sommes amalgam\'ees $\hat\otimes^u$ (donc tous les co\'egalisateurs puisqu'il y a un objet initial), elle est {\it cocompl\`ete}, \ie admet toutes les (petites) colimites (\cf \cite[IX th. 1, V th. 2]{mL}).    
    
   Les  ${ \rm{{ucolim}}}$ commutent aux $\hat\otimes^u$, donc \`a la localisation uniforme $(\,)\{\frac{\lambda}{g}\}^u$ en vertu de la formule \eqref{e11}. 

 \subsubsection{Exemples}\label{E9}  $(1)$ Si $\sK$ est de caract\'eristique $p>0$, le compl\'et\'e de sa cl\^oture radicielle est $\, \rm{{ucolim}}\, \sK^{\frac{1}{p^i}}$.

 \smallskip\noindent  $(2)$ Soient $\sA$ une $\sK$-alg\`ebre de Banach spectrale, et $g$ un \'el\'ement de $\sA^{\s}$. Voyons $\sA$ comme $\sK\langle T\rangle$-alg\`ebre via $T \mapsto g$. Notons que $\sK\langle T^{\frac{1}{p^{i+1}}}\rangle$ admet la base orthogonale $1, T^{\frac{1}{p^{i+1}}}, \ldots, T^{\frac{p-1}{p^{i+1}}}$ sur $\sK\langle T^{\frac{1}{p^{i}}}\rangle$. On pose
   \begin{equation}\label{eqX}  \displaystyle \sA\langle g^{\frac{1}{p^{i }}}\rangle := \sA\hat\otimes^u_{\sK\langle T\rangle} \sK\langle T^{\frac{1}{p^{i }}}\rangle,\;\; \sA\langle g^{\e}\rangle := {\rm{ucolim}}\,\sA\langle g^{\frac{1}{p^{i }}}\rangle =  \sA\hat\otimes^u_{\sK\langle T\rangle} \sK\langle T^{\e}\rangle.\end{equation}
 Alors d'apr\`es la rem. \ref{r.2}, $\sA\langle g^{\frac{1}{p^{i }}}\rangle \to \sA\langle g^{\frac{1}{p^{i +1}}}\rangle$ est isom\'etrique, et il en est donc de m\^eme de $\sA\to \sA\langle g^{\frac{1}{p^{\infty }}}\rangle$.  En outre, on a 
   \begin{equation}\label{eqY} \displaystyle \sA\langle g^{\frac{1}{p^{i }}}\rangle = (\sA\langle T^{\frac{1}{p^{i }}}\rangle/ (T-g))^u\end{equation} (\cf formule \eqref{eu}).

    \smallskip\noindent    $(3)$ Toute $\sK$-alg\`ebre de Banach uniforme $\sB$ est colimite (uniforme) filtrante de $\sK$-alg\`ebres affino\"{\i}des r\'eduites $\sB_\alpha$, en telle mani\`ere que ${\rm{colim}}\, \sB_\alpha\to \sB$ est bijectif \cite[2.6.2]{KL}\label{indub}. Il suffit, pour tout sous-ensemble fini $\alpha \subset \sB^{\s}$, de prendre pour $\sB_\alpha$ l'alg\`ebre affino\"{\i}de r\'eduite image dans $\sB$ de l'objet libre $\sK\langle T_s\rangle_{s\in \alpha} $ sur $\alpha$. L'homomorphisme ${\rm{colim}}\,\sB_\alpha^{\s}   {\to} \sB^{\s} $ est clairement surjectif, et il est injectif puisque chaque $\sB_\alpha^{\s} \to \sB^{\s} $ l'est. D'o\`u le r\'esultat  en compl\'etant puis inversant $\varpi$.
  
  Notons toutefois que les $\sK$-alg\`ebres affino\"{\i}des r\'eduites sont par d\'efinition topologiquement de pr\'esentation finie. Mais elles ne sont pas de pr\'esentation finie dans $\sK\hbox{-}\bf uBan$ au sens cat\'egorique, \ie ne commutent pas aux colimites filtrantes, comme on le voit d\'ej\`a avec $\sK\langle T\rangle $: on a $\;{\rm{colim}}\, {\rm{Hom}}(\sK\langle T\rangle, \sB_\alpha) = {\rm{colim}}\, \sB_\alpha^{\s},\; {\rm{Hom}}(\sK\langle T\rangle, {\rm{ucolim}}\,\sB_\alpha) = \widehat{{\rm{colim}}\, \sB_\alpha^{\s}} $. 
  
\smallskip\noindent    $(4)$ Les alg\`ebres $\sK\langle \frac{T}{\varpi^i}  \rangle $ de fonctions analytiques sur les {``disques  ferm\'es"} de rayon $\vert \varpi\vert^i  $ forment un syst\`eme inductif de $\sK$-alg\`ebres de Banach uniformes. La colimite au sens usuel est l'alg\`ebre des germes de fonctions analytiques en l'origine. La semi-norme limite est donn\'ee par $\vert \sum a_j T^j\vert  :={ \rm{lim}_i \max_j} \vert a_j \varpi^{ij}\vert = \vert a_0\vert $ (compte tenu de ce que pour $i>>0$, $\max_j$ est atteint pour le premier $j$ tel que $a_j\neq 0$). La colimite {uniforme} est donc $\sK$, les applications $\sK\langle \frac{T}{\varpi^i}  \rangle \to \sK$ \'etant donn\'ees par l'\'evaluation en $0$. 
  
 \smallskip Ce dernier exemple est un cas particulier du r\'esultat suivant (non utilis\'e dans la suite de l'article, mais qu'on pourra mettre en regard du th. \ref{T5} - colimite {\it vs.} limite, voisinages tubulaires de $f=0$ {\it vs.} compl\'ementaires de voisinages tubulaires de $g=0$).
 
 \begin{prop}\label{P2'} Soient $\sB$ une $\sK$-alg\`ebre de Banach uniforme, et $f\in \sB^{\s}$. On a un isomorphisme canonique  
 \begin{equation}\label{eqZ} {\rm{ucolim}}\, \sB\{\frac{f}{\varpi^i}\}^u \stackrel{\sim}{\to}  (\sB/f\sB)^u   \end{equation} 
 (o\`u $(\;)^u$ d\'esigne le compl\'et\'e-s\'epar\'e pour la semi-norme spectrale associ\'ee).
 \end{prop}
       
          \begin{proof}  Comme la norme de l'image $f$ dans  $\sB\{\frac{f}{\varpi^i}\}^u$ est $\leq \vert \varpi\vert^i$, l'image de $f$ dans ${\rm{ucolim}}\, \sB\{\frac{f}{\varpi^i}\}^u$ est nulle. Par ailleurs, $\sB\{\frac{f}{\varpi^i}\}/ f = \sB\langle U\rangle / (f, \varpi^iU-f)= \sB/f$. On en d\'eduit un diagramme
        \[  \xymatrix{  \sB  \ar[d] \ar[rd] \\  {{\rm{ucolim}}\, \sB\{\frac{f}{\varpi^i}\}^u }  \ar@<2pt>[r] & (\sB/ f \sB)^u   \ar@<2pt>[l] } \] 
               o\`u les deux triangles commutent. Pour conclure, il suffit de voir que les fl\`eches partant de $\sB$ sont des \'epimorphismes. Or c'est clair pour $\sB \to (\sB/f\sB)^u$ (\'epimorphisme extr\'emal non n\'ecessairement surjectif) et $\sB\to \sB\{\frac{f}{\varpi^i}\} $, et les colimites pr\'eservent les \'epimorphismes.
                    \end{proof}  
          
          \begin{cor} Soient $\sA$ une $\sK$-alg\`ebre de Banach uniforme, et $g\in \sA^{\s}$. On un isomorphisme canonique  
 \begin{equation}\label{ec1} {\rm{ucolim}}\, \sA\langle T^{\e}\rangle \{\frac{T-g}{\varpi^i}\}^u \stackrel{\sim}{\to} \sA\langle g^{\e}\rangle . \end{equation} 
          \end{cor} 
              
            \begin{proof} D'apr\`es \eqref{eqY} et \eqref{eqZ}, on a 
          $ \,{\rm{ucolim}}_i\, \sA\langle T^{\frac{1}{p^{j }}}\rangle\{\frac{T-g}{\varpi^i}\}^u \stackrel{\sim}{\to}   \sA\langle g^{\frac{1}{p^{j }}}\rangle , \,$
             d'o\`u le r\'esultat en passant \`a la colimite uniforme sur $j$ et intervertissant avec la colimite uniforme sur $i$.
                  \end{proof}

 \newpage
   \section{ La cat\'egorie bicompl\`ete des alg\`ebres perfecto\"{\i}des.} 
       
   \medskip   
              Il s'agit d'alg\`ebres de Banach uniformes $\sA$ dont l'endomorphisme de Frobenius sur la r\'eduction modulo $p$ de $\sA^{\s}$ est surjectif.  
  Nous passons en revue les propri\'et\'es g\'en\'erales de ces alg\`ebres, notamment le basculement et la pr\'eservation par extension finie \'etale. Nous verrons aussi que les limites et colimites sont repr\'esentables, et que le plongement de la cat\'egorie des alg\`ebres perfecto\"{\i}des dans celles des alg\`ebres de Banach uniformes poss\`ede un adjoint \`a droite. Les colimites sont donc les colimites uniformes, mais les limites ne sont pas les limites uniformes en g\'en\'eral.
     
 \medskip {\it On note $F$ l'endomorphisme de Frobenius $x\mapsto x^p$ de tout anneau commutatif de caract\'eristique $p>0 $}.

  \medskip\subsection{Corps perfecto\"{\i}des}  
 
   \subsubsection{} Soit $\sK$ un corps complet pour une valeur absolue {\it{non archim\'edienne non discr\`ete}}, de corps r\'esiduel $k$ de caract\'eristique $p>0$ .   
  Suivant \cite[\S 3]{S1}, on dira que  $\sK$ est un {\it corps perfecto\"{\i}de} s'il v\'erifie l'une des conditions \'equivalentes suivantes:
  
   $i)$ $\,\sK^{\s}/p\stackrel{F}{\to} \sK^{\s}/p$ est {surjectif},   
   
      $ii)$ $\,\sK^{\s}/\varpi\stackrel{F}{\to} \sK^{\s}/\varpi$ est  surjectif (pour tout $\varpi\in \sK^{\ss}\setminus 0$ tel que $p\sK^{\s}\subset \varpi\sK^{\s}$)
   
    $iii)$ $ \sigma: \sK^{\s}/\varpi_{\frac{1}{p}} \stackrel{ x\mapsto x^p}{\to} \sK^{\s}/\varpi$  est bijectif (pour tout $\varpi_{\frac{1}{p}}$ de norme $\vert\varpi\vert^{\frac{1}{p}}$).
      
    \smallskip   
 La topologie de $\sK^{\s}$ est alors la topologie $\varpi$-adique, le groupe de valuation $\Gamma$ relatif \`a $\varpi$ est $p$-divisible (donc dense dans $\R$), et $k$ est parfait.  Si $\,\sK$ est perfecto\"{\i}de de caract\'eristique $p$, il est parfait. 
     
       \smallskip\noindent  {\bf N. B.} Si l'\'ecriture abr\'eg\'ee $\sK^{\s}/\varpi$ semble sans ambigu\"{\i}t\'e, il convient en revanche de ne pas confondre le $ \sK^{\s}/\varpi $-module plat $\,\sK^{\ss}/\varpi \sK^{\ss}\cong \sK^{\ss}\otimes_{\sK^{\s}}\sK^{\s}/\varpi \,$ avec son quotient non plat $\,\sK^{\ss}/\varpi \sK^{\s}$, qui est l'id\'eal maximal de $\sK^{\s}/\varpi $.  
       
        \subsubsection{Exemples} $(1)$ {\it Le corps perfecto\"{\i}de cyclotomique $\hat K_\infty$}\label{E10} Soit de nouveau $K_0$ le corps de fractions de l'anneau de Witt $W({k})$ d'un corps parfait $k$. On fixe une fois pour toutes un syst\`eme compatible $(\zeta_{p^i})_{i \in \N}$ de racines primitives $p^i$-i\`emes de $1$ dans une cl\^oture alg\'ebrique (d'o\`u un isomorphisme $\Z_p\cong {\rm{lim}} \, \mu_{p^i}$). Consid\'erons la tour cyclotomique 
  \[K_i := K_0(\zeta_{p^i})=  K_0 \otimes_{\Q_p} \, \Q_p(\zeta_{p^i}).\] 
    On a  $ K_i^{\s}=  W({k})[\zeta_{p^i}]$, d'uniformisante $\zeta_{p^i}-1$,  et $\; K_i^{\s} \equiv (K_{i+1}^{\s})^p \;\mod \, p$.
        Le compl\'et\'e $\hat K_\infty$ de $K_\infty := \cup K_i$ est donc un corps perfecto\"{\i}de. Son groupe de valuation est  $  \frac{1}{p-1}  \Z[\frac{1}{p}]$. 
    
   \smallskip\noindent $(2)$ Soit $\pi$ une uniformisante de $K_0$.    On a  $ K_0(\pi^{\frac{1}{p^i}})^{\s}=  W({k})[ \pi^{\frac{1}{p^i}}]$,  et $\; K_0(\pi^{\frac{1}{p^i}})^{\s} \equiv (K_0(\pi^{\frac{1}{p^{i+1}}})^{\s})^p \;\mod \, p$.
        Le compl\'et\'e de $ K_0(\pi^{\e})$ est donc un corps perfecto\"{\i}de.  

\smallskip Cet argument \'el\'ementaire est de port\'ee limit\'ee: le cas du compositum de ces deux exemples lui \'echappe.  On a toutefois le r\'esultat suivant (que nous n'utiliserons que marginalement).

 {\begin{prop}\cite[prop. 6.6.6\footnote{Voir aussi l'article pr\'ecurseur \cite{FM}.}]{GR1}\label{P3} Soient $\sK$ un corps complet pour une valeur absolue non archim\'edienne non discr\`ete et  $\sK^s$ une cl\^oture s\'eparable. Alors $\sK$ est un corps perfecto\"{\i}de si et seulement s'il est profond\'ement ramifi\'e, \ie $\Omega_{(\sK^s)^{\s}/\sK^{\s}}$ est annul\'e par $\sK^{\ss}$. 
\end{prop}} 
   
\begin{cor}\label{C2} La compl\'etion $\sL$ de toute extension alg\'ebrique d'un corps perfecto\"{\i}de $\sK$ est un corps perfecto\"{\i}de.  \end{cor} 

\begin{proof} Remarquons que $\sK$ est parfait. Pour les extensions finies $\sL$, le r\'esultat suit imm\'ediatement de la proposition. On obtient le cas g\'en\'eral par passage \`a la colimite des sous-extensions finies. \end{proof}

   \medskip\subsection{Alg\`ebres perfecto\"{\i}des} 
  \subsubsection{}\label{algperf}  Suivant \cite[\S 5]{S1}, on dira qu'une alg\`ebre de Banach {uniforme} $\,\sA\,$ sur le corps perfecto\"{\i}de $\,\sK\,$ est une {\it $\sK$-alg\`ebre perfecto\"{i}de} si elle v\'erifie l'une des conditions \'equivalentes suivantes:
  
   $i)$ $\,\sA^{\s}/p\stackrel{F}{\to} \sA^{\s}/p$ est {surjectif},   
   
      $ii)$ $\,\sA^{\s}/p\stackrel{F}{\to} \sA^{\s}/p$ est presque {surjectif} dans le cadre $(\sK^{\s}, \sK^{{\ss}})$,
   
    $iii)$ $\,\sA^{\s}/\varpi\stackrel{F}{\to} \sA^{\s}/\varpi$ est {surjectif},  
   
$iv)$ $\sA^{\s}/\varpi_{ {\frac{1}{p}}}  \stackrel{x\mapsto x^p}{\to} \sA^{\s}/\varpi$ est bijectif (avec $\vert \varpi_{ {\frac{1}{p}}}\vert^p=\vert\varpi\vert$, \cf rem. \ref{r2}), 
   
   $v)$ Frobenius induit un isomorphisme (\resp presque-isomorphisme) de 
      $\sK^{\s}$-alg\`ebres
     \[\sA^{\s}/\varpi_{ {\frac{1}{p}}}\otimes_{\sK^{\s}/\varpi_  {\frac{1}{p}}, \sigma} \sK^{\s}/\varpi = \sA^{\s}/\varpi \otimes_{\sK^{\s}/\varpi , F} \sK^{\s}/\varpi  \stackrel{\cong}{\to}  \sA^{\s}/\varpi. \]

  Les morphismes d'alg\`ebres perfecto\"{\i}des sont les homomorphismes continus d'alg\`ebres, de sorte que la cat\'egorie $\sK\hbox{-}{\bf Perf}$ des $\sK$-alg\`ebres perfecto\"{\i}des est une sous-cat\'egorie pleine de $\sK\hbox{-}\bf uBan$. Elle admet un objet initial $\sK$ et un objet final $0$.  
  
  \smallskip      Si $\sA$ est une $\sK$-alg\`ebre de Banach fix\'ee, les $\sK$-alg\`ebres perfecto\"{\i}des $\sB$ munies d'un morphisme $\sA \stackrel{\phi}{\to} \sB$ sont appel\'ees {\it $\sA$-alg\`ebres perfecto\"{\i}des}.
   Elles forment de mani\`ere \'evidente une cat\'egorie $\sA\hbox{-}{\bf Perf}$.

 \subsubsection{Remarques}\label{r5} $(1)$ Si $\sA$ est une $\sK$-alg\`ebre perfecto\"{\i}de multiplicativement norm\'ee (donc int\`egre), le compl\'et\'e de son corps de fractions est un corps perfecto\"{\i}de.  

\smallskip\noindent  $(2)$ Si $\car\, \sK = p>0$, une $\sK$-alg\`ebre de Banach uniforme $\sA$ est perfecto\"{\i}de si et seulement si elle est parfaite (\ie $F$ est bijectif).

Par ailleurs, une $\sK$-alg\`ebre de Banach $\sA$ est uniforme si et seulement si elle est r\'eduite et $\sA^p$ est ferm\'e \cite[lem. 3.5]{G}\footnote{L'hypoth\`ese {``r\'eduite"}, n\'ecessaire, est oubli\'ee dans l'\'enonc\'e de \loccit, mais utilis\'ee dans la preuve sous la forme que $a \mapsto \vert a^{p^m}\vert^{\frac{1}{pm}}$ est une norme.}, autrement dit si l'endomorphisme de Frobenius est injectif et d'image ferm\'ee.  On en d\'eduit qu'une $\sK$-alg\`ebre de Banach $\sA$ est perfecto\"{\i}de si et seulement si elle est parfaite (l'uniformit\'e est automatique).  

      Il r\'esulte alors du cor. \ref{C1} que le foncteur $\sA \mapsto \sA^{\s}$ induit une \'equivalence de la cat\'egorie des $\sK$-alg\`ebres perfecto\"{\i}des  vers celle des $\sK^{\s}$-alg\`ebres ${\mathfrak{A}}$ parfaites, plates, compl\`etes, et telles que  ${\mathfrak{A}}= {\mathfrak{A}}_\ast$.
      
            \medskip En toute caract\'eristique, on a (\cf \cite[5.6]{S1}):
      
      \begin{lemma}\label{L14'} Le foncteur $\sA \mapsto \sA^{\s}$ induit une \'equivalence de la cat\'egorie des $\sK$-alg\`ebres perfecto\"{\i}des  vers celle des $\sK^{\s}$-alg\`ebres ${\mathfrak{A}}$ plates, compl\`etes, telles que  ${\mathfrak{A}}= {\mathfrak{A}}_\ast$ et que $\,{\mathfrak{A}}/\varpi_{\frac{1}{p}}\stackrel{x\mapsto x^p}{\to}\,{\mathfrak{A}}/\varpi$ soit {bijectif}.  
      \end{lemma}  
      
      \begin{proof} Compte tenu du cor. \ref{C1}, il reste \`a v\'erifier que si $\mathfrak A$ v\'erifie ces conditions, l'alg\`ebre de Banach $\mathfrak A[\frac{1}{\varpi}]$ est uniforme. Mais cela suit du sorite \ref{s1} $(5f)$. 
      \end{proof}  

On a aussi \'equivalence avec la cat\'egorie analogue de $\sK^{\s a}$-alg\`ebres.
 
    \subsubsection{Exemples.}\label{E11} 
  $(1)$ L'exemple standard est  $\sK\langle T^{\e}\rangle$,  le compl\'et\'e de $\cup \sK[T^{\frac{1}{p^i}}] $.  On a  $\sK\langle T^{\e}\rangle   = {\rm{{ucolim}}}\, \sK[T^{\frac{1}{p^i}}]$ et $\sK\langle T^{\e}\rangle^{\s}=  \sK^{\s}\langle T^{\e}\rangle$, le compl\'et\'e $\varpi$-adique de $\cup \sA^{\s}[T^{\frac{1}{p^i}}] $. Plus g\'en\'eralement, si $\sA$ est perfecto\"{\i}de, $\sA\langle T^{\e}\rangle $ l'est aussi.   Le couple $(  \sA\langle T^{\e}\rangle, (T^{\e}))$ est universel parmi les couples $(  \sB , (g^{\e}))$ o\`u $\sB$ est une $\sA$-alg\`ebre perfecto\"{\i}de et $(g^{\e})$ une suite compatible dans $\sB$ de racines $p^m$-i\`emes d'un \'el\'ement $g\in \sB^{\s}$.

  \smallskip \noindent $(2)$ {\it la $\hat K_\infty$-alg\`ebre perfecto\"{\i}de $\hat A_\infty$}
    Reprenons les notations de l'ex.  \ref{E10} $(1)$.
  Soit $A_i : = K_i^{\s}[[T_1^{{\,\frac{1}{p^i}}},\ldots, T_n^{{\,\frac{1}{p^i}}}]]\otimes_{K_i^{\s}} K_i$ la $K_i$-alg\`ebre de Banach form\'ee des s\'eries \`a coefficients born\'es dans $K_i$, munie de la norme supremum des coefficients (qui est multiplicative et induit la topologie $\varpi$-adique). Pour lorsque $i\in \N$ varie, elles forment un sys\`eme inductif de $K_0$-alg\`ebres de Banach uniformes dont les morphismes de transition sont isom\'etriques.
 
   On a $A_0^{\s}= A := W(k)[[T_1,\ldots, T_n]]$ et
pour tout $i$, $A_i^{\s} = K_i^{\s}[[T_1^{{\,\frac{1}{p^i}}},\ldots, T_n^{{\,\frac{1}{p^i}}}]]$ est un anneau noeth\'erien r\'egulier local complet de dimension $n+1$. L'unique id\'eal premier de $A_i^{\s}$ au-dessus de  $ \,p\,$ est engendr\'e par $\zeta_{p^i}-1$.
  En outre, $A_i^{\s}\equiv  (A_{i+1}^{\s})^p \;\mod\, p$.  
 
 Le compl\'et\'e $ {\displaystyle{\hat A_\infty : =  {\rm{{ucolim}}}\,   A_i}}$ de $A_\infty := \cup A_i$ est donc une $\hat K_\infty$-alg\`ebre perfecto\"{\i}de, dont l'anneau des \'el\'ements born\'es est contenu dans $ \hat A_\infty^{\s} = \widehat{\hat K_\infty^{\s}[[T_1^{{\,\frac{1}{p^\infty}}},\ldots, T_n^{{\,\frac{1}{p^\infty}}}]]}$ (mais non \'egale: $\sum \varpi_{\frac{1}{p^n}} T^n$ n'est pas 
 dans $\hat A_\infty^{\s}$); on a en fait un isomorphisme canonique 
 \begin{equation}\label{e20} \hat A_\infty^{\s} \cong W(k[[T_1^{\e},\ldots, T_n^{\e}]]) \hat\otimes \hat K_\infty^{\s}.\end{equation}
 En effet, on a un morphisme canonique $W(k[[T_1^{\e},\ldots, T_n^{\e}]]) \hat\otimes \hat K_\infty^{\s} \to \hat A_\infty^{\s}$ entre $W(k)$-modules $p$-adiquement complets sans torsion, qui est un isomorphisme modulo $p$.
 
   \smallskip \noindent $(3)$ Pour tout corps perfecto\"{\i}de $\sK$, $\widehat{\sK^{\s}[[T_1^{{\,\frac{1}{p^\infty}}},\ldots, T_n^{{\,\frac{1}{p^\infty}}}]]}[\frac{1}{p}]$ est perfecto\"{\i}de: la surjectivit\'e de $F$ modulo $p$ se v\'erifie imm\'ediatement. 
     
     \begin{prop}\label{P4} Supposons $\sK$ de caract\'eristique $p>0$.
   L'inclusion des $\sK$-alg\`ebres perfecto\"{\i}des dans les $\sK$-alg\`ebres de Banach uniformes admet un adjoint \`a gauche: $\displaystyle{\sA \mapsto {\rm{{ucolim}}}\, \sA^{\frac{1}{p^i}}= {\rm{{ucolim_F}}}\, \sA}$, la compl\'etion de la cl\^oture radicielle. Elle admet aussi un adjoint \`a droite: ${\displaystyle \sA \mapsto  {\rm{{ulim_F}}}\, \sA }$. 
   \end{prop} 
   
 \begin{proof} $(1)$ \cf \cite[prop. 5.9]{S1}.
   \noindent $(2)$ en d\'ecoule (\cf \cite[\S 6.3]{GR2}).
 \end{proof}

  \subsection{Basculement.} 
   \subsubsection{} On suppose $\car \sK = 0$. L'application multiplicative  $\displaystyle  \lim_{x\to x^p} \sK^{\s}  \to     \lim_{F} \sK^{\s}/ \varpi  $ \'etant bijective, on obtient \`a la limite une application multiplicative continue 
  $\displaystyle {\footnotesize{\#}}  :  \lim_{F} \sK^{\s}/ \varpi  \to \sK^{\s}$.  
  Il existe $\varpi^\flat\in \lim_{F} \sK^{\s}/ \varpi $ tel que $\vert {\footnotesize{\#}}(\varpi^\flat)\vert  = \vert \varpi \vert$. 
   
  On v\'erifie \cite[lem. 3.4]{S1} que $ \displaystyle \sK^{\flat {\s}} :=  \lim_{F} \sK^{\s}/ \varpi $ est l'anneau de valuation du corps perfecto\"{\i}de de caract\'eristique $p$
\begin{equation}\displaystyle \sK^\flat :=  (\lim_{F} \sK^{\s}/ \varpi )[\frac{1}{\varpi^\flat}],\;\;\; \vert x^\flat\vert  := \vert{\footnotesize{\#}}(x^\flat)\vert,\end{equation}  
 et que ${\footnotesize{\#}}$ induit un isomorphisme $\, 
 \sK^{\flat {\s}}/\varpi^\flat \cong \sK^{\s}/\varpi$.
  
  \subsubsection{Remarque.} Dans l'autre sens, il n'y a pas de choix canonique d'un corps perfecto\"{\i}de $\sK$ de car. $0$ tel que $\sK^\flat$ soit un corps perfecto\"{\i}de de car. $p$ donn\'e (\`a {``isomorphisme de Frobenius"} pr\`es, ces corps $\sK$ correspondent aux points de degr\'e $1$ de la courbe de Fargues-Fontaine \cite{FF}).
 
     \subsubsection{}\label{basc}  Soit $\sA$ une $\sK$-alg\`ebre de Banach uniforme.  
     L'application multiplicative  $\displaystyle   \lim_{x\to x^p} \sA^{\s}     \to  \lim_{F} \sA^{\s}/ \varpi \, $ est un hom\'eomorphisme \cite[th. 6.3.92]{GR2}, ce qui fournit une application multiplicative continue  
    \begin{equation}\displaystyle {\footnotesize{\#}}: \lim_{F} \sA^{\s}/ \varpi   \to \sA^{\s} .\end{equation}
     En outre,  $\displaystyle \lim_{F} \sA^{\s}/ \varpi  $ est canoniquement une $\sK^{\flat\sm o}$-alg\`ebre parfaite $\varpi^\flat$-adiquement compl\`ete. La norme associ\'ee n'est autre que autre que 
    \begin{equation}\vert a^\flat\vert : = \vert{\footnotesize{\#}}(a^\flat)\vert.    \end{equation}
     On obtient ainsi une $\sK^\flat$-alg\`ebre de Banach perfecto\"{\i}de (la {\it bascul\'ee} de $\sA$)       \begin{equation}\,\displaystyle \sA^\flat := (\lim_{F} \sA^{\s}/ \varpi) [\frac{1}{\varpi^\flat}]    \end{equation} qui v\'erifie $\displaystyle \sA^{\flat {\s}} =  \lim_{F} \sA^{\s}/ \varpi  $, et ${\footnotesize{\#}}$ s'\'etend en une application multiplicative $\sA^\flat\to \sA$. 
         Ces constructions sont fonctorielles en $\sA$; on obtient ainsi un foncteur 
         \begin{equation} \flat:\,  \sK\hbox{-}{\bf{uBan}}  \to  \sK^\flat\hbox{-}{\bf{Perf}}.\end{equation}
 S'il n'y a pas d'ambigu\"{\i}t\'e, on notera aussi $\flat$ sa restriction \`a $ \sK\hbox{-}{\bf{Perf}}$.
 
       Par ailleurs, puisque $\sA$ est uniforme,  le noyau de $F^n$ sur $\sA^{\s}/\varpi$ est $\varpi^{\frac{1}{p^n}}(\sA^{\s}/\varpi)$ (par le point $(5f)$ du sorite \ref{s1} appliqu\'e \`a la norme spectrale). Donc le noyau de l'homomorphisme   $\, \displaystyle \lim_{F} (\sA^{\s}/ \varpi)  \to \sA^{\s}/ \varpi \,$  induit par ${\footnotesize{\#}}$  est $\, \displaystyle \lim_{F}  \varpi^{\frac{1}{p^n}}(\sA^{\s}/\varpi) = \varpi^\flat \lim_{F}  (\sA^{\s}/\varpi) \,$ (noter que $\varpi^{\frac{1}{p^n}}a_n$ ne d\'etermine pas $a_n$, mais que $F$ l\`eve l'ambig\"uit\'e).  
       On conclut:
     
     \begin{prop}\label{P5}  $\sA^\flat$ est perfecto\"{\i}de, et ${\footnotesize{\#}}$ induit un homomorphisme injectif     \begin{equation} \sA^{\flat \sm o}/\varpi^\flat \inj \sA^{\s}/\varpi.    \end{equation}
      C'est un isomorphisme si et seulement si $\sA$ est perfecto\"{\i}de.
    \qed \end{prop}

       \subsubsection{Exemples}\label{E12} $(1)$ $ \sK\langle T^{\e}\rangle^{\flat}$ est canoniquement isomorphe \`a $ \sK^{\flat}\langle T^{\e}\rangle$ (les \'el\'ements not\'es $T$ se correspondant par ${\footnotesize{\#}}$).
       
   \smallskip\noindent    $(2)$ {\it La $\hat K_\infty^\flat$-alg\`ebre perfecto\"{\i}de $A_\infty^\flat$}.  Reprenons l'ex. \ref{E10}, \ref{E11} $(2)$. Posons $ \varpi^\flat := (\zeta_{p^i}) -1$. Alors $\hat K_\infty^{\flat \sm o}$ s'identifie au compl\'et\'e $\varpi^\flat$-adique de la cl\^oture radicielle de $ {k[[ \varpi^\flat ]]}$, et   $\hat A_\infty^{\flat {\s}} $ au 
 compl\'et\'e $\varpi^\flat$-adique 
   de la cl\^oture radicielle de $k[[\varpi^\flat, T_1, \ldots, T_n]]$ (comme on le voit en r\'eduisant mod. $\varpi^\flat$). Les \'el\'ements de $\,\hat A_\infty^\flat$ et de $\hat A_\infty$ not\'es $T_i$ se correspondent par ${\footnotesize{\#}}$. 
    
       \smallskip\noindent    $(3)$ Le bascul\'e de $\widehat{\sK^{\s}[[T_{\leq n}^{\e}]]}[\frac{1}{p}]$ est $\widehat{\sK^{\flat \s}[[T_{\leq n}^{\e}]]}[\frac{1}{\varpi^\flat}]$.
       
   \subsubsection{Remarques}\label{r6}    \smallskip\noindent  $(1)$ {\it Si $\sA$ est perfecto\"{\i}de, ${\footnotesize{\#}}(\sA^{\flat\sm o})$ engendre un sous-$\sK^{\s}$-module (alg\`ebre) dense de $\sA^{\s}$}.  
    En effet, soit $\mathfrak A$ l'adh\'erence de ce sous-module. L'image de ${\mathfrak A} \to \sA^{\s}/\varpi \cong \sA^{\flat \sm o}/\varpi^\flat$ est $\sA^{\s}/\varpi $, donc $\mathfrak A = \sA^{\s}$ d'apr\`es le lemme de Nakayama pour les modules complets.
 
   \smallskip\noindent    $(2)$
      Si $\sA$ est int\`egre, il en est de m\^eme de $\sA^\flat$. En effet, si $a^\flat, b^\flat\in \sA^\flat$ v\'erifient $a^\flat  b^\flat=0$, on a $  {\footnotesize{\#}}(a^\flat){\footnotesize{\#}}(b^\flat)  =  {\footnotesize{\#}} (a^\flat  b^\flat ) = 0 $ dans $\sA$,  d'o\`u $  {\footnotesize{\#}}(a^\flat)=0$ ou $  {\footnotesize{\#}}(b^\flat)=0$, et finalement $  a^\flat =0$ ou $  b^\flat =0$. 
      
   \smallskip\noindent    $(3)$ Si la norme de $\sA$ est multiplicative, il en est de m\^eme de celle de $\sA^\flat$. Et r\'eciproquement si $\sA$ est perfecto\"{\i}de, comme il suit du point $(1)$ (ou de l'int\'egrit\'e de $\sA^{\s}/\sA^{\ss} \cong \sA^{\flat\s}/\sA^{\flat\ss}$).
   
      \smallskip\noindent    $(4)$  Les \'el\'ements idempotents de $\sA$ sont en bijection avec ceux de $\sA^{\s}/\varpi$ et aussi avec ceux de $\sA^\flat$ (via $e= e^2 \in \sA^{\s} \mapsto \bar e \in \sA^{\s}/\varpi \mapsto e^\flat := (\ldots, \bar e, \bar e)\in \sA^{\flat\s} \mapsto \# e^\flat = e$).
 
        \smallskip\noindent    $(5)$   Un morphisme $\sA \stackrel{\phi}{\to} \sB$ d'alg\`ebres perfecto\"{\i}de est isom\'etrique si et seulement si $\phi^\flat$ l'est. En effet, la condition se traduit par l'injectivit\'e de $\phi^{\s} $ mod. $\varpi$, qui s'identifie \`a $\phi^{\flat\s}$ mod. $\varpi^\flat$.

  \subsubsection{}\label{bascu} Pla\c cons-nous dans le cadre $(\sK^{\s}, \sK^{{\ss}})$. Voici le premier th\'eor\`eme fondamental de la th\'eorie perfecto\"{\i}de (\cf \cite[th. 5.2, rem. 5.18, th. 3.7]{S1}).
    
    \begin{thm}\label{T3}   
   \begin{enumerate}\item
   Le basculement $\sA \mapsto \sA^\flat$ induit une \'equivalence de cat\'egories entre $\sK$-alg\`ebres perfecto\"{\i}des et $\sK^\flat$-alg\`ebres perfecto\"{\i}des (ind\'ependante du choix de $(\varpi,\,\varpi^\flat)$). 
   \item Le passage \`a $(\;)^{\s}$ suivi de la r\'eduction mod. $\varpi$ (\resp $\varpi^\flat$) induit une \'equivalence entre chacune d'elles et la cat\'egorie des $ \sK^{\sm o\, a}/\varpi $-alg\`ebres plates $\bar{\mathfrak{A}}$ dont le Frobenius induit un isomorphisme $\;\bar{\mathfrak{A}}/ \varpi_{\frac{1}{p} }\otimes_{\sK^{\s}/\varpi_{\frac{1}{p}}, \sigma} \sK^{\s}/\varpi \to \bar{\mathfrak{A}}$. \qed
   \end{enumerate} 
\end{thm}
    
   Dans   \cite[th. 1.2, th. 3.2]{F}, J.-M. Fontaine esquisse une d\'emonstration par extension directe de ses constructions bien connues en th\'eorie de Hodge $p$-adique: $\sK^{\s}$ est quotient de $W(\sK^{\flat \s } )$, et si l'on pose  
      \begin{equation}\label{e21} (\sA^{\flat})^{\sharp \sm o} :=  W(\sA^{\flat {\s}} )\otimes_{W(\sK^{\flat {\s}} ) } \sK^{\s} , \; (\sA^{\flat})^\sharp := (\sA^{\flat})^{\sharp\sm o}[\frac{1}{\varpi}],\end{equation}
    on obtient un foncteur   $\sharp: \sK^\flat\hbox{-}{\bf{Perf}} \to \sK\hbox{-}{\bf{Perf}}$ quasi-inverse de $\flat$.
       En outre 
    \begin{equation} \label{e22}\displaystyle \mathfrak A \mapsto ( \lim_F \bar{\mathfrak{A}})^{\sharp\sm o},\;\; (\resp \;  \mathfrak A \mapsto \lim_F \bar{\mathfrak{A}}) \end{equation} 
         est quasi-inverse de la r\'eduction mod.   $\varpi $ (\resp $\varpi^\flat$). 
   Voir \cite[\S 6.3]{GR2} pour les d\'etails (ainsi que \cite{FF}, \cite{KL}).   
        
  \smallskip   Pour $\sA$ perfecto\"{\i}de fix\'ee, on d\'eduit de $(1)$ que le basculement induit une \'equivalence de cat\'egories entre $\sA$-alg\`ebres perfecto\"{\i}des et $\sA^\flat$-alg\`ebres perfecto\"{\i}des (nous l'utiliserons dans le cas $\sA =  \sK\langle T^{\e}\rangle $). 
    
      \subsubsection{}\label{nat}  Plus haut, on a associ\'e fonctoriellement \`a toute $\sK$-alg\`ebre de Banach uniforme $\sA$ une $\sK^\flat$-alg\`ebre perfecto\"{\i}de $\sA^\flat$ et un monomorphisme $\sA^{\flat\sm o}/\varpi^\flat\inj \sA^{\s}/\varpi$. 
      Posons  
       \begin{equation}\label{e23}\natural := \sharp \circ \flat\, : \;\; \sK\hbox{-}{\bf{uBan}}\to \sK\hbox{-}{\bf{Perf}}. \end{equation}
       
      \begin{prop}\label{P6} \begin{enumerate} 
      \item Le foncteur $\,\natural\,$ est adjoint \`a droite de l'inclusion $\iota$ des $\sK$-alg\`ebres perfecto\"{\i}des dans les $\sK$-alg\`ebres de Banach uniformes. En outre, pour les normes spectrales, la co\"unit\'e d'adjonction $\sA^\natural \to \sA
     $ est isom\'etrique; c'est un isomorphisme si et seulement si $\sA$ est perfecto\"{\i}de.
     \item Le foncteur $\flat$ est adjoint \`a droite de $\iota \circ \sharp$.  
     \end{enumerate}
      \end{prop} 
      
      \begin{proof} $(1)$ L'adjonction vient de ce que $\natural \iota \cong  id_{\sK\hbox{-}{\bf{Perf}}}$, et l'isom\'etrie de la co\"unit\'e d'adjonction de ce que $\sA^{\natural\sm o}/\varpi \to \sA^{\s}/\varpi$ est injectif (prop. \ref{P5}). 
      
    \noindent  $(2)$ s'en d\'eduit par composition d'adjonctions, puisque $\flat \cong \flat_{\mid {\sK\hbox{-}{\bf{Perf}}}} \circ \natural $ et que $\flat_{\mid {\sK\hbox{-}{\bf{Perf}}}}$ admet comme adjoint \`a gauche son quasi-inverse $\sharp$.   \end{proof}

        En langage cat\'egorique, $\sK\hbox{-}{\bf{Perf}}$ est une {\it sous-cat\'egorie cor\'eflexive} (pleine) de $\sK\hbox{-}{\bf{uBan}}$ \cite{HS}. 
                
  \subsubsection{}  Une importante application du basculement concerne les produits tensoriels   (cf. \cite[prop. 6.18]{S1}):
     
    \begin{prop}\label{P7} Soient $\sK$ un corps perfecto\"{\i}de et $\sA$ une $\sK$-alg\`ebre perfecto\"{\i}de. Si $\sB,\sC$ sont deux $\sA$-alg\`ebres perfecto\"{\i}des, leur somme amalgam\'ee en tant qu'alg\`ebre perfecto\"{\i}de existe et co\"{\i}ncide avec la somme amalgam\'ee $\sB\hat\otimes^u_\sA\, \sC$ en tant qu'alg\`ebre de Banach uniforme. 
     L'homomorphisme canonique 
        \begin{equation}( \sB^{\s}\hat\otimes_{\sA^{\s}}\, \sC^{\s})_\ast \to ( \sB\hat\otimes_\sA\, \sC )^{\s} \end{equation} est un isomorphisme. En particulier, si $\sA,\sB,\sC$ sont munies de leur norme spectrale, on a  $\sB\hat\otimes_\sA\, \sC = \sB\hat\otimes^u_\sA\, \sC$, de boule unit\'e $(\sB^{\s}\hat\otimes_{\sA^{\s}}\, \sC^{\s})_\ast$.        
     \end{prop}  
     
  \begin{proof} Il est clair que Frobenius induit une bijection $(\sB^{\s}\hat\otimes_{\sA^{\s}}\, \sC^{\s})/\varpi_{\frac{1}{p}} \to (\sB^{\s}\hat\otimes_{\sA^{\s}}\, \sC^{\s})/\varpi$. D'apr\`es le lemme \ref{L14'}, il suffit alors de montrer que $\sB^{\s}\hat\otimes_{\sA^{\s}}\, \sC^{\s} $ est presque sans $\varpi$-torsion.  
  
  Si $\car\, \sK =p$, cela d\'ecoule du fait que $\sB^{\s}\hat\otimes_{\sA^{\s}}\, \sC^{\s}$ est parfaite (\cf rem. \ref{r5} $(2)$).
  
  Si $\car\, \sK = 0$,  il suffit d'apr\`es l'\'equivalence $(2)$ du th. \ref{T3} de voir que  $\sB^{\s a}/\varpi  \otimes_{\sA^{\s a}/\varpi}\, \sC^{\s a}/\varpi$   est plate sur $\sK^{\s a}/\varpi $, ce qui s'obtient par basculement: $\sB^{\s a}/\varpi  \otimes_{\sA^{\s a}/\varpi}\, \sC^{\s a}/\varpi  \cong  \sB^{\flat\s a}/\varpi^\flat  \otimes_{\sA^{\flat\s a}/\varpi^\flat}\, \sC^{\flat\s a}/\varpi^\flat$. Cela montre aussi que les sommes amalgam\'ees commutent au basculement.
    La troisi\`eme assertion r\'esulte de la seconde en vertu du sorite \ref{s1}$(2e)(5)$. \end{proof} 
   
    Il s'ensuit que la cat\'egorie des alg\`ebres perfecto\"{\i}des admet des colimites finies, qui se calculent comme dans la cat\'egorie des alg\`ebres de Banach uniformes (ce r\'esultat sera renforc\'e au \S \ref{lcolp}). Elles commutent \`a l'\'equivalence de basculement.
  
  On d\'eduit aussi de la proposition que le foncteur $\,- \hat\otimes_\sK \sA\,$ est adjoint \`a gauche du foncteur oubli $\sA\hbox{-}{ \bf Perf} \to \sK\hbox{-}{ \bf Perf} $.

  \medskip \subsection{Extensions \'etales finies d'alg\`ebres perfecto\"{\i}des.}\label{eeap}
   
                 \subsubsection{} Soit $\sK$ un corps perfecto\"{\i}de, et pla\c cons-nous de nouveau dans le cadre $(\sK^{\s}, \sK^{{\ss}})$. Voici le second th\'eor\`eme fondamental de la th\'eorie perfecto\"{\i}de (\cf \cite[th. 7.9]{S1}\cite{KL}), qui g\'en\'eralise le {``th\'eor\`eme de puret\'e"}  de Faltings.

   \begin{thm}\label{T4}  Soient $\sK$ un corps perfecto\"{\i}de et $\sA$ une $\sK$-alg\`ebre perfecto\"{\i}de. 
      
     \medskip \noindent $(1)$  On a des \'equivalences de cat\'egories 
     
     \medskip\noindent   $   \{\sA^{\s a}$-alg\`ebres \'etales  finies$\}   \stackrel{\sim}{\to}$  
     
     \medskip\centerline{
     $\{\sA$-alg\`ebres perfecto\"{\i}des \'etales finies$\}   \stackrel{\sim}{\to}$ } 
     
        \medskip\rightline{ $\{ \sA$-alg\`ebres \'etales finies$\}  $ }    
     \medskip \noindent   induites par inversion de $\varpi$ et oubli de la norme respectivement.
        
              \medskip\noindent $(2)$ Ces \'equivalences sont compatibles au basculement.   \end{thm}

   En particulier, toute extension \'etale finie $\sB$ de $\sA$ est perfecto\"{\i}de (la norme spectrale de $\sB$ \'etant l'unique norme spectrale compl\`ete, et aussi l'unique norme spectrale compatible avec la topologie canonique du $\sA$-module projectif fini $\sB$, \cf \ref{ns}), et tout $\sK$-automorphisme de la $\sK$-alg\`ebre $\sB$ qui fixe $\sA$ est continu, donc isom\'etrique eu \'egard aux normes spectrales.

 \subsubsection{Exemple prophylactique.}\label{E13} On ne peut s'affranchir de l'hypoth\`ese {``\'etale"} dans le point $(1)$ du th\'eor\`eme, comme le montre l'exemple suivant, o\`u l'on reprend les notations de \ref{E11} $(2)$ avec $p= n= 2$.
 
 Posons $g : = T_1 + T_2$ et consid\'erons la $\hat K_\infty$-alg\`ebre de Banach (multiplicativement norm\'ee) $\hat A_\infty[\sqrt g]$. Elle n'est pas perfecto\"{\i}de. En effet, tout \'el\'ement $a_1 + a_2 \sqrt g$ de $\hat A_\infty[\sqrt g]^{\s}$ v\'erifie  $2a_1, 2a_2, a_1^2 - a_2^2g \in \hat A_\infty^{\s}$ comme on le voit en calculant traces et norme; en particulier $a_1(0)\in \sK^{\s}$. L'\'el\'ement $f_1 := \frac{i+1}{2}(\sqrt{T_1+T_2} -\sqrt{T_1} -\sqrt{T_2}) $ est dans $\hat A_\infty[\sqrt g]^{\s} $ comme on le voit en calculant son carr\'e. Montrons que $f_1$ n'a pas de racine carr\'ee modulo $2$: si $b = a_1+a_2 \sqrt{g} \in \hat A_\infty[\sqrt g]^{\s}$ v\'erifie $  b^2 - f_1   \in     2\hat A_\infty[\sqrt g]^{\s}\subset \hat A_\infty^{\s}[\sqrt g]$, alors  $2a_1a_2 - \frac{i+1}{2} \in  \hat A_\infty^{\s},$ d'o\`u $\vert a_1(0)a_2(0)\vert = 2^{-{\frac{3}{2}}}$, ce qui contredit le fait que $a_1(0), 2a_2(0)\in \sK^{\s}$.
 
 \smallskip L'exemple \ref{E3} en car. $2$ est du m\^eme type. 
 
        \subsubsection{} Notons que les foncteurs en jeu sont pleinement fid\`eles. Commen\c cons par \'etablir la premi\`ere \'equivalence du th\'eor\`eme, c'est-\`a-dire:
                
           \begin{prop}\label{P8'} $(1)$ Si $\sA$ est une $\sK$-alg\`ebre perfecto\"{\i}de, toute $\sA^{\s a}$-alg\`ebre \'etale finie $\mathfrak B$ fournit, apr\`es inversion de $\varpi$, une $\sA$-alg\`ebre perfecto\"{\i}de $\sB$ (de mani\`ere fonctorielle si on munit $\sB$ de la norme spectrale), telle que $\sB^{\s a} \cong \mathfrak B$. 
           
           $(2)$ R\'eciproquement, si $\sB$ est une $\sA$-alg\`ebre perfecto\"{\i}de \'etale finie (\resp galoisienne de groupe $G$), alors $\sB^{\s a}$ est une $\sA^{\s a}$-alg\`ebre \'etale finie (\resp galoisienne de groupe $G$).  \end{prop} 
              
   \begin{proof} $(1)$ $\mathfrak B $ est $\varpi$-adiquement compl\`ete (lemme \ref{L1}) et plate sur $\sK^{\s a}$. D'autre part, d'apr\`es \cite[th. 3.5.13 ii]{GR1}, on a $(\mathfrak B/\varpi) \otimes_{(\sA^{\s a}/\varpi)} F^\ast(\sA^{\s a}/\varpi) \cong F^\ast(\mathfrak B/\varpi)$. Ainsi $F$ induit un isomorphisme  $\mathfrak B/\varpi_{ {\frac{1}{p}}}\otimes_{\sK^{\s}/\varpi_  {\frac{1}{p}}, \sigma} \sK^{\s}/\varpi   \stackrel{\cong}{\to}  \mathfrak B/\varpi$ puisqu'il le fait pour $\sA^{\s a}/\varpi$. Donc $\sB := \mathfrak B[\frac{1}{\varpi}]$ est perfecto\"{\i}de et $\sB^{\s a} \cong \mathfrak B$ (lemme \ref{L14'}).  
 
 \smallskip $(2)$ $a)$ Commen\c cons par le cas o\`u l'extension $\sA \inj \sB $ est galoisienne de groupe $G$. D'apr\`es le lemme \ref{L1} et la prop. \ref{P7}, on a $\displaystyle \sB^{\s a} \otimes_{\sA^{\s a}} \sB^{\s a}= (\sB\otimes_{\sA}\sB)^{\s a} = \prod_G\, \sB^{\s a}$. Ainsi $\sB^{\s a}$ est galoisienne sur $\sA^{\s a}$, donc \'etale finie (prop. \ref{P2} $(1)$). 
 
  \smallskip $b)$ En g\'en\'eral, le rang de $\sB$ sur $\sA$ \'etant fini continu et born\'e, on peut supposer, en d\'ecomposant  $\sA$ en un nombre fini de facteurs, que ce rang est constant, \'egal \`a $r$. D'apr\`es le lemme \ref{L0}, on a une extension galoisienne $\sA \inj \sC$ de groupe $\mathfrak S_r$ se factorisant par $\sB$ et telle que $\sB = \sC^{\mathfrak S_{r-1}}$.
   D'apr\`es le pas $(a)$, $\sC^{\s a}$ est une $\sA^{\s a}$-alg\`ebre galoisienne de groupe $\mathfrak S_r$. On a aussi  $\sB^{\s a}= (\sC^{\s a})^{\mathfrak S_{r-1}}$ (lemme \ref{L6} $(2)$), qui est donc \'etale finie de rang constant sur  $\sA^{\s a}$, \'egal au rang de $\sB$ sur $\sA$ (prop. \ref{P2} $(3)$).    \end{proof}

          \subsubsection{}\label{demo} Compte tenu de cette proposition, \'etablir la seconde \'equivalence du th\'eor\`eme revient \`a prouver que     
             \smallskip   $(\ast)_\sA$: {\it toute $\sA$-alg\`ebre  $\sB$  \'etale finie est perfecto\"{\i}de. }

 En voici une esquisse de d\'emonstration, suivant l'approche de \cite[th. 3.6.21]{KL}, qui ne fait plus intervenir la presque-alg\`ebre  (le lecteur v\'erifiera que le recours \`a quelques \'enonc\'es ult\'erieurs ne cr\'ee pas de cercle vicieux). En car. $p$, $(\ast)_\sA$ suit de ce qu'une alg\`ebre \'etale finie sur un anneau parfait est parfaite.  Supposons alors $\car\, \sK =0$.

            \smallskip\noindent  $(a)$ On prouve d'abord $(\ast)_\sL$ 
         pour toute extension finie $\sL$ de $\sK$, \cf cor. \ref{C2}. 
          
          \smallskip\noindent  $(b)$ On montre que la transform\'ee de Gelfand ${\Gamma(\sA)}$ est perfecto\"{\i}de: c'est un produit uniforme de corps perfecto\"{\i}des  $\mathcal H(x)$, \cf prop. \ref{P12}.  On d\'eduit de l\`a et du pas pr\'ec\'edent que pour toute $\sA$-alg\`ebre \'etale finie $\sB$, $\sB \otimes_\sA \Gamma(\sA)$ est perfecto\"{\i}de et spectrale.
          
        \smallskip\noindent  $(c)$ On a $\Gamma(\sA^\flat)\cong \Gamma(\sA)^\flat$, \cf rem. \ref{Gelbasc}. Les corps perfecto\"{\i}des $\mathcal H(x)^\flat$ sont colimites uniformes de localisations rationnelles $\sA_\alpha^\flat$ de $\sA^\flat$, et $\mathcal H(x)$ est colimite uniforme des bascul\'es $\sA_\alpha$ qui sont des localisations rationnelles de $\sA^\flat$, \cf rem. \ref{r9}. 
    On a alors un carr\'e essentiellement commutatif d'\'equivalences 
     \[ \begin{CD} 2\hbox{-}{\rm{colim}}\, ({\sA_\alpha^\flat }\hbox{-}{\bf Alg}^{et. fin})   @>>> \mathcal H(x)^\flat \hbox{-}{\bf Alg}^{et. fin}   \\\     @V VV   @VVV    \\\  2\hbox{-}{\rm{colim}}\, ({\sA_\alpha  }\hbox{-}{\bf Alg}^{et. fin})   @>>> \mathcal H(x)  \hbox{-}{\bf Alg}^{et. fin}    \end{CD}\] o\`u les \'equivalences verticales sont donn\'ees par $\sharp$ (compte tenu de $(\ast)_{\sA^\flat}$ et $(\ast)_{\sA_\alpha^\flat}$), et o\`u les \'equivalences horizontales sont donn\'ees par l'approximation {``\`a la Elkik"} \cite[prop. 5.4.53]{GR1} et [SGA 4, exp. 7 lemme 5.6] (voir aussi \cite[7.5]{S1}) (notant que ${\rm{colim}}\,{\sA_\alpha^{\s} }$ est hens\'elien, de compl\'et\'e $\sA^{\s}$)\footnote{On utilise ici le fait que pour les alg\`ebres \'etales finies,  objets et morphismes sont de pr\'esentation finie. Il n'est pas facile de trouver un expos\'e sur les $2$-limites et $2$-colimites de cat\'egories orient\'e vers les applications plut\^ot que vers les g\'en\'eralisations. Nous renvoyons \`a l'appendice A de \cite{W}, qui fait exception.}.  
         
       \smallskip\noindent  $(d)$ Par compacit\'e de $\mathcal M(\sA)$, on obtient un recouvrement fini de  $\mathcal M(\sA)$ par des localisations rationnelles $\mathcal M(\sA_\alpha)$ tel que $\sB \otimes_\sA \sA_\alpha$ soit perfecto\"{\i}de. Pour passer des $\sA_\alpha$ \`a $\sA$, on utilise les bonnes propri\'et\'es faisceautiques des alg\`ebres perfecto\"{\i}des, \cf \cite[th. 7.9]{S1}\cite[th. 3.6.21]{KL}.   
       
       \smallskip\noindent  Le point $(2)$ du th\'eor\`eme suit du point $(1)$, des \'equivalences du th. \ref{T3} et de l'\'equivalence remarquable de Grothendieck (\cf \ref{eqrem}):  
  
     \smallskip   \centerline{$\sA^{{\s} a}$-${\bf{Alg}}^{et. f} \cong \sA^{{\s} a}/\varpi$-${\bf{Alg}}^{et. f} \cong \sA^{\flat{\s} a}/\varpi^\flat$-${\bf{Alg}}^{et. f} \cong \sA^{\flat {\s} a}$-${\bf{Alg}}^{et. f} . $\qed}

 \subsubsection{Remarques.} $(1)$ \cite{KL} \'etablit la prop. \ref{P8'} par voie analytique en introduisant des anneaux de Robba;\cite{S1} l'\'etablit en m\^eme temps que $(\ast)_\sA$ en \'elaborant les points $(c)$ et $(d)$.
 
 \smallskip \noindent $(2)$ Comme sugg\'er\'e dans \cite{F} et d\'evelopp\'e dans \cite[3.6]{KL}, on peut d\'efinir une notion d'alg\`ebre perfecto\"{\i}de sur un corps $p$-adique $\sK$ non n\'ecessairement perfecto\"{\i}de: c'est une $\sK$-alg\`ebre de Banach uniforme $\sA$ telle que $F$ soit surjectif sur $\sA^{\s}/p$ et qu'il existe $\varpi_{\frac{1}{p}}$ avec $\varpi_{\frac{1}{p}}^p \equiv p\, \mod\, p^2\sA^{\s}$. Cela implique l'existence de $\varpi_{s}\in \sA$ de norme $\vert p\vert^s$ pour tout $s\in \Z[\frac{1}{p}]$, avec $\varpi_{-s} = \varpi_s^{-1}$, d'o\`u $\sA^{\ss} = (\sA^{\ss})^2$. On a des analogues du th. \ref{T3} et de la prop. \ref{P7} \cite[3.6.5, 3.6.11]{KL}\footnote{on prendra garde que les ``uniform algebras" de \cite{KL} sont nos alg\`ebres spectrales.}, ainsi que la g\'en\'eralisation correspondante du \ref{T4}  \cite[3.6.21, 5.5.9]{KL} (presque \'etale s'entend ici dans le cadre $(\sA^{\s}, \sA^{\ss})$). De nouveau, l'argument galoisien de la prop. \ref{P8'} permet de se dispenser des anneaux de Robba de \cite{KL}.
 
   \smallskip \noindent $(3)$ Tout anneau local complet r\'egulier ramifi\'e de caract\'eristique mixte $(0,p)$ est isomorphe \`a un anneau de la forme $W(k)[[T_{\leq n}]]/(p-f)$ o\`u $f$ est dans le carr\'e de l'id\'eal maximal et non divisible par $p$. Le m\^eme argument que pour le cas non ramifi\'e (ex. \ref{E11} (2)) montre que $\widehat{W(k)[[T^{\e}_{\leq n}]]/(p-f)}$ est perfecto\"{\i}de au sens g\'en\'eralis\'e ci-dessus, \cf \cite[4.9]{Sh}.
   
     \smallskip \noindent $(4)$ Le th\'eor\`eme de presque-puret\'e fournit des exemples de modules auto-duaux (via la trace, \cf rem. \ref{r1'} $(2)$), et par suite r\'eflexifs, qui ne sont pas de type fini, m\^eme sur un anneau de valuation; c'est le cas du $\sK^{\s}$-module $\sB^{\s}$ dans l'exemple \ref{E4}.
 
        \subsubsection{} Voici une premi\`ere cons\'equence du th\'eor\`eme.
        
        \begin{cor}\label{C2'}  Soit $\sB$ une extension \'etale finie d'une $\sK$-alg\`ebre perfecto\"{\i}de $\sA$. Alors $\sB^{\s}$ est entier sur $\sA^{\s}$; c'est en fait la fermeture (compl\`etement) int\'egrale de $\sA^{\s}$ dans $\sB$. 
        \end{cor}
         
      \begin{proof} Comme la fonction rang de $\sB$ sur ${\Spec} \sA$ est continue et born\'ee, on peut la supposer constante \'egale \`a $r$ en d\'ecomposant $\sA$ en produit fini. Soit $b\in \sB^{\s}$ et soit $\displaystyle \chi_b(T) = \sum_0^r\, (-1)^i \, Tr_{(\wedge^i\sB)/\sA}(\cdot b) T^{r-i}$ son polyn\^ome caract\'eristique. On a $\chi_b(b) = 0$. Par ailleurs, comme $\sB^{\s a}$ est presque projectif fini sur $\sA^{\s a}$ d'apr\`es \ref{T4}, on a  $Tr_{(\wedge^i\sB)/\sA}(\cdot b) = Tr_{(\wedge^i\sB^{\s a} )/\sA^{\s a} }(\cdot b) \in \sA^{\s a}_\ast= \sA^{\s}$. Donc $b$ est entier sur $\sA^{\s}$. La seconde assertion suit de l\`a et du lemme \ref{L6} $(1)$. \end{proof}

 \subsubsection{\it Remarque.} Si $\sA$ n'a qu'un nombre fini de composantes connexes, on peut aussi raisonner comme suit: en d\'ecomposant $\sA$ et $\sB$, on se ram\`ene aussit\^ot au cas o\`u $\sA$ et $\sB$ sont connexes. Il existe alors une cl\^oture galoisienne $\sC$ (\cite[prop. 5.3.9]{Sz}, de groupe not\'e $G$. Comme $G$ agit par isom\'etries eu \'egard aux normes spectrales,  le lemme \ref{L6} $(2)$ dit que $\sC^{\s}$ est entier sur $\sA^{\s}$, donc $\sB^{\s}$ aussi.
       
           \subsubsection{}  
  Le th\'eor\`eme \ref{T4} admet une r\'eciproque partielle, que nous n'utiliserons pas\footnote{voir aussi \cite[prop. 3.6.22]{KL}.}:
 
 \begin{prop}\label{P8} Soit $\sB$ une $\sK$-alg\`ebre perfecto\"{\i}de, extension \'etale finie d'une sous-$\sK$-alg\`ebre $\sA$. Alors $\sA$ est perfecto\"{\i}de.
 \end{prop}
 
 \begin{proof} En raisonnant comme au pas $b)$ de la prop. \ref{P8'} , on se ram\`ene au cas galoisien de groupe $G$. Alors   $\sA = \sC^G$  est perfecto\"{\i}de d'apr\`es la prop. \ref{P13} ci-dessous.    \end{proof}

 \subsection{Monomorphismes (et recadrage).}        
     \subsubsection{}\label{monop} Dans $\sK\hbox{-}\bf Perf$, les homomorphismes continus injectifs sont des monomorphismes. 
     
     La r\'eciproque est vraie en car. $p$, car le plongement $\sK\hbox{-}\bf Perf\to \sK\hbox{-}\bf uBan$ admet un adjoint \`a gauche (cela d\'ecoule aussi formellement du fait que $\sK\hbox{-}\bf Perf$ admet des {\it objets libres}: $S\mapsto \sK\langle T_s^{\e}\rangle_{s\in S}$ est adjoint \`a gauche du foncteur oubli $\sK\hbox{-}{\bf Perf}\to {\bf Ens}$). 
 
 \newpage\smallskip En car. $0$, les monomorphismes sont les morphismes $\phi$ dont le bascul\'e $\phi^\flat$ est injectif. Un tel morphisme n'est pas n\'ecessairement injectif, comme le montre l'exemple suivant. Ainsi, si $\phi$ est injectif il en est de m\^eme de $\phi^\flat$, mais la r\'eciproque n'est pas vraie.
   
   \subsubsection{\it Exemple prophylactique.}  Soit $\hat K_\infty$ le corps perfecto\"{\i}de cyclotomique, et consid\'erons la tour d'extensions perfecto\"{\i}des $\hat K_\infty[p^{\frac{1}{p^i}}]$ et le compl\'et\'e $\sL$ de leur r\'eunion, qui est encore perfecto\"{\i}de. Soit $\hat K_\infty\langle T^{\e}\rangle \stackrel{\phi}{\to} \sL$ le morphisme d'\'evaluation en $p$ (\ie $T^{{\frac{1}{p^i}}}\mapsto p^{{\frac{1}{p^i}}}$), qui n'est ni injectif ni surjectif. 
 
 Basculons: $\phi^{\flat}$ est donn\'e par l'\'evaluation en un \'el\'ement $p^\flat= (\cdots, p^{\frac{1}{p}}, p)$ de $ \sL^{\flat \s}$ ($\sL^{\flat }$ est la compl\'etion de la tour d'extensions s\'eparables $\hat K_\infty[p^{\frac{1}{p^i}}]^\flat$, et $p^\flat$ n'est dans aucune d'entre elles). Le noyau de $\phi^\flat$ est un id\'eal ferm\'e premier radiciel, et  l'alg\`ebre de Banach int\`egre $\sB := \hat K_\infty^\flat\langle T^{\e}\rangle/{\ker} \phi^\flat$ est perfecto\"{\i}de spectrale \cite[3.1.6 d]{KL}. Ce ne peut \^etre un corps car la norme serait multiplicative \cite{K}, et le plongement dans $\sL^\flat$ serait alors isom\'etrique; basculant dans l'autre sens, on aurait une factorisation de $\phi$ en un morphisme surjectif $\hat K_\infty\langle T^{\e}\rangle \to \sB^\sharp $ (\cf rem. \ref{r8-} ci-dessous) suivi d'un plongement isom\'etrique dans $\sL$, ce qui contredirait le fait que $\phi$ est d'image dense mais non surjectif. 
 
Le morphisme $\hat K_\infty^\flat\langle T \rangle \to \hat K_\infty^\flat\langle T^{\e}\rangle$ induit un hom\'eomorphisme $D_\infty := \sM(\hat K_\infty^\flat\langle T^{\e}\rangle) \to D := \sM(\hat K_\infty^\flat\langle T \rangle)$, les corps r\'esiduels $\sH(x)$ de $D_\infty$ \'etant les perfections compl\'et\'ees de ceux de $D$. Le spectre $\sM(\sB)$ est un sous-espace connexe (\cf \cite[7.4.2]{Be}) de $D_\infty$ non r\'eduit \`a un point d'apr\`es ce qui pr\'ec\`ede, donc il contient un point $x_\infty$ de {``dimension"} $d(x)\geq 1$ (\cf \cite[9.2.3]{Be}); le point $x\in D$ correspondant est alors de type $2$ ou $3$ \cite[9.1]{Be}. D'apr\`es \cite[4.4]{P}, il existe $f\in \sK^\flat\langle T\rangle$ et $r \in ]0, 1]$ tels que le domaine affino\"{\i}de de $D$ d\'efini par $\vert f\vert \leq r$ ait un bord de Shilov r\'eduit au point $x$; si $\sA$ est l'alg\`ebre affino\"{\i}de (non n\'ecessairement stricte) associ\'ee, le morphisme $\sA \to \sH(x)$ est donc isom\'etrique. Il en est alors de m\^eme de $\widehat{\sA^{\e}} \to \widehat{\sH(x)^{\e}}= \sH(x_\infty)$. Par ailleurs $\hat K_\infty^\flat\langle T^{\e} \rangle\to   \widehat{\sA^{\e}}$ est injectif (apr\`es extension des scalaires, c'est une localisation {``de Weierstrass"}), donc le compos\'e $\hat K_\infty^\flat\langle T^{\e} \rangle \to \sH(x_\infty)$ aussi. Comme il se factorise \`a travers $\sB$, on en d\'eduit que ${\ker} \, \phi^\flat =0$.

  \subsubsection{}\label{prp}  Reprenons l'exemple du monomorphisme $\sB \inj g^{\f}\sB^{\s}[\frac{1}{\varpi}]$ de \ref{E6} ($g$ non diviseur de z\'ero). Supposons $\sB$ perfecto\"{\i}de. Nous verrons plus tard (lemme \ref{uL}) que les conditions \'equivalentes du sorite \ref{S2} sont satisfaites (de sorte que $g^{\f}\sB^{\s}[\frac{1}{\varpi}]$ est la fermeture int\'egrale compl\`ete de $\sB^{\s}$ dans $\sB[\frac{1}{g}]$); elle n'est pas n\'ecessairement \'egale \`a $\sB$ comme le montre l'exemple $(\sK + T_1^{\e} \sK\langle  T_1^{\e} , T_2^{\e} \rangle, g= T_1)$. La question suivante est ouverte:  
  
  \begin{qn}\label{q} {\it $g^{\f}\sB^{\s}[\frac{1}{\varpi}]$ est-elle perfecto\"{\i}de?}\end{qn}

 \medskip\noindent C'est clair en car. $p$ puisque cette derni\`ere est parfaite. En car. $0$, en revanche, m\^eme si la multiplication par $g$ est isom\'etrique, une approche directe \`a partir des isomorphismes 
  \[g^{-\frac{1}{p^{k+1}}}\sB^{\s}/\varpi_{\frac{1}{p}} \stackrel{x\mapsto x^p}{\cong}   g^{-\frac{1}{p^k}}\sB^{\s}/\varpi \]
   se heurte \`a l'\'eventualit\'e que $\lim \, (g^{-\frac{1}{p^k}}\sB^{\s}/\varpi) \to \lim \, (g^{-\frac{1}{p^k}}\sB^{\s}/\varpi_{\frac{1}{p}})$ ne soit pas surjectif\footnote{J'ignore si cette \'eventualit\'e se produit.}.  
  De m\^eme, une approche \`a partir de l'isomorphisme 
  \[ \sB^{\s} \cong  W(\sB^{\flat{\s}})\otimes_{W(\sK^{\flat{\s}}\langle T^{\e}\rangle)}\sK^{\s}\langle T^{\e}\rangle  \]
    bute sur le probl\`eme de la compatibilit\'e de $g^{\f} (-)$ au produit tensoriel (un quotient, dans ce cas).   
  
   \subsubsection{}\label{app} Pour clarifier la situation, nous introduisons la notion de {\it $\sK\langle T^{\e}\rangle$-alg\`ebre presque perfecto\"{\i}de}. Pla\c cons-nous dans le cadre $(\sK^{\s}\langle T^{\e}\rangle, T^{\e}\sK^{\ss} \sK^{\s}\langle T^{\e}\rangle)$, et reprenons les notations de \ref{recc}. 
   
   \begin{defn}\label{D1} Soit $\sB$ une $\sK\langle T^{\e} \rangle$-alg\`ebre de Banach uniforme. On dit que $\sB$ (ou $\sB^{\hat{a}}$) est {\emph{presque perfecto\"{\i}de}} si $\sB^{\hat{a}}$ est isomorphe \`a une $\sK\langle T^{\e} \rangle$-alg\`ebre perfecto\"{\i}de. 
  \end{defn}
  
 Notons $\sK\langle T^{\e} \rangle^{\hat a}\hbox{-}{\bf{Perf}}$ (\resp $\sK\langle T^{\e} \rangle^{\hat a}\hbox{-}{\bf{pPerf}}$) la sous-cat\'egorie pleine de $\sK\langle T^{\e} \rangle\hbox{-}{\bf{uBan}}$ dont les objets sont perfecto\"{\i}des (\resp presque perfecto\"{\i}des).

  \begin{lemma}\label{L15}  L'adjoint \`a droite $(\;)^\natural$ du plongement $\sK\langle T^{\e} \rangle\hbox{-}{\bf{Perf}}\inj \sK\langle T^{\e} \rangle\hbox{-}{\bf{uBan}}$ induit un adjoint \`a droite du plongement $\sK\langle T^{\e} \rangle^{\hat a}\hbox{-}{\bf{Perf}}\inj \sK\langle T^{\e} \rangle^{\hat a}\hbox{-}{\bf{uBan}}$, encore not\'e $\natural$. 
    
  Le plongement de $\sK\langle T^{\e} \rangle^{\hat a}\hbox{-}{\bf{Perf}}\inj \sK\langle T^{\e} \rangle^{\hat a}\hbox{-}{\bf{pPerf}}$ 
   est une \'equivalence, de quasi-inverse induit par $\natural$. \end{lemma}  

\begin{proof} Consid\'erons le diagramme commutatif 
  \[ \begin{CD}  \sK\langle T^{\e} \rangle\hbox{-}{\bf{Perf}}  @> i >>   \sK\langle T^{\e} \rangle\hbox{-}{\bf{Alg}}  \\@VV(\,)^{\hat a} V @V(\,)^{\hat a} VV     \\\ (\sK\langle T^{\e} \rangle^{\hat a}\hbox{-}{\bf{Perf}}   @> j >>   \sK\langle T^{\e} \rangle^{\hat a}\hbox{-}{\bf{Alg}}   . \end{CD}\] Alors $(\,)^{\hat a} \circ \natural \circ (\,)_{\hat\ast}$ est adjoint \`a droite de $(\,)^{\hat a} \circ i \circ (\,)_{\hat{!!}}= j$, et $(\,)^{\hat a} \circ \natural = ((\,)^{\hat a} \circ \natural \circ (\,)_{\hat\ast}) \circ (\,)^{\hat a}$. 
  
  La seconde assertion en d\'ecoule.  
\end{proof} 
  
 \begin{sorite}\label{S3} Soit $\sB$ une $\sK\langle T^{\e} \rangle$-alg\`ebre de Banach uniforme.
  Les conditions suivantes sont \'equivalentes:
\begin{enumerate}
\item $\sB$ est presque perfecto\"{\i}de,
\item $F$ est presque surjectif sur $\sB^{\s}/\varpi$, 
 \item  $\sB^{\natural {\s}}\inj \sB^{\s}$ est un presque-isomorphisme,
 \item  $\sB^\natural \inj \sB$ est un presque-isomorphisme,
    \item il existe une $ \sK\langle T^{\e} \rangle$-alg\`ebre perfecto\"{\i}de $\,\sB'\,$ et un morphisme $\,\sB'\to \sB\,$ de $ \sK\langle T^{\e} \rangle$-alg\`ebres de Banach qui est un presque-isomorphisme,
      \item $(\sB^{\hat a})_{\hat{!!}}$ est perfecto\"{\i}de. 
  \end{enumerate} \end{sorite}

 \begin{proof}   Les implications $ (1) \Rightarrow (2),\;   (3)\Rightarrow  (4) \Rightarrow (5)$ et $(6) \Rightarrow (1)$ sont claires. 
 
 \smallskip\noindent $(2)\Rightarrow (3)$: il suffit de prouver que  $\sB^\natural \to \sB$ est presque surjectif sous $(2)$  (\cf prop. \ref{P6}).  D'apr\`es le lemme de Nakayama (\cf \S \ref{MLN}),  il suffit de montrer que $\sB^{\natural \sm o}/\varpi \cong \sB^{\flat\sm o}/\varpi^\flat \to \sB^{\s}/\varpi$ est presque surjectif.  Posons $ \mathfrak L :=  \sB^{\s}/\varpi $, de sorte que $\displaystyle \sB^{\flat\sm o}= \lim_F \mathfrak L$. Il suffit de montrer que $\displaystyle \lim_F \mathfrak L \to \mathfrak L$ est presque surjectif. Or $(2)$ se r\'e\'ecrit comme la presque surjectivit\'e du morphisme  $\;    \mathfrak L \otimes_{\sK^{\s}\langle T^{\e}\rangle/\varpi  , F}  \sK^{\s}\langle T^{\e}\rangle/\varpi   \stackrel{\varphi}{\to}  \mathfrak L  $ induit par Frobenius. Cette derni\`ere entra\^{\i}ne la surjectivit\'e de $F: \sK^{\ss}T^{\e} \mathfrak L \to\sK^{\ss}T^{\e} \mathfrak L $. Par le lemme de Mittag-Leffler pour les $\F_p$-espaces vectoriels , on en d\'eduit la surjectivit\'e de $\displaystyle \sK^{\ss}T^{\e} \lim \mathfrak L \to\sK^{\ss}T^{\e} \mathfrak L $, d'o\`u la presque surjectivit\'e de $\displaystyle \lim_F \mathfrak L \to \mathfrak L$.

  \smallskip\noindent $(5)\Rightarrow (1)$: d'apr\`es le lemme \ref{L5}, $(6)$ implique que $(\sB')^{\s} \to \sB^{\s}$ est un presque-isomorphisme, d'o\`u $(1)$.
  
   \smallskip\noindent $(3)\Rightarrow (6)$:  quitte \`a remplacer $\sB$ par $\sB^\natural$, $(4)$ permet de supposer $\sB$ perfecto\"{\i}de pour prouver $(6)$.   D'apr\`es le lemme pr\'ec\'edent, on a $[((\sB^{\hat a})_{\hat{!!}} )^\natural]^{\hat a} = [((\sB^{\hat a})_{\hat{!!}} )^{\hat a}]^\natural = (\sB^{\hat a})^\natural$, et le compos\'e  \[[[((\sB^{\hat a})_{\hat{!!}} )^\natural]^{\hat a}]_{\hat{!!}}  \to ((\sB^{\hat a})_{\hat{!!}} )^\natural  \to (\sB^{\hat a})_{\hat{!!}} . \] est l'identit\'e. Donc le morphisme injectif $((\sB^{\hat a})_{\hat{!!}} )^\natural  \to (\sB^{\hat a})_{\hat{!!}}$ est un isomorphisme.
   \end{proof}

    \subsubsection{Remarques}   $(1)$ Supposons $g$ non-diviseur de z\'ero. Dans la classe d'isomorphisme de $\sB^{\hat a} $, il y a une alg\`ebre perfecto\"{\i}de initiale $(\sB^{\hat a})_{\hat{!!}}$ et une alg\`ebre perfecto\"{\i}de finale $(g^{\f}\sB^{\s}[\frac{1}{\varpi}])^\natural  $ (qui serait $(\sB^{\hat a})_{\hat{\ast}}$ si la question ci-dessus avait une r\'eponse positive).

\smallskip\noindent $(2)$ 
Perfecto\"{\i}de implique presque perfecto\"{\i}de mais {\it pas vice-versa}, m\^eme si $g$ est non-diviseur de z\'ero: la sous-$\sK\langle T_1^{\e}\rangle$-alg\`ebre $\sB =     \sK + (T_1^{\e} \sK\langle  T_1^{\e} , T_2^{\e} \rangle)^- + (T_2 \, \sK\langle  T_1^{\e} , T_2^{\e} \rangle)^-$ de $ \sK\langle  T_1^{\e} , T_2^{\e} \rangle$  est presque perfecto\"{\i}de mais pas perfecto\"{\i}de, puisque son image modulo $(T_1^{\e})$ n'est pas parfaite.
 
\smallskip\noindent $(3)$  En d\'epit du point $(5)$ du lemme (et du lemme \ref{L5}, qui traite des morphismes et non des {``presque-morphismes"}), il n'est pas clair qu'une $\sK\langle T^{\e} \rangle$-alg\`ebre de Banach uniforme presque isomorphe \`a une $\sK\langle T^{\e} \rangle$-alg\`ebre perfecto\"{\i}de soit presque perfecto\"{\i}de.

 \begin{lemma}\label{L16'} Le foncteur $\sA \mapsto \sA^{\s}$ induit une \'equivalence de la cat\'egorie des $\sK\langle T^{\e} \rangle^{\hat a}$-alg\`ebres presque perfecto\"{\i}des vers celle des $\sK\langle T^{\e} \rangle^{\s a}$-alg\`ebres ${\mathfrak{A}^a}$ plates sur $\sK^{\s}$, compl\`etes et telles que $\,{\mathfrak{A}^a}/\varpi_{\frac{1}{p}}\stackrel{x\mapsto x^p}{\to}\,{\mathfrak{A}^a}/\varpi$ soit {bijectif}.  
      \end{lemma}  
      
  La preuve est analogue \`a celle du lemme \ref{L14'}. \qed

 \begin{prop}\label{P10}    Si $\sA$ est presque perfecto\"{\i}de, on a une \'equivalence de cat\'egories 
 
 \medskip\noindent {$\{$ $\sA^{\hat a}$-alg\`ebres presque perfecto\"{\i}des presque \'etales finies$\}  \stackrel{\sim}{\to}   \{\sA^{\s a}$-alg\`ebres \'etales  finies$\}  $ }  

\medskip  donn\'ee par $(\;)^{\s a}$.
   \end{prop}  
 
 La preuve est analogue \`a celle de la prop. \ref{P8'}. \qed

 \medskip
  \subsection{Epimorphismes (et localisation).}

         \subsubsection{}\label{r8} Comme le plongement $\sK\hbox{-}\bf Perf \to  \sK\hbox{-}\bf uBan$ est fid\`ele et admet un adjoint \`a droite, il pr\'eserve et refl\`ete les \'epimorphismes. 
         
         Soit $\sB$ une $\sK$-alg\`ebre perfecto\"{\i}de et  $\phi: \sB \to \sC$  un \'epimorphisme extr\'emal de $\sK$-alg\`ebres de Banach uniformes (\cf \ref{id}). Puisque $\sK\hbox{-}{\bf{Perf}}$ est une {sous-cat\'egorie cor\'eflexive} de $\sK\hbox{-}{\bf{uBan}}$ dont les co\"unit\'es d'adjonction sont des monomorphismes, elle est stable par quotient extr\'emal  (\cf \cite[prop. 4]{HS}), donc $\sB$ est perfecto\"{\i}de (mais vu comme \'epimorphisme de $\sK$-alg\`ebres affino\"{\i}des, $\phi$ n'est peut-\^etre pas extr\'emal).
      
      \smallskip Par exemple, consid\'erons une $\sK$-alg\`ebre perfecto\"{\i}de spectrale $\sA$  et un \'el\'ement $g\in \sA^{\s }$, et formons l'alg\`ebre spectrale $\sA\langle g^{\e}\rangle$, dont $\sA$ est une sous-alg\`ebre norm\'ee (ex. \ref{E9} (2)). Alors $ \sA \hat\otimes_\sK \sK\langle T^{\e}\rangle$ est une $\sK$-alg\`ebre perfecto\"{\i}de spectrale (prop. \ref{P7}) et le morphisme canonique $\phi: \sA \hat\otimes_\sK \sK\langle T^{\e}\rangle \to \sA\langle g^{\e}\rangle$ est un \'epimorphisme extr\'emal (\cf formules \eqref{eqX} et \eqref{eu}), donc $ \sA\langle g^{\e}\rangle$ est perfecto\"{\i}de.
       Notant $g^\flat \in \sA\langle g^{\e}\rangle^\flat$ le syst\`eme inverse des $g^{\frac{1}{p^k}}$ modulo $p$, on a $g= {\footnotesize{\#}} g^\flat$.

     \subsubsection{\it Remarque.}\label{r8-} Signalons quelques r\'esultats relatifs aux \'epimorphismes que nous n'utiliserons pas dans la suite. Si $\phi: \sB \to \sC$ un morphisme de $\sK$-alg\`ebres de Banach uniformes d'image dense, et si $\sB$ est perfecto\"{\i}de, il en est de m\^eme de $\sC\,$ (\cf  \cite[th. 3.6.17 (b)]{KL}).     
  De surcro\^{\i}t, si $\phi $ est surjectif et $\sB$ spectrale,  la norme spectrale de $\sC$ est la norme quotient et l'homomorphisme induit sur les boules unit\'e est presque surjectif
      \cite[prop. 3.6.9 (c)]{KL}. 
      
      On en d\'eduit que $\phi$ est surjectif si et seulement si $\phi^\flat$ l'est; en effet, la condition se traduit par la presque surjectivit\'e de $\phi^{\s} $ mod. $\varpi$, qui s'identifie \`a $\phi^{\flat\s}$ mod. $\varpi^\flat$.

\smallskip Par ailleurs, noter que $\ker \phi^\flat$ est un id\'eal ferm\'e radiciel, donc le quotient $\sA^\flat/\ker \phi^\flat$ est perfecto\"{\i}de \cite[3.1.6 d]{KL}, et on a une factorisation canonique de $\phi$ dans $\sK\hbox{-}\bf Perf$ en \[\sA \stackrel{\phi_1}{\to} (\sA^\flat/ \ker \phi^\flat)^\sharp \stackrel{\phi_2}{\to} (\sA/\ker \phi)^u \stackrel{\phi_3}{\to} \sB\]  o\`u $\phi_1$ est surjectif, $\phi_2$ est \`a la fois un \'epimorphisme extr\'emal et un monomorphisme, et $\phi_3\circ \phi_2$ est un monomorphisme.

      \subsubsection{}\label{locp} 
       Passons aux localisations affino\"{\i}des (\cf \S \ref{loc}), et 
      commen\c cons par examiner l'exemple standard de l'alg\`ebre perfecto\"{\i}de spectrale $\sK\langle T^{\e}\rangle $.  Si $\car \sK = 0$, on prend $\varpi$ de la forme ${\footnotesize{\#}}(\varpi^\flat)$,  et on pose $  \varpi^{\frac{1}{p^i}} := {\footnotesize{\#}}((\varpi^\flat)^{\frac{1}{p^i}} )$; notons que le bascul\'e de  $\sK\langle T^{\e}\rangle$ est $\sK^\flat\langle T^{\e}\rangle$. 

  Pour tout $j$ et tout $m$,   $\sK\langle T^{\frac{1}{p^m}}\rangle\{ {\frac{\varpi^j}{T}}\}^{\s}  $ est le 
        compl\'et\'e $\sK^{\s}\langle T^{\frac{1}{p^m}}, ( {\frac{\varpi^j}{T}})^{\frac{1}{p^m}}\rangle $ de $\sK^{\s}[ T^{\frac{1}{p^m}},  ({\frac{\varpi^j}{T}})^{\frac{1}{p^m}}]:= \sK^{\s}[ T^{\frac{1}{p^m}},  U^{\frac{1}{p^m}}]/( T^{\frac{1}{p^k}} U^{\frac{1}{p^k}} - \varpi^{\frac{j}{p^k}})_{k\leq m}$ 
 tandis que $\sK\langle T^{\frac{1}{p^m}}\rangle\{ {\frac{\varpi^j}{T}}\}_{\leq 1}  $ est le 
        compl\'et\'e $\sK^{\s}\langle T^{\frac{1}{p^m}},  {\frac{\varpi^j}{T}}\rangle $ de $\sK^{\s}[ T^{\frac{1}{p^m}},  {\frac{\varpi^j}{T}}]:=  \sK^{\s}[ T^{\frac{1}{p^m}},  U]/(TU -\varpi^j  )$.  
        Passant \`a la colimite compl\'et\'ee (pour $m\to \infty$), on obtient  
       \begin{equation}\label{e24} \sK\langle T^{\e}\rangle\{{\frac{\varpi^j}{T}}\}^{\s} = \sK^{\s}\langle T^{\e},  ({\frac{\varpi^j}{T}})^{\e}\rangle :=  \widehat{ \sK^{\s}[ T^{\e},  U^{\e}]/( T^{\frac{1}{p^k}} U^{\frac{1}{p^k}} - \varpi^{\frac{j}{p^k}})_k },\end{equation}  
      alors que la boule unit\'e est 
        \begin{equation}\label{e25}  {\sK\langle T^{\e}\rangle\{{\frac{\varpi^j}{T}}\}}_{\leq 1} =  \sK^{\s} \langle T^{\e},  {\frac{\varpi^j}{T}}\rangle = \sK^{\s} \langle T^{\e}\rangle \hat\otimes_{\sK^{\s}\langle T \rangle}  \sK^{\s}\langle   T, {\frac{\varpi^j}{T}} \rangle. \end{equation}  
       Ainsi   \begin{equation}\label{e25,}  {\sK\langle T^{\e}\rangle\{{\frac{\varpi^j}{T}}\}}  =   \sK\langle T^{\e}\rangle \hat\otimes_{\sK\langle T \rangle}  \sK\langle   T, {\frac{\varpi^j}{T}} \rangle =  \sK\langle T^{\e}\rangle \hat\otimes_{\sK\langle T \rangle}  \sK\langle   T \rangle\{{\frac{\varpi^j}{T}}\} \end{equation} 
      est perfecto\"{\i}de d'apr\`es \eqref{e24}, en particulier uniforme, mais non spectrale en g\'en\'eral d'apr\`es \eqref{e25}. 
        Pour toute $\sK\langle T^{\e}\rangle$-alg\`ebre de Banach uniforme $\sA$, on a alors, en notant $g$ l'image de $T$ et en combinant les formules  \eqref{e11} et \eqref{e25,}:
        
         \begin{equation}\label{e25'}    \sA\{ \frac{\varpi^j}{g} \}   \cong   \sA\hat\otimes_{  \sK\langle T^{\e} \rangle}  \sK\langle T^{\e} \rangle \{\frac{\varpi^j}{T}\}   .  \end{equation}

                \begin{prop}\label{P11}  Soient $\sA$ une $\sK\langle T^{\e}\rangle$-alg\`ebre perfecto\"{\i}de munie de sa norme spectrale, et $g= T\cdot 1$.
                 Alors pour tout $j\in \N$, $\sA\{{\frac{\varpi^j}{T}}\}$ est perfecto\"{\i}de et le morphisme canonique 
        \begin{equation}\label{e26}(\sA^{\s} \hat\otimes_{\sK^{\s}\langle T^{\e}\rangle}   {\sK^{\s}\langle T^{\e}, ({\frac{\varpi^j}{T}})^{\e}\rangle})_{\ast} \to     \sA\{{\frac{\varpi^j}{g}}\}^{\s} \end{equation} 
         est un isomorphisme. 
         En particulier, $\sA\hat\otimes_{\sK\langle T^{\e}\rangle}   {\sK\langle T^{\e}\rangle\{\frac{\varpi^j}{g}}\}^u$ est spectrale. 
         
         En outre, si $\car \sK =0$, $\sA\{{\frac{\varpi^j}{g}}\}^\flat$ s'identifie \`a $\sA^\flat\{{\frac{(\varpi^\flat)^j}{g^\flat}}\}$ o\`u $g^\flat$ est l'image de $T$ dans $\sA^\flat$ (de sorte que ${\footnotesize{\#}} g^\flat = g$).
           \end{prop} 
       
      \begin{proof} Compte tenu des formules \eqref{e24}\eqref{e25'}, c'est un corollaire de la prop. \ref{P7}.
      \end{proof}

 \subsubsection{Remarque}\label{r9}  
 Le m\^eme argument montre que $\sA\{\frac{f_1, \ldots, f_n}{g}\}$ est perfecto\"{\i}de si $f_1, \ldots, f_n, g$ admettent des suites compatibles de racines $p^{m}$-i\`emes. C'est encore vrai sans supposer cette condition: 
  Scholze le d\'eduit du r\'esultat pr\'ec\'edent - qu'il prouve \`a nouveaux frais \cite[def. 2.13, th. 6.3 (ii), lem. 6.4 (iii)]{S1}) plut\^ot que de le d\'eduire du r\'esultat sur les $\hat\otimes$ - , gr\^ace \`a un lemme d'approximation de $f_i, g\in \sA^{\s}$ par des \'el\'ements du type ${\footnotesize{\#}} {\tilde f_i}^\flat, {\footnotesize{\#}} {\tilde g}^\flat$ \cite[cor. 6.7 (i)]{S1}.   
 
   \medskip\subsection{Produits (et transform\'ee de Gelfand).}\label{prodGp} 

         \subsubsection{} Les produits uniformes  
          d'alg\`ebres perfecto\"{\i}des sont perfecto\"{\i}des, du fait que 
     $( \prod  \sA^{\alpha{\s}})/\varpi \cong   \prod  ( \sA^{\alpha{\s}}/\varpi ).$
       
      \begin{prop}\label{P12} Si $\sA\neq 0$ est perfecto\"{\i}de, il en est de m\^eme de sa transform\'ee de Gelfand $\Gamma(\sA)$.           \end{prop}
 
      \begin{proof} Comme $\Gamma(\sA)$ est produit uniforme de corps $\mathcal H(x)$, il suffit de voir que ceux-ci sont perfecto\"{\i}des. Or $\mathcal H(x)$ est colimite uniforme de localisations rationnelles de $\sA$, qui sont perfecto\"{\i}des, \cf rem. \ref{r9} et    \cite[2.4.17]{KL}; donc $\mathcal H(x)$ est perfecto\"{\i}de (\cf \S \ref{lcolp}).  
 \end{proof}

 \subsubsection{Remarque.}\label{Gelbasc} On peut montrer que si $\sA$ est perfecto\"{\i}de, $\Gamma(\sA^\flat)\cong \Gamma(\sA)^\flat$, comme cons\'equence de la commutation des localisations rationnelles au basculement et de ce que  $\mathcal M(\sA) \cong \mathcal M(\sA^\flat)$ \cite[th. 3.3.7]{KL}. 
 J'ignore si $\Gamma(\sA^\flat)\cong \Gamma(\sA)^\flat$ vaut encore sans supposer $\sA$ perfecto\"{\i}de.

   \subsection{Limites.}\label{lp}   
     
    \subsubsection{}   La cat\'egorie des $\sK$-alg\`ebres perfecto\"{\i}des est {\it compl\`ete}, \ie admet toutes les (petites) limites. Cela d\'ecoule de ce que $\sK\hbox{-}\bf uBan$ est compl\`ete et $\sK\hbox{-}{\bf Perf}$ en est une sous-cat\'egorie cor\'eflexive, \ie le plongement $\sK\hbox{-}{\bf Perf} \inj \sK\hbox{-}\bf uBan$ admet un adjoint \`a droite (\cf prop. \ref{P4}, \ref{P6}).
    
  Si $\car \sK = p$, ce plongement admet aussi un adjoint \`a gauche (\cf prop. \ref{P5}), donc refl\`ete les limites: les limites perfecto\"{\i}des sont les ${\rm{ulim}}$.  Comme $\flat: \sK\hbox{-}\bf uBan \to \sK^\flat\hbox{-}{\bf Perf}$ admet un adjoint \`a gauche (prop. \ref{P6}), il commute donc aux ${\rm{ulim}}$.
  
   Si $\car \sK = 0$, on les obtient par basculement \`a partir du cas de caract\'eristique $p$: les limites perfecto\"{\i}des sont les ${\rm{ulim}}^\natural$.     
     
 \subsubsection{} En caract\'eristique $0$, {\it limite perfecto\"{\i}de et limite uniforme ne co\"{\i}ncident pas}.
 
  Ce ph\'enom\`ene se pr\'esente d\'ej\`a dans le cas des limites finies: {\it l'intersection de deux sous-alg\`ebres perfecto\"{\i}des d'une alg\`ebre perfecto\"{\i}de n'est pas n\'ecessairement perfecto\"{\i}de}, comme on le voit dans l'ex.  \ref{E13}, o\`u  $\hat A_\infty[\sqrt{g}] = \widehat{Q( A_\infty)}(\sqrt{g}) \cap \hat A_\infty\langle g^{\frac{1}{2^\infty}}\rangle \subset \widehat{Q({ A_\infty \langle g^{\frac{1}{2^\infty}}\rangle})}$  
 (\cf \ref{r8}, rem. \ref{r5} $(1)$ et ex. \ref{E18}). 
 
      Dans la direction positive, on a:

    \begin{prop}\label{P13}  L'anneau des invariants $\sB^G$ d'une $\sK$-alg\`ebre perfecto\"{\i}de $\sB$ sous un groupe \emph{fini} d'isom\'etries est une $\sK$-alg\`ebre perfecto\"{\i}de, et identifiant $G$ \`a un groupe d'isom\'etries de $\sB^\flat$ par basculement, on a $\, \sB^{\flat G}= \sB^{G\flat}$.
     \end{prop}  
     
     \begin{proof} D'apr\`es la discussion qui pr\'ec\`ede, il suffit de montrer la premi\`ere assertion.  Notons $(\sB^{\flat  G} )^\natural$ la $\sK$-alg\`ebre perfecto\"{\i}de associ\'ee \`a $\sB^{\flat  G}  $ par basculement.
    On a    $\; \sB  \cong  W(\sB^{\flat {\s}} )[\frac{1}{p}]\otimes_{W(\sK^{\flat {\s}} ) [\frac{1}{p}]} \sK, $ d'o\`u 
    \centerline{$\; \sB^G  \cong (W(\sB^{\flat {\s}} )[\frac{1}{p}]\otimes_{W(\sK^{\flat {\s}} )[\frac{1}{p}] } \sK)^G = (W(\sB^{\flat {\s}}) [\frac{1}{p}])^G \otimes_{W(\sK^{\flat {\s}} )[\frac{1}{p}] } \sK$} 
    \centerline{$= (W(\sB^{\flat {\s}})^G [\frac{1}{p}]) \otimes_{W(\sK^{\flat {\s}} )[\frac{1}{p}] } \sK =  W( \sB^{\flat {\s} G})[\frac{1}{p}]\otimes_{W(\sK^{\flat {\s}} )[\frac{1}{p}] } \sK \cong (\sB^{\flat  G} )^\natural ,  $}
     la premi\`ere \'egalit\'e s'obtenant en identifiant $G$-invariants et image du projecteur de Reynolds $ \, b\mapsto \frac{1}{\vert G\vert}$ \begin{small}$\displaystyle\sum_{\gamma\in G}$\end{small} $\, \gamma(b)$. 
  \end{proof} 
   
\subsubsection{Remarque}\label{r10} Ceci ne s'\'etend pas aux groupes infinis. Reprenons l'exemple \ref{E11} $(2)$ avec $n=1$: l'action galoisienne de $\Z_p$ sur $T_1^\e\subset  A_\infty$ s'\'etend par continuit\'e \`a l'alg\`ebre perfecto\"{\i}de  $\hat A_\infty$, mais l'alg\`ebre $A_0$ de ses invariants n'est pas perfecto\"{\i}de.  
 
         \subsubsection{}\label{lplu}  Supposons  $\car \sK = 0\,$, notons encore $\sK^\flat$ son bascul\'e, et prenons  $\varpi$ de la forme ${\footnotesize{\#}}(\varpi^\flat)$.  
       On a vu que la limite perfecto\"{\i}de d'un syst\`eme projectif $(\sA^{\alpha\flat})$ de $\sK^\flat$-alg\`ebres perfecto\"{\i}des co\"{\i}ncide avec  $\sA^\flat :=   {\rm{{ulim}}}\,\sA^{\alpha \flat} $. 
       
    Dans le cas d'un syst\`eme projectif $(\sA^{\alpha})$ de $\sK$-alg\`ebres perfecto\"{\i}des, de limite uniforme $\sA$, la limite perfecto\"{\i}de n'est autre, on l'a vu, que l'alg\`ebre $\sA^\natural$, dont le bascul\'e est d\'etermin\'e par
    \begin{equation}\label{e26'}\,\displaystyle (\sA^\natural)^{\flat\sm o} = \sA^{\flat\sm o}   = \lim_{F} \sA^{\s}   = \lim_{F,\alpha} \sA^{\alpha \sm o} = \lim_\alpha   \sA^{\alpha\flat  \sm o} = \lim_{F,\alpha}  ( \sA^{\alpha \sm o}/\varpi) .\end{equation}
      L'exemple suivant montre que $\sA^\natural \to \sA$ n'est pas un isomorphisme en g\'en\'eral.  
    
  \subsubsection{Exemple prophylactique.}\label{E17} 
 Reprenons l'ex. \ref{E13}.
  Nous verrons au th. \ref{T5} que la limite uniforme des         $\hat A_\infty\langle g^{\frac{1}{2^\infty}}\rangle \langle\frac{\varpi^j}{g} \rangle$ est   $g^{-\frac{1}{2^\infty}}\hat A_\infty\langle g^{\frac{1}{2^\infty}}\rangle$ qui n'est autre que la fermeture compl\`etement int\'egrale de 
        $\hat A_\infty\langle g^{\frac{1}{2^\infty}} \rangle$ dans  $\hat A_\infty\langle g^{\frac{1}{2^\infty}}\rangle[\frac{1}{g}]$. Passant aux $2\Z_2$-invariants, la limite uniforme $\sB $ des     $\hat A_\infty\langle g^{\frac{1}{2^\infty}}\rangle \langle\frac{\varpi^j}{g} \rangle^{2\Z_2}$ n'est autre que la fermeture compl\`etement int\'egrale de 
        $\hat A_\infty[\sqrt g] = \hat A_\infty\langle g^{\frac{1}{2^\infty}}\rangle)^{2\Z_2}$ dans  $\hat A_\infty[\sqrt g, \frac{1}{g}]$ (\cf ex. \ref{E18}). Or $\hat A_\infty$ est compl\`etement int\'egralement clos ({\it ibid.}), et comme la fermeture int\'egrale d'un anneau compl\`etement int\'egralement clos dans toute extension alg\'ebrique de son corps de fractions est compl\`etement int\'egralement close (\cite[ch. V, \S 1, ex. 14]{B2}), on a $\sB = \hat A_\infty[\sqrt g]$, qui n'est pas perfecto\"{\i}de.

  \medskip\subsection{Colimites.}\label{lcolp}      
       
     \subsubsection{}  La cat\'egorie des $\sK$-alg\`ebres perfecto\"{\i}des admet des colimites filtrantes:
    ce sont les colimites {uniforme}s. En effet, si l'endomorphisme de Frobenius est surjectif sur les  $\sB_\alpha^{\s}/\varpi$, il l'est sur 
    \begin{equation}\label{eco} \,{\rm{colim}}\,(\sB_\alpha^{\s}/\varpi) \cong ({\rm{colim}}\,\sB_\alpha^{\s})/\varpi \cong   ({\rm{{ucolim}}}\,\sB_\alpha)^{\s}/\varpi\end{equation}  
    
  Comme elle admet aussi des sommes amalgam\'ees $\hat\otimes^u= \hat\otimes $, elle est {\it cocompl\`ete}, les colimites se calculant comme dans la cat\'egorie des $\sK$-alg\`ebres de Banach uniformes. Les colimites d'alg\`ebres perfecto\"{\i}des commutent \`a l'\'equivalence de basculement.

    \subsubsection{Exemples}\label{E16} $(1)$ Le compl\'et\'e d'une extension alg\'ebrique de $\sK$ (qui est parfait) s'\'ecrit  $\hat \sK_\infty = {\rm{{ucolim}}}_i\,  \sK_i$ o\`u $\sK_i$ parcourt les extensions finies s\'eparables de $\sK$, qui sont des corps perfecto\"{\i}des d'apr\`es le cor. \ref{C2} ou le th. \ref{T4}; donc $\hat \sK_\infty$ est un corps perfecto\"{\i}de.
    
 De m\^eme,  
   $\hat \sK_\infty \hat\otimes_{\hat \sK_0}^u \hat \sK_\infty = {\rm{{ucolim}}}_i\,  (\sK_i  \otimes_{\sK_0} \sK_i)$ est une $\sK$-alg\`ebre perfecto\"{\i}de.

    \smallskip\noindent $(2)$  Le coproduit d'une famille $(\sB_\alpha)$ de $\sK$-alg\`ebres perfecto\"{\i}des est repr\'esent\'e  par un produit  tensoriel  compl\'et\'e  infini  $\displaystyle \hat\otimes_\sK \sB_\alpha= \hat\otimes_\alpha\, \sB_\alpha$. Toute colimite des $(\sB_\alpha)$ est quotient de $\hat\otimes_\alpha \sB_\alpha$.  Pour toute $\sK$-alg\`ebre perfecto\"{\i}de, le morphisme $\hat\otimes_{\alpha\in \sB^{\flat\sm o}}  \,\sK\langle T_\alpha^{\e}\rangle \to \sB,   \; T_\alpha \mapsto {\footnotesize{\#}} \alpha,$  a une image dense (\cf rem. \ref{r6} $(1)$).

\bigskip

 \newpage  \section{Analyse perfecto\"{\i}de autour du ``th\'eor\`eme d'extension de Riemann". }\label{APER}

  \medskip  D'apr\`es le th\'eor\`eme d'extension de Riemann, les fonctions analytiques born\'ees sur un polydisque complexe priv\'e d'un sous-espace analytique - une hypersurface d'\'equation $\,g=0\,$ pour fixer les id\'ees - se prolongent en des fonctions analytiques sur tout le polydisque. L'analogue $p$-adique de ce r\'esultat est connu \cite{Ba}; avec les notations de l'ex. \ref{E11} $(2)$, on a  
    \begin{equation}\label{e27} \displaystyle{A  \, = \,\lim \, A_0\{ \frac{\varpi^{j}}{g}\}^{\s} .}\end{equation}
     Ce sont des versions perfecto\"{\i}des de cet \'enonc\'e que nous visons dans ce paragraphe.
       
  \medskip\subsection{Entr\'ee en mati\`ere.}\label{eem}  Donnons une courte preuve de l'\'enonc\'e plus faible
     \begin{equation}\label{e27'} \displaystyle{A  \, = \,\lim \, A_0\{ \frac{\varpi^{j}}{g}\}_{\leq 1}  }\end{equation}
  ($A = \sK_0^{\s}[[T_{\leq n}]], \; A_0 := A[\frac{1}{p}]$). 
     Quitte \`a faire un changement de variables $T_1, \ldots, T_n$, on peut supposer que $g$ est sous forme de Weierstrass en la variable $T_n$,
   \ie que $\,g\in A\setminus ({p}, T_1, \cdots T_{n-1})A\,$
   (\cite[VII, \S3, n. 7, lemme 3]{B2}). Le th\'eor\`eme de pr\'eparation de Weierstrass \cite[VII, \S3, n. 8, prop. 5]{B2} \'enonce que 
  $\;A/gA\cong  \bigoplus_{s=0}^{s=r-1}\,  \sK_0^{\s}[[T_{<n}]] \, T_n^s\,$ 
  (o\`u $r$ est la valuation $T_n$-adique de $g$ modulo $({p}, T_{<n})$),  
 d'o\`u l'on d\'eduit, d'apr\`es Nakayama, que
   \[A = \bigoplus \,  \sK_0^{\s}[[T_{<n},  g]] \, T_n^s.\]
 Appliquant $\otimes_{\sK\langle g\rangle_{\leq 1}}\sK\langle g, \frac{p^j}{g}\rangle_{\leq 1}$ et quotientant par la torsion $p$-primaire \'eventuelle, on en d\'eduit (en vertu de la formule \eqref{e11} et du point 2e) du sorite \ref{s1}) une d\'ecomposition 
   \[A_0\{ \frac{{p}^{j}}{g}\}_{\leq 1} \cong \bigoplus \,  \sK_0^{\s}[[T_{<n}, g]][\frac{1}{{p}}]_{\leq 1}\{ \frac{{p}^{j}}{g}\} \, T_n^s,\]
    compatible avec les fl\`eches de transition quand $j$ augmente.   
   On est donc ramen\'e \`a d\'emontrer l'\'enonc\'e initial dans le cas particulier $g=T_n$, qui se traite directement en observant que $A_0\{ \frac{{p}^{j}}{T_n}\}_{\leq 1}= A\langle \frac{{p}^{j}}{T_n}\rangle  $ et en consid\'erant le d\'eveloppement en puissances de $T_n $ \`a coefficients dans $ \sK_0^{\s}[[T_{<n}]]$.\qed

 On peut ensuite remplacer $A_0$ par $A_i$; mais passer \`a la colimite uniforme pour en d\'eduire l'\'enonc\'e analogue pour l'alg\`ebre perfecto\"{\i}de $\hat A_\infty$ ne va pas de soi; nous y reviendrons dans l'ex. \ref{E18}, o\`u nous montrerons aussi l'\'egalit\'e (plus forte) $A = \lim \, A_0\{ \frac{\varpi^{j}}{g}\}^{\s}$.

     \medskip\subsection{Limite uniforme de localis\'ees d'une alg\`ebre perfecto\"{\i}de.}\label{lul}     
 
 On compare ici une alg\`ebre perfecto\"{\i}de $\sA$ \`a la limite uniforme de ses localis\'ees $\sA\langle \frac{\varpi^j}{g}\rangle$.
 
 \subsubsection{}\label{lu} Soient $\sK$ un corps perfecto\"{\i}de de caract\'eristique r\'esiduelle $p$, et $\varpi$ un \'el\'ement de $\sK^{\ss}\setminus 0$ tel que $p\sK^{\s}\subset \varpi\sK^{\s}$. Si $\car \, \sK = 0$,  on prend $\varpi$ de la forme ${\footnotesize{\#}}(\varpi^\flat)$,  et on sp\'ecifie pour chaque $i$ une racine $p^i$-\`eme de $\varpi$ en posant $  \varpi^{\frac{1}{p^i}} := {\footnotesize{\#}}((\varpi^\flat)^{\frac{1}{p^i}} )$. 
  
  \noindent Etant donn\'e $r\in \N[\frac{1}{p}]$ et une $(\sK^{\s}/\varpi^r)[T^{\e}]$-alg\`ebre $R$,  consid\'erons les $(\sK^{\s}/\varpi^r)[T^{\e}]$-alg\`ebres 
 \begin{equation}\label{e28} R^{[j]} :=  R\otimes_{\sK^{\s}[T^{\e}] } \,\sK^{\s}[T^{\e}, (\frac{\varpi^{j}}{T})^{\e}], \; ({j\in \N}). \end{equation}
 Elles forment un syst\`eme projectif,  l'\'el\'ement $u_{[j+1]}^s:= 1\otimes (\frac{\varpi^{j+1}}{T})^{s}$ de $ R^{[j+1]}$ s'envoyant sur $\,\varpi^s u_{[j] }^s\in R^{[j]}$. On a $ T^s u_{[j] }^s = \varpi^{js} \cdot 1_{R^{[j]}}$.
    La proposition suivante est une version simplifi\'ee d'un r\'esultat remarquable de P. Scholze \cite[prop. II.3.1]{S2}.    
  
  \begin{prop}\label{P15} Le morphisme canonique $\;\displaystyle  R^a \to  \lim_j (R^{[j]})^a $  est un presque-isomorphisme dans  le cadre $(\sK^{\s}[T^{\e}], (\varpi T)^{\e}\sK^{\s}[T^{\e}])$.    \end{prop}

   Nous utiliserons librement le fait qu'une $\lim$ ou $\lim^1$  index\'ee par un ensemble ordonn\'e peut se calculer sur un sous-ensemble cofinal.
  
 \begin{proof}   Introduisons les sous-$R$-modules 
 \begin{equation}\label{e29} R^{[j]k} :=  \sum_{s\in \N[\frac{1}{p}] \cap [0, \frac{1}{p^k}]} \, R.u_{[j] }^s  \subset  R^{[j]}, \;\end{equation}  qui forment un double syst\`eme projectif de $R$-modules.  
   
   Comme $u_{[j'] }^s$ s'envoie sur $0$ dans $R^{[j]}$ si $\,s({j'- j})\geq r$, $ R^{[j']} \to  R^{[j]}$  se factorise \`a travers $R^{[j]k}$ d\`es que $ j'\geq j+ rp^k$. On en d\'eduit
\begin{equation}\label{e30} \displaystyle  {\lim_{j}}  R^{[j]} \cong {\lim_{j}}  R^{{[j]}k} \cong {\lim_{jk}}  R^{{[j]}k}.\end{equation}  
 
Comme $T^{ \frac{1}{p^k}}u_{[j] }^s\in S^j:= \Im (R\to R^{[j]})$ si $s\leq \frac{1}{p^k}$, on a $T^{ \frac{1}{p^k}} R^{[j]k} \subset S^j $, et le quotient est annul\'e par $T^{ \frac{1}{p^k}}$.  
 On en d\'eduit  (\cf \ref{MLN})
\begin{equation}\label{e32}  \, (S^j)^a \cong \displaystyle \lim_k (R^{[j]k} )^a.\end{equation}
    
  Par ailleurs, le noyau de $R\surj S^j \cong \Im(R \to  R[U_{[j]}^\e]/(T^{\frac{1}{p^k}}U_{[j] }^{\frac{1}{p^k}}- \varpi^{\frac{1}{p^k}} )_{k \in \N})$ est annul\'e par $T^{\frac{r}{j}}$ pour tout $j $ de la forme  $p^\ell$: en effet, si $a$ est dans ce noyau, et si on le voit comme \'el\'ement de $R[U_{[j]}^\e]$, il s'\'ecrit $(T^{\frac{1}{p^k}}U_{[j]}^{\frac{1}{p^k}} - \varpi^{\frac{j}{p^k}})\sum a_s U_{[j]}^s$ pour $k$ convenable (qu'on peut supposer $\geq \ell$ puisque $T^{\frac{1}{p^k}}U_{[j]}^{\frac{1}{p^k}}- \varpi^{\frac{1}{p^k}}$ divise $T^{\frac{1}{p^\ell}}U_{[j]}^{\frac{1}{p^\ell}}- \varpi^{\frac{1}{p^\ell}}$ si $k\geq \ell$). En comparant les coefficients, on obtient alors $T^{\frac{n}{p^k}}a = - \varpi^{\frac{j(n+1)}{p^k}}a_{\frac{n}{p^k}} $ pour tout $n\in \N$, d'o\`u le r\'esultat en prenant $n = r p^{k-\ell}$.  
   Compte tenu de  \eqref{e30}\eqref{e32}, on conclut que     $\displaystyle \,  R^a \cong  \lim_j (S^j)^a\cong \lim_{jk} (R^{[j]k})^a \cong \lim_j (R^{[j]})^a,  $  d'o\`u l'assertion.
    \end{proof}

\subsubsection{} Voici une variante perfecto\"{\i}de du th\'eor\`eme d'extension de Riemann. 
  
  \begin{thm}\label{T5} Soient $\sA$ une $\sK$-alg\`ebre (presque) perfecto\"{\i}de et $g$ un \'el\'ement de $ \sA^{\s}$ non diviseur de z\'ero. On suppose que $\sA$ contient une suite compatible de racines $p^m$-i\`emes de $g$. Alors 
      \begin{equation}\label{e37} \lim  {\sA}\{\frac{\varpi^j}{g} \}^{\s}\, = g^{\f}\sA^{\s},\end{equation} 
   qui est la fermeture compl\`etement int\'egrale de $\sA^{\s}$ dans $\sA[\frac{1}{g}].$ 
              \end{thm}

 \begin{proof} Le cas presque perfecto\"{\i}de se ram\`ene au cas perfecto\"{\i}de en rempla\c cant $\sA$ par $\sA^\natural$ (qui contient encore $g^{\e}$). Supposons donc $\sA$ perfecto\"{\i}de.  D'apr\`es la prop. \ref{P11}, $\sA$ perfecto\"{\i}de implique $ \sA\{ \frac{\varpi^j}{g}\}$ perfecto\"{\i}de, et en outre $ \sA^{\s}\langle (\frac{\varpi^j}{g})^{\e}\rangle       {\to}  \sA\{ \frac{\varpi^j}{g}\}^{\s} \,$ est un presque-isomorphisme dans le cadre $(\sK^{\s},  \sK^{\ss})$, donc aussi dans $(\sK^{\s}[T^{\e}], (\varpi T)^{\e}(\sK^{\s}[T^{\e}])$. 
   
\noindent Appliquant la prop. \ref{P15} avec $ R := \sA^{\s}/\varpi^r $ et $(R^{[j]} )^a \cong  (\sA\{ \frac{\varpi^j}{g}\}^{\s}/\varpi^r)^a$, on obtient  
\begin{equation}\label{eqmod}(\sA^{\s}/\varpi^r )^a \cong \lim  (\sA\{ \frac{\varpi^j}{g}\}^{\s}/\varpi^r)^a.\end{equation}
 En passant \`a la limite sur $r\to \infty$, puis en intervertissant les limites, cela donne
\[(\sA^{\s} )^a \cong \lim  (\sA\{ \frac{\varpi^j}{g}\}^{\s} )^a.\]
 Appliquant le foncteur $(\;)_\ast$ des presque-\'el\'ements, qui commute aux limites, et compte tenu de la formule \eqref{e6'} et du lemme \ref{L14}, on trouve  $  g^{\f}\sA^{\s} \cong \lim \sA\{ \frac{\varpi^j}{g}\}^{\s }.$

  \medskip Montrons ensuite que $ g^{\f} \sA     \subset  \sA^\ast_{\sA[\frac{1}{g}]}$ est une \'egalit\'e, c'est-\`a-dire la condition $(4)$ du sorite \ref{S2}:
   
  \begin{lemma}\label{uL} Pour toute $\sK\langle T^{\e}\rangle$-alg\`ebre perfecto\"{\i}de, les conditions \'equivalentes du sorite \ref{S2} sont satisfaites.
  \end{lemma}  
     
  On va \'etablir la condition $(1)$ de ce sorite:  $(a\in \sB^{\s}[\frac{1}{g}], a^p \in g^{\f} \sB^{\s}) \Rightarrow a \in g^{\f} \sB^{\s}$. Notons $a_j$ l'image de $a$ dans $\sB\{ \frac{\varpi^j}{g}  \}^u$. 
On a  $a\in \sB^{\s}[\frac{1}{g}]\Rightarrow \exists m\in \N, \forall j\in \N, g^m a_j \in \sB\{ \frac{\varpi^j}{g}  \}^{\s} $ (et multipliant par $(\frac{\varpi^j}{g})^m$), $\varpi^{jm}a_j\in   \sB\{ \frac{\varpi^j}{g}  \}^{\s} $. Si $a_j^p\in  \sB\{ \frac{\varpi^j}{g}  \}^{\s} $, on conclut du sorite \ref{s1} $(5e)$ que $a_j \in  \sB\{ \frac{\varpi^j}{g}  \}^{\s} $. D'o\`u, \`a la limite, $a\in g^{\f} \sB^{\s}$. \qed
  
   Cela termine la preuve du th\'eor\`eme. \end{proof}

 \subsubsection{Remarque}\label{r12}  Si la multiplication par $g$ est {\it isom\'etrique}, le th\'eor\`eme donne:    $\, {\rm{ulim}}\,   {\sA}\{\frac{\varpi^j}{g} \}\, = g^{\f}\sA$, qui est la fermeture compl\`etement int\'egrale de $\sA $ dans $\sA[\frac{1}{g}]$. En outre
 la preuve 
   se simplifie un peu par rapport \`a celle de \eqref{e37}: on n'a alors \`a invoquer la prop. \ref{P15} que dans le cas o\`u $R\to R^{[j]}$ est injectif, \ie $R= S^j$. 

  \begin{cor}\label{C3}  \begin{enumerate}
  \item Munissons $\sA$ de sa norme spectrale.  Alors 
    $(g^{\f}\sA^{\s}[\frac{1}{\varpi}])^\natural$ est la plus grande $\sA$-alg\`ebre perfecto\"{\i}de spectrale contenue dans $\sA[\frac{1}{g}]$.  
    \item On a $(g^{\f}\sA^{\s})^{\flat} \cong g^{\flat \f}\sA^{\flat \s} ,$ donc aussi $(g^{\f}\sA^{\s}[\frac{1}{\varpi}])^{\natural }\cong (g^{\flat \f}\sA^{\flat \s}[\frac{1}{\varpi^\flat}])^\sharp$.
    \end{enumerate} 
  \end{cor}

  \begin{proof} $(1)$ Soit $\sA'$ une $\sA$-alg\`ebre perfecto\"{\i}de spectrale ($\sA\to \sA'$ est donc suppos\'ee continue) contenue dans $\sA[\frac{1}{g}]$ (en tant que $\sA$-alg\`ebre {``abstraite"}, non topologis\'ee). D'apr\`es le lemme \ref{L11},  les morphismes $\sA\{ \frac{\varpi^j}{g} \}\to \sA'\{ \frac{\varpi^j}{g} \}$ sont des isomorphismes d'alg\`ebres de Banach. D'apr\`es le point $(2)$ du th\'eor\`eme pr\'ec\'edent, on a donc, en passant \`a la limite uniforme, $g^{\f}\sA^{\s}[\frac{1}{\varpi}] = g^{\f}\sA^{'\s}[\frac{1}{\varpi}]$. Ainsi $(g^{\f}\sA^{\s}[\frac{1}{\varpi}])^\natural$ contient $\sA^{'\natural} = \sA'$.  
  
\smallskip \noindent  $(2)$ On a  $(g^{\f}\sA^{\s})^{\flat} = (\lim \sA\{ \frac{\varpi^j}{g} \}^{\s})^\flat =  \lim \sA^\flat\{ \frac{\varpi^{\flat j}}{g^\flat} \}^{\s} = g^{\flat \f}\sA^{\flat \s}$ en vertu du th. \ref{T5} et de la prop. \ref{P11}.
  \end{proof} 
 
 \subsubsection{\it Remarque.} On peut se demander si, sans supposer que $\sA$ contienne $g^{\e}$, il demeure vrai que $\lim \sA\{ \frac{\varpi^j}{g} \}^{\s}$ est la fermeture compl\`etement int\'egrale de $\sA^{\s}$ dans $\sA[\frac{1}{g}]$. Supposons $\car  \sK = 0$ comme il est loisible,  et posons $\sK' := \widehat{\sK(\zeta_{p^\infty})}, \, \sA' := \sK'\hat\otimes^u_\sK \sA\langle g^{\e}\rangle$. Le groupe des $\sK$-automorphismes continus de $\sK'$ est un sous-groupe ferm\'e $\Gamma$ de $\Z_p^\times$. En utilisant par exemple la remarque \ref{r.2},
     on voit que $\sA$ se plonge isom\'etriquement dans $\sA'$. Or $\sK'$ est un corps perfecto\"{\i}de (prop. \ref{P3}) et $\sA' $ est une $\sK'$-alg\`ebre perfecto\"{\i}de (\cf \ref{r8}).      
Compte tenu de la formule \eqref{e10}, on a alors 
 $$ \lim {\sA}\{ \frac{\varpi^j}{g} \}^{\s} \subset (\lim\,  {\sA'}\{ \frac{\varpi^j}{g} \}^{\s})^G = (g^{\f}\sA'^{\s})^G = ( \sA'^{{\s}\ast}_{\sA'[\frac{1}{g}]})^G   =   (\sA'^G)^{{\s}\ast}_{(\sA'^G)[\frac{1}{g}]}. $$  
  Dans l'autre sens, si $a\in  \sA^{\s}[\frac{1}{g}] $ est presque entier sur $\sA^{\s}$, ses images dans chaque ${\sA}\{ \frac{\varpi^j}{g} \}^{\s}[\frac{1}{g}] \subset  {\sA}\{ \frac{\varpi^j}{g} \}$ le sont sur ${\sA}\{ \frac{\varpi^j}{g} \}^{\s} $ donc sont dans ${\sA}\{ \frac{\varpi^j}{g} \}^{\s} $ (point 5) du sorite \ref{s1}), et on a donc $a\in \lim {\sA}\{ \frac{\varpi^j}{g} \}^{\s} $.  
 Si $\sA'^G = \sA$, on conclut que $ \lim {\sA}\{ \frac{\varpi^j}{g} \}^{\s} = \sA^{{\s}\ast}_{\sA[\frac{1}{g}]}$. Notons que $\sA' $ est le compl\'et\'e d'une sous-alg\`ebre stable sous $G$ dont les $G$-invariants se r\'eduisent \`a $\sA$, mais il n'est pas clair que l'\'egalit\'e $\sA'^G = \sA$ (du type Ax-Sen-Tate) vaille en g\'en\'eral.

   \subsubsection{Exemple: $\hat A_\infty$}\label{E18}  Comme les anneaux noeth\'eriens $A^{\s}_i$ sont int\'egralement clos et entiers les uns sur les autres, il est facile de voir que $A_\infty = {\cup A^{\s}_i}$ est compl\`etement int\'egralement clos. Le passage au compl\'et\'e $\hat A_\infty$ ne va pas de soi; on le contourne en utilisant la technique de \cite[V, n. 4 prop. 15]{B2} (voir aussi  \cite[lemma A6]{An}). 
     
 On a $\hat A_\infty^{\s} = \widehat{\cup A_i^{\s}}$ et $\hat A_\infty^{\s}$ s'identifie \`a $W(k[[T_{\leq n}^{\e}]])\hat\otimes_{W(k)} \hat K_\infty^{\s}$; ainsi, tout \'el\'ement $a\in \hat A_\infty^{\s}$ s'\'ecrit de mani\`ere unique sous la forme
\begin{equation}\label{ecr} \sum \varpi^{s} [a_s],\;\; a_s\in k[[T_{\leq n}^{\e}]],\end{equation}  et o\`u $s$ d\'ecrit une suite discr\`ete finie ou tendant vers l'infini dans $\frac{1}{p-1}\N[\frac{1}{p}]$.
 De mani\`ere analogue, 
  tout \'el\'ement $b\in \widehat{Q(  \hat A_\infty)}^{\s}$ s'\'ecrit de mani\`ere unique sous la forme
\begin{equation}\label{ecr'} \sum \varpi^{s} [b_s],\;\; b_s\in k((T_{\leq n}^{\e})). \end{equation}
  Par approximation successive, on montre alors que si les puissances de $b\in Q( \hat A_\infty^{\s})$ sont contenues dans un $ \hat A_\infty^{\s}$-module de type fini, alors chaque coefficient $[b_s]$ est le relev\'e d'un \'el\'ement $b_s\in k((T_{\leq n}^{\e}))$ dont toutes les puissances sont contenues dans un $k[[T_{\leq n}^{\e}]]$-module de type fini. On est donc ramen\'e \`a la compl\`ete int\'egralit\'e de $k[[T_{\leq n}^{\e}]]$, qui d\'ecoule de l'int\'egralit\'e des $k[[T_{\leq n}^{\frac{1}{p^i}}]]$ puisque ces anneaux noeth\'eriens sont entiers (et m\^eme finis) les uns sur les autres.

\smallskip Soient $\overline{Q(A)}$ une cl\^oture alg\'ebrique du corps des fractions $Q(A)$ de $A$, et $F$  le sous-corps engendr\'e par $\cup A_i$ et les racines $p$-primaires d'un \'el\'ement $g\in A$; c'est une extension galoisienne infinie de $Q(A)$ de groupe $G_0:=   \Z_p^{n+1}\ltimes \Z_p^\times$ (le dernier facteur $\Z_p$ agissant sur $g^\e$ selon l'action kummerienne usuelle). Soit $\tilde A$ la fermeture int\'egrale de $A$ dans $F$, et $\hat{\tilde A}$ son compl\'et\'e $p$-adique et $\sA' := \hat{\tilde A}[\frac{1}{p}]$; c'est une alg\`ebre de Banach multiplicativement norm\'ee et $\sA'^{\s} = \hat{\tilde A}$.
 L'action de $G_0$ se prolonge \`a $\tilde R$ et $\sA'^{\s}$.  Un th\'eor\`eme de type Ax-Sen-Tate d\^u \`a J.-P. Wintenberger \cite{Wi} montre, compte tenu de ce que $A$ est r\'egulier, que  $\sA'^{\s G_0}= A$ (et plus pr\'ecis\'ement que $H^1(G_0 , \tilde R)$ est anuul\'e par $p$ si $p$ est impair et par $4$ si $p=2$, ce qui implique le r\'esultat compte tenu de la suite exacte \cite[prop. 1]{Wi})
\[ 0 \to (R/p^m R)^{G_0} \to (\tilde R /p^m \tilde R)^{G_0} \to H^1(G_0 , \tilde R)(p^m)\to 0.\]  
 Par ailleurs, on a un \'epimorphisme extr\'emal $\hat A_\infty\langle g^{\e}\rangle \to \sA'$, donc $\sA'$ est perfecto\"{\i}de (\cf \ref{r8}; c'est m\^eme un isomorphisme si $\hat A_\infty\langle g^{\e}\rangle$ est multiplicativement norm\'ee), et on peut lui appliquer le th. \ref{T5}. On conclut comme ci-dessus que
   \begin{equation}A = \lim A_0\{ \frac{p^j}{g}\}^{\s}.   \end{equation}
   Si $G_i\subset G_0$ d\'esigne le sous-groupe ouvert correspondant \`a l'extension galoisienne $Q(A_i)$, on a encore  que $H^1(G_i , \tilde R)$ est annul\'e par $p$ si $p$ est impair et par $4$ si $p=2$; pour $G = \lim G_i$, $H^1(G , \tilde R) = {\rm{colim}}\, H^1(G_i , \tilde R)$ l'est de m\^eme,  et on conclut par le m\^eme argument que  $\sA'^{\s G} = \widehat{\rm{colim}}\, A_i^{\s} = A_{\infty}^{\s}$ et 
   \begin{equation} A_{\infty}^{\s} = \lim  A_\infty\{ \frac{p^j}{g}\}^{\s}.\end{equation}

 \medskip\subsection{Limite uniforme vs. limite perfecto\"{\i}de}
     La proposition suivante analyse la situation abord\'ee en \ref{lplu}, et pr\'epare le terrain pour le th. \ref{T6}.  
   De nouveau, $\sK$ d\'esigne un corps perfecto\"{\i}de de caract\'eristique $0$, $\varpi$ un \'el\'ement de $\sK^{\ss}\setminus 0$ tel que $p\sK^{\s}\subset \varpi\sK^{\s}$, et $\varpi_{\frac{1}{p}}$ un \'el\'ement de norme $\vert\varpi\vert^{\frac{1}{p}}$.

       \begin{prop}\label{P14} Soient $(\sA^\alpha)$ un syst\`eme projectif de $\sK$-alg\`ebres perfecto\"{\i}des, et $\sA^\natural \to \sA$ le morphisme canonique entre limite perfecto\"{\i}de et limite uniforme.  
      \item Consid\'erons les conditions suivantes:
 \begin{enumerate} 
 \item  $\sA$ est perfecto\"{\i}de,
 \item  $\sA^\natural \to \sA$ est un isomorphisme;
  \item le morphisme compos\'e $\sA^{\natural\s }/\varpi \to \sA^{\s }/\varpi \to \lim\,({\sA^{\alpha \s }}/\varpi)$ (induit par les morphismes $\sA^{\natural\s }/\varpi\cong \sA^{\flat \sm o}/\varpi^\flat \to \sA^{\alpha\flat \sm o}/\varpi^\flat\cong  \sA^{\alpha  \sm o}/\varpi  $)  est presque surjectif,
   \item  Le morphisme $\displaystyle    \lim\,({\sA^{\alpha \s }}  /\varpi) \otimes_{\sK^{\s}/\varpi, F} \sK^{\s}/\varpi  \stackrel{\phi}{\to}  \lim\,({\sA^{\alpha \s }}/\varpi)$ induit par Frobenius est presque surjectif,
   \item  $  \lim\,({\sA^{\alpha \s }}/\varpi)   \to   \lim\,({\sA^{\alpha \s }}/\varpi_{\frac{1}{p}})$ est presque surjectif. \end{enumerate}
    On a 
   {\centerline{ $\; (1)  \Leftrightarrow (2) \;\Leftarrow \; (3)  \Leftrightarrow  (4) \Leftrightarrow (5) .$}}
\end{prop}  
 
  \begin{proof}   Posons $\mathfrak L :=  \lim\,({\sA^{\alpha \s }}/\varpi)$ et $\mathfrak L _{\frac{1}{p}} :=  \lim\,({\sA^{\alpha \s }}/\varpi_{\frac{1}{p}} )$ pour abr\'eger. D'apr\`es la formule \eqref{e26'}, on a 
  \begin{equation} \displaystyle (\sA^\natural)^{\flat\sm o} = \sA^{\flat\sm o}   = \lim_{F} \mathfrak L. \end{equation} 
      
\noindent $(1) \Leftrightarrow (2)$ est clair (\cf \ref{lp}). 
 
\smallskip\noindent  $(3) \Rightarrow (4)$: $(3)$ signifie que $\displaystyle \lim_{F} \mathfrak L\to  \mathfrak L$ est presque surjectif, donc la derni\`ere fl\`eche de transition aussi, ce qui donne $(4)$. 
  
 \smallskip\noindent   $(4) \Leftrightarrow (5)$: notons $\sigma$ l'isomorphisme $ \sK^{\s}/\varpi_{\frac{1}{p}} \stackrel{ x\mapsto x^p}{\to} \sK^{\s}/\varpi$, et $\rho$ le morphisme canonique $ \mathfrak L/\varpi_{\frac{1}{p}}   \to  \mathfrak L _{\frac{1}{p}} $.
   Puisque les $\sA^{\alpha}$ sont perfecto\"{\i}des, on a un syst\`eme compatible d'isomorphismes de Frobenius  $\,\displaystyle    ({\sA^{\alpha \s }}  /\varpi_{\frac{1}{p}}) \otimes_{\sK^{\s}/\varpi_{\frac{1}{p}}, \sigma\,} \sK^{\s}/\varpi  \stackrel{ \sim}{\to}    \, {\sA^{\alpha \s }}/\varpi ,\,$   
   et par passage \`a la limite un isomorphisme 
    $\; \psi:\;  \mathfrak L _{\frac{1}{p}}\otimes_{\sK^{\s}/\varpi_{\frac{1}{p}}, \sigma} \sK^{\s}/\varpi \stackrel{ \sim}{\to}  \mathfrak L.$ 
    On a d'autre part un isomorphisme 
     \[ \iota:\;   \mathfrak L\otimes_{\sK^{\s}/\varpi , F} \sK^{\s}/\varpi \stackrel{ \sim}{\to}    \mathfrak L/\varpi_{\frac{1}{p}}\otimes_{\sK^{\s}/\varpi_{\frac{1}{p}}, \sigma} \sK^{\s}/\varpi ,\]
      et $\phi$ est le compos\'e
     $ \,\psi \circ (\rho\otimes 1)\circ \iota  $. Donc $\phi$ est presque surjectif si et seulement si $\rho$ l'est.

  \medskip $(4)\Rightarrow (3)$:   
 $(4)$ \'equivaut \`a la presque surjectivit\'e de $  \mathfrak L  \otimes_{\sK^{\s}/\varpi, F} \sK^{\s}/\varpi  \stackrel{F\otimes 1}{\to} \mathfrak L$.
Cette derni\`ere entra\^{\i}ne la surjectivit\'e de $F:  \sK^{\ss} \mathfrak L  {\to} \sK^{\ss}\mathfrak L$. Par le lemme de Mittag-Leffler pour les $\F_p$-espaces vectoriels, on en d\'eduit la surjectivit\'e de $\displaystyle  \lim_F  \sK^{\ss} \mathfrak L \to \sK^{\ss}\mathfrak L$, d'o\`u la presque surjectivit\'e de $\displaystyle    \lim_F \mathfrak L \to  \mathfrak L$.

\medskip $(3) + (4)\Rightarrow (1)$: d'apr\`es le lemme \ref{L13}, $\sA^{\s}/\varpi \to  \mathfrak L$ est injectif. Sous $(3)$, il est presque surjectif, donc est en fait un  
presque-isomorphisme, et $(4)$ entra\^{\i}ne alors que le morphisme $\displaystyle  ( \sA^{ \s }/\varpi ) \otimes_{\sK^{\s}/\varpi, F} \sK^{\s}/\varpi  \stackrel{\phi}{\to} \sA^{ \s }/\varpi $ induit par Frobenius est presque surjectif.  
     \end{proof}

 \subsubsection{Remarque}\label{r11} On peut appliquer le m\^eme argument dans le cadre $(\sK^{\s}[ T^{\e}],  T^{\e}\sK^{\ss}.\sK^{\s}[T^{\e}])$. Soit $(\sA^\alpha)$ un syst\`eme projectif de $\sK\langle T^{\e}\rangle$-alg\`ebres perfecto\"{\i}des, de limite uniforme $\sA$. Si 
 $\displaystyle    \lim\,({\sA^{\alpha \s }}  /\varpi)  \to   \lim\,({\sA^{\alpha \s }}  /\varpi_{\frac{1}{p}})  $  est presque surjectif, alors $\sA^\natural \to \sA$ est un presque-isomorphisme dans ce cadre, \ie $\sA$ est presque perfecto\"{\i}de (\cf sorite \ref{S3}); en outre, $\sA^{\s}/\varpi \to \lim\,({\sA^{\alpha \s }}  /\varpi)$ est un presque-isomorphisme.

  \medskip\subsection{Limite uniforme d'alg\`ebres perfecto\"{\i}des {``en gigogne"}.}\label{lun}     
  On \'etudie ici dans quelle mesure la limite uniforme d'un syst\`eme projectif d'alg\`ebres perfecto\"{\i}des $\sB^{j}$ est (presque) perfecto\"{\i}de, dans la situation o\`u les $\sB^{j}$ se d\'eduisent les unes des autres par localisation ($ \sB^{i}\{ \frac{\varpi^j}{g}\}   \cong \sB^{j}$).  

  \subsubsection{} Reprenons la situation et les notations de \ref{lu}, et pla\c cons-nous dans le cadre $(\sK^{\s}[T^{\e}], (\varpi T)^{\e}\sK^{\s}[T^{\e}])$.    
  
  \begin{prop}\label{P16}  Soit $(R^i)_{i\in \N}$ un syst\`eme projectif de $(\sK^{\s}/\varpi^r)[T^{\e}]$-alg\`ebres dont les morphismes de transition $R^i \to R^j$ se factorisent par $R^{i[j]}$ de mani\`ere compatible en $(i,j)$, ces factorisations induisant des isomorphismes $(R^{i[j]})^a \cong (R^j)^a$. 
     
        Alors $ \displaystyle     \;   {\lim_{j}}^1 (R^j)^a = 0.$  
      \end{prop}

  \begin{proof}   Posons $u_{i[j]}^s:= 1\otimes (\frac{\varpi^{j }}{T})^{s}\in  R^{i[j]}=    R^{i}\otimes_{\sK^{\s}[T^{\e}] } \,\sK^{\s}[T^{\e}, (\frac{\varpi^{j}}{T})^{\e}]$. Pour toute sous-$(\sK^{\s}/\varpi^r)[T^{\e}]$-alg\`ebre $R'^i$ de $R^i$, introduisons les sous-$R'^i$-modules 
\begin{equation}\label{e34} (R'^{i})^{[j]k} :=  \sum_{s\in \N[\frac{1}{p}] \cap [0, \frac{1}{p^k}]} \, R'^i.u_{i[j]}^s \subset  R^{i[j]}.  \end{equation} 
Pour $R'^i= R^i$, on obtient ainsi un triple syst\`eme projectif $(R^{i[j]k})$ de $(\sK^{\s}/\varpi^r)[T^{\e}]$-modules index\'e par $(i,j,k)$. 
  Notons $\bar R^{i'[i]}$ l'image de $R^{i'[i]}$ dans $R^i$, et $\bar u_{[i]}^s$ l'image de $u_{i'[i]}^s$ (elle ne d\'epend pas de $i'$).

    Comme $u_{i[j'] }^s$ s'envoie sur $0$ dans $R^{i[j]}$ si $\,s({j'- j})\geq r$, $ R^{i[j']} \to  R^{i[j]}$  se factorise \`a travers $R^{i[j]k}$ d\`es que $ j'\geq j+ rp^k$, d'o\`u 
\begin{equation}\label{e36} \,\displaystyle  {\lim_{j}}^1 (R^{j} )^a \cong {\lim_{ij}}^1 (R^{i[j]})^a \cong  {\lim_{ijk}}^1 (R^{i[j]k})^a .\end{equation}  

  Comme $\varpi^{\frac{1}{p^k}}u_{i[j]}^s  $ est l'image dans $R^{i[j]}$ de $\varpi^{ \frac{1}{p^k}- s }u_{i[j+1]}^s$ pour $s\leq \frac{1}{p^k} $,  le conoyau de $R^{i[j+1]k} \to R^{i[j]k} $ est annul\'e par $\varpi^{\frac{1}{p^k}}$. 

   Comme $T^{\frac{1}{p^k}}u_{i[j]}^s$ est dans l'image de $ R^i$ pour $s\leq \frac{1}{p^k} $, le conoyau de  $R^{i,[j ],k+1} \to R^{i[j]k} $ est annul\'e par $T^{\frac{1}{p^{k }}}$.

 Enfin,  
   $(\bar R^{i'[i]})^{[j], k+1}$ est le sous-$R^{i'}$-module de $(\bar R^{i'[i]})^{[j]}$ 
   engendr\'e par les produits des images de $\bar u_{[i]}^s$ et des $ u_{i[j]}^t$ pour $t\leq {\frac{1}{p^{k+1}}}$, c'est-\`a-dire par les  $\varpi^{(i-j)s}u_{i[j]}^{s+t}$.  Supposons $ i'\geq  i \geq j + \frac{rp^{k+1}}{p-1}$. Alors la double in\'egalit\'e $(s+t\geq \frac{1}{p^k}, \, t\leq \frac{1}{p^{k+1}})$ implique  $s(i-j) \geq r$, de sorte que, sous la condition $t\leq \frac{1}{p^{k+1}}$, $\varpi^{(i-j)s}u_{i[j]}^{s+t}$ est nul si $s+t\geq \frac{1}{p^k}$, et appartient \`a $R^{i'[j] k}$ sinon; d'o\`u l'inclusion $(\bar R^{i'[i]})^{[j], k+1}\subset  R^{i'[j] k}$. Compte tenu de ce que $R^{i' [i]}\to R^i $ est un presque-isomorphisme, on en d\'eduit que l'image de $R^{i, [j], k+1 } $ dans $R^{i  [j] k } $ est presque contenue dans celle de $R^{i'[j]k}$.
  
  Prenons alors $j= k$, choisissons une suite croissante $i_k   \geq k + \frac{rp^{k+1}}{p-1}$, et consid\'erons le diagramme commutatif
   \[   \begin{CD} R^{i_{k+1}, [k+1], k+1}  @> a >> R^{i_{k+1}, [k], k+1}  @>  >>  R^{i_{k}, [k], k +1}    \\       @V VV   @V bVV @VVd V     \\\     R^{i_{k+1}, [k+1], k } @>  >>  R^{i_{k+1}, [k ], k } @> c >>  R^{i_{k}, [k], k }. \end{CD}\]
    D'apr\`es ce qui pr\'ec\`ede, $\Coker a$ (\resp $\Coker b$) est annul\'e par $ \varpi^{\frac{1}{p^{k+1 }}}$  (\resp $T^{\frac{1}{p^{k }}}$),  $\Im d$ est presque contenu dans $\Im c$, donc $\Coker c$ est presque quotient de $\Coker d$, donc presque annul\'e par $(\varpi T)^{\frac{1}{p^{k }}}$. On trouve que le compos\'e $cba$ est annul\'e par $(\varpi T)^{\frac{3}{p^k}}$. 
     Puisque  $\sum \frac{3 }{p^k} $ converge, on obtient $\,\lim^1 (R^{i_{k }, [k ], k  })^a= 0\,$ (\cf \cite[lemme 2.4.2 $iii)$]{GR1}). 
  
  Compte tenu de \eqref{e36}, on conclut que  $\,\displaystyle { {\lim_{j}}^1 (R^j)^a =    { \lim_{ijk}}^1 (R^{i   [j ], k  })^a = 0}$. 
    \end{proof}

 \subsubsection{}\label{ec} Si $\sB$ est une $\sK\langle T^{\e} \rangle$-alg\`ebre presque perfecto\"{\i}de, et $g$ l'image de $T$ dans $\sB$, posons $\sB^{[j]} := \sB\{ \frac{\varpi^j}{g}\}$. C'est une $\sK\langle T^{\e} \rangle$-alg\`ebre perfecto\"{\i}de: en effet, il d\'ecoule du lemme  \ref{L11} que $\sB^{[j]} \cong \sB^{\natural [j]}$, qui est perfecto\"{\i}de d'apr\`es la prop. \ref{P11}.

 \smallskip Soit $\sA$ est une $\sK\langle T^{\e} \rangle$-alg\`ebre presque perfecto\"{\i}de.  Rappelons que la cat\'egorie $\sA^{\hat a}\hbox{-}{\bf{uBan}}$ des $\sA^{\hat a}$-alg\`ebres de Banach uniformes a pour objets les $\sA$-alg\`ebres de Banach uniformes et pour morphismes ceux obtenus apr\`es application du foncteur $(\;)^{\s a}$ (\cf \S \ref{recc}). Les $\sA^{\hat a}$-alg\`ebres presque perfecto\"{\i}des forment une sous-cat\'egorie pleine $\sA^{\hat a}\hbox{-}{\bf{pPerf}}$ (\cf \S \ref{app}). D'apr\`es la formule \eqref{e11}, on a un isomorphisme canonique
 \begin{equation} \sB^{[j]} \cong \sB \hat\otimes_{\sA} \sA^{[j]}. \end{equation}  

   Notons $2\hbox{-}{\rm{lim}}\,\sA^{[j]}\hbox{-}{\bf{Perf}} $ la cat\'egorie des syst\`emes projectifs de $\sA^{[j]}$-alg\`ebres perfecto\"{\i}des $\sB^{j}$  dont les morphismes de transition $\sB^{i}\to \sB^{j}$ se factorisent \`a travers des isomorphismes 
\begin{equation}\label{e40} \sB^{i}\{ \frac{\varpi^j}{g}\}   \cong \sB^{j}.\end{equation}
  La localisation induit un foncteur

 \centerline{$\sA^{\hat a}\hbox{-}{\bf{pPerf}} \; \stackrel{\varsigma}{\to} \; 2\hbox{-}{\rm{lim}}\,\sA^{[j]}\hbox{-}{\bf{Perf}}  $.}
 
     \begin{thm}\label{T6} Le foncteur $\varsigma$ admet ${\rm{ulim}}$ comme quasi-inverse \`a gauche et adjoint \`a droite. En particulier, il induit une equivalence de $\sA^{\hat a}\hbox{-}{\bf{pPerf}} $ avec une sous-cat\'egorie (pleine) cor\'eflexive de $2\hbox{-}{\rm{lim}}\,\sA^{[j]}\hbox{-}{\bf{Perf}}  $. \end{thm} 
   
  \begin{proof}  Commen\c cons par prouver que pour tout objet $(\sB^{j})$ de $2\hbox{-}{\rm{lim}}\,\sA^{[j]}\hbox{-}{\bf{Perf}}, $ la limite uniforme $\sB$ est presque perfecto\"{\i}de (c'est clair si $\car\, \sK=p$ puisque la perfection est pr\'eserv\'ee \`a la limite des $\sA^{[j]{\s}}$). Pla\c cons-nous de nouveau dans le cadre $(\sK^{\s}[T^{\e}], (\varpi T)^{\e}\sK^{\s}[T^{\e}])$. 

 En appliquant la prop. \ref{P16} avec $ R^j := \sB^{j\sm o}/\varpi^r $, on obtient  $ \displaystyle { \lim_{j}}^1  (\sB^{ j\sm o}/\varpi^r )^a = 0 $.
 Comme $\sB^{j\sm o} $ est plate sur $\sK^{\s}$, on a une suite exacte
\[ 0\to \sB^{j\sm o}/\varpi^{1-\frac{1}{p}} \to  \sB^{j\sm o}/\varpi\to  \sB^{j\sm o}/\varpi^{\frac{1}{p}} \to 0\]
et par presque-nullit\'e de la $ \displaystyle {\lim_{j}}^1$, on obtient encore une suite exacte \`a la limite:
 \[\displaystyle 0\to \lim_{j}(\sB^{j\sm o}/\varpi^{1-\frac{1}{p}})^a \to \lim_{j}( \sB^{j\sm o}/\varpi)^a \to  \lim_{j}(\sB^{j\sm o}/\varpi^{\frac{1}{p}})^a \to 0.\]
 La variante de prop. \ref{P14} signal\'ee dans la rem. \ref{r11} entra\^{\i}ne alors que $\sB$ est presque perfecto\"{\i}de.
 
 Que ${\rm{ulim}} $, ou de mani\`ere \'equivalente ${\rm{ulim}}^\natural $, soit quasi-inverse \`a gauche de $\varsigma$ d\'ecoule du th. \ref{T5}.  
 Enfin, soient $\sC$ un objet de $\sA^{\hat a}\hbox{-}{\bf{pPerf}} $ et $(\sB^j)_j$ un objet de  $2\hbox{-}{\rm{lim}}\,\sA^{[j]}\hbox{-}{\bf{Perf}}  $. L'application ${\rm{Hom}}(\varsigma(\sC)  , (\sB^j)_j)\to {\rm{Hom}}(\sC  , {\rm{ulim}}\, \sB^j) $ induite par ${\rm{ulim}} $ admet comme inverse (fonctoriellement en $(\sC, (\sB^j)_j)$) l'application induite par $(\;)^{[j]}$ suivi du morphisme canonique $(\sB^{[j]})_j\to (\sB^j)_j$.  
 Pour la seconde assertion, voir \cite[IV 4 ex. 4]{M}.  \end{proof}

  \subsubsection{Remarques} $(1)$ On ne peut remplacer les $\sK\langle T^{\e}\rangle$-alg\`ebres par des $\sK\langle T \rangle$-alg\`ebres, comme on le voit sur l'ex. \ref{E17}. 

\smallskip\noindent $(2)$ En vertu du lemme \ref{L15}, on pourrait remplacer $\sA^{\hat a}\hbox{-}{\bf{pPerf}} $ par sa sous-cat\'egorie pleine \'equivalente $\sA^{\hat a}\hbox{-}{\bf{Perf}} $, et ${\rm{ulim}}$ par ${\rm{ulim}}^\natural$. Si la question \ref{q} a une r\'eponse positive, il est m\^eme inutile de substituer ${\rm{ulim}}^\natural$ \`a ${\rm{ulim}}$.

\begin{qn} {\it $\varsigma\,$ est-il essentiellement surjectif $\,$ (et par suite une \'equivalence)?} \end{qn}
 
\medskip  Il revient au m\^eme de demander si l'on retrouve les $\sB^j$ par localisation affino\"{\i}de de ${\rm{ulim}}\, \sB^j$. 
   Voici un cas particulier important o\`u c'est bien le cas:
 
 \begin{prop}\label{P17} Soit $\sB'$ une $\sA[\frac{1}{g}]$-alg\`ebre \'etale finie. 
 \begin{enumerate} 
 \item Les $\sB^j := \sB'\otimes_{\sA[\frac{1}{g}]} \sA^{[j]}$ forment un objet de $2\hbox{-}{\rm{lim}}\,\sA^{[j]}\hbox{-}{\bf{Perf}}  $. On note ${} \sB$ la limite uniforme. 
 \item Les co\"{u}nit\'es d'adjonction $ {} \sB^{[j]}\stackrel{\eta_j}{\to} \sB^j$ sont des isomorphismes. 
 \item Elles proviennent en fait d'un isomorphisme $  {} \sB[\frac{1}{g}]\stackrel{\sim}{\to}  \sB' $. 
 \end{enumerate}
 \end{prop} 
 
 \begin{proof} $(1)$ Les  $\sB^j $ forment un syst\`eme projectif de $\sA^{[j]}$-alg\`ebres \'etales finies, donc perfecto\"{\i}des (th. \ref{T4}). On a donc 
 $\sB^{i}\otimes_{\sA^{[i]}} \sA^{[j]} = \sB'\otimes_{\sA[\frac{1}{g}]} \sA^{[i]} \otimes_{\sA^{[i]}} \sA^{[j]} \cong \sB^{j}$, et comme  $\sB^{i}$ est projectif fini sur $\sA^{[i]}$, $\sB^{i}\otimes_{\sA^{[i]}} \sA^{[j]} = \sB^{i}\hat\otimes_{\sA^{[i]}} \sA^{[j]} \cong \sB^{i}\{ \frac{\varpi^j}{g}\}  $. 
 
 \noindent \smallskip $(2)$ La fermeture int\'egrale $ \sA^{\s +}_{\sB'}$ s'envoie naturellement vers la fermeture int\'egrale $\sA^{[j] \s +}_{\sB^j}$ qui n'est autre que $\sB^{j \s}$ d'apr\`es le cor. \ref{C2'}, ce qui fournit \`a la limite un morphisme canonique  $\sA^{\s +}_{\sB'}   \to {} \sB^{\s} $, d'o\`u aussi $\sA^{\s +}_{\sB'}[\frac{1}{\varpi g}]= \sB'  \stackrel{\delta}{\to} {} \sB[\frac{1}{g}]$. Le morphisme compos\'e 
    $\delta_j:\, \sB^j = \sB' \otimes_{\sA[\frac{1}{g}]} \sA^{[j]}  \stackrel{\delta\otimes 1}{\to}  {} \sB[\frac{1}{g}]  \otimes_{\sA[\frac{1}{g}]} \sA^{[j]} =  {} \sB  \otimes_{\sA} \sA^{[j]}  \to  {} \sB   \hat\otimes_{\sA} \sA^{[j]} =  {} \sB^{[j]} $   est 
 un inverse \`a droite de $\eta_j$.  
  Or par adjonction, tout morphisme de la forme $\varsigma(\sC) \to (\sB^j)_j$ se factorise de mani\`ere unique en un morphisme $\varsigma(\sC) \to \varsigma( {} \sB)$. Appliquant ceci \`a $\sC = {} \sB$ et au morphisme de composantes $\eta_j (\delta_j \eta_j ) = \eta_j$, on en d\'eduit $\delta_j \eta_j = id_{{} \sB^{[j]}}$.  Donc $\eta_j$ est un isomorphisme d'inverse $\delta_j$ \footnote{cet argument d'adjonction montre plus g\'en\'eralement que toute sous-cat\'egorie \'epi-cor\'eflexive est mono-cor\'eflexive \cite[prop. 1]{HS}.}. 
  
\noindent $(3)$ Par ailleurs, $\delta$ est injectif. En effet, le noyau $\mathfrak I$ de $\sB' \to {}\sB[\frac{1}{g}]$ est un id\'eal de trace nulle sur  $\sA[\frac{1}{g}]$: en effet, $tr(\mathfrak I)$ est d'image nulle dans chaque $\sA^{[j]}$, donc est nul puisque $\sA[\frac{1}{g}] $ s'envoie injectivement dans ${\rm{ulim}}\,\sA^{[j]}$ (lemme \ref{L11'}). Donc $\mathfrak I$ est lui-m\^eme nul puisque $\sB'$ est \'etale sur $\sA[\frac{1}{g}] $.  

 ll reste \`a voir que $\delta$ est surjectif. On peut remplacer $\sA^{\s}$ par $g^{\f}\sA^{\s}$, ce qui permet de supposer $\sA$ (compl\`etement) int\'egralement ferm\'e dans $\sA[\frac{1}{g}]$.  
 
  \smallskip $a)$  Supposons pour commencer $\sB'$ galoisien sur ${\sA[\frac{1}{g}]} $ de groupe $G$. On a alors 
  \[ {}\sB[\frac{1}{g}]^G=({\rm{ulim}}\, (\sB'\otimes_{\sA[\frac{1}{g}]} \sA^{[j]})^G)[\frac{1}{g}] = ({\rm{ulim}}\,  \sA^{[j]})[\frac{1}{g}]  = \sA[\frac{1}{g}].  \]  
On conclut par le lemme \ref{L3} (avec $R= \sA[\frac{1}{g}],\, S' =\sB', \,S = {}\sB[\frac{1}{g}]$).

  \smallskip $b)$ En g\'en\'eral, le rang de $\sB'$ sur $\sA[\frac{1}{g}]$ \'etant fini continu et born\'e, et $\sA$ \'etant $\sA$ int\'egralement ferm\'e dans $\sA[\frac{1}{g}]$, on peut supposer, en d\'ecomposant  $\sA$ en un nombre fini de facteurs, que ce rang est constant, \'egal \`a $r$. D'apr\`es le lemme \ref{L0}, on a une extension galoisienne $\sA[\frac{1}{g}] \inj \sC'$ de groupe $\mathfrak S_r$ se factorisant par $\sB'$ et telle que $\sB' = \sC'^{\mathfrak S_{r-1}}$; d'o\`u $\sB^j = \sC^{j \mathfrak S_{r-1}}$ et \`a la limite ${}\sB = {}\sC^{\mathfrak S_{r-1}}$. D'apr\`es le pas $a)$, on a  $\sC'= {}\sC[\frac{1}{g}]$, donc $\sB'=   \sC'^{\mathfrak S_{r-1}} =  {}\sC[\frac{1}{g}]^{\mathfrak S_{r-1}} = {}\sB [\frac{1}{g}]$. \end{proof}

   \subsubsection{\it Remarque.} Le fait que $\delta$ et $\delta_j$ soient des isomorphismes implique que $\tilde \sB  \otimes_{\sA} \sA^{[j]}  \to  \tilde \sB   \hat\otimes_{\sA} \sA^{[j]}$ en est un.

\newpage  \section{Le {``{lemme d'Abhyankar}"} perfecto\"{\i}de.}\label{LAP}   
   
\subsection{}   On fixe un corps perfecto\"{\i}de $\sK$ et un \'el\'ement $\varpi\in \sK^{{\ss}}$ admettant une suite compatible de racines $p^m$-i\`emes, tel que $\vert p\vert \leq \vert \varpi \vert$. On
 prend pour cadre 
 \[(\mathfrak V =  \sK^{\s}[T^{\e}], \, \mathfrak m =   (\varpi T)^{\e}\mathfrak V ).\]
 
 \smallskip Soit $\sA$ est une $\sK\langle T^{\e} \rangle$-alg\`ebre presque perfecto\"{\i}de.  Rappelons encore que la cat\'egorie $\sA^{\hat a}\hbox{-}{\bf{pPerf}}$ a pour objets les $\sA$-alg\`ebres de Banach uniformes $\sB$ qui sont presque isomorphes \`a $\sB^\natural$, et pour morphismes ceux obtenus apr\`es application du foncteur $(\;)^{\s a}$ (\cf \S \ref{recc}). 
   
 On dira qu'un objet $\sB^{\hat a}$ de $\sA^{\hat a}\hbox{-}{\bf{uBan}}$ est {\it presque \'etale fini} sur $\sA^{\hat a}$ si $\sB^{a}$ est \'etale fini sur $\sA^{a}$. 
 Rappelons que d'apr\`es la prop. \ref{P10}, si $\sA$ est presque perfecto\"{\i}de,  l'inversion de $\varpi$ induit une \'equivalence
\begin{equation}\label{ep10}\sA^{\s a}\hbox{-}{\bf{Alg}}^{et.fin} \; \stackrel{\xi}{\to} \; \sA^{\hat a}\hbox{-}{\bf{pPerf}}^{p.et.fin} ,\end{equation} de quasi-inverse $(\;)^{\s a}$,
entre  $\sA^{\s a}$-alg\`ebres \'etales finies et $\sA^{\hat a}$-alg\`ebres presque perfecto\"{\i}des presque \'etales finies. 
  
      \subsection{}\label{tp}  Le th\'eor\`eme principal de cet article est une extension du th. \ref{T4} au cas ramifi\'e (on retrouve \ref{T4} dans le cas o\`u $T^{\frac{1}{p^i}}.1 = 1$).    
     
       \begin{thm}\label{T7}  Soient $\sK$ un corps perfecto\"{\i}de et $\sA$ une $ \sK\langle T^{\e}\rangle$-alg\`ebre presque perfecto\"{\i}de. On note $g $ l'\'el\'ement $T.1$ de $\sA^{\s}$. 
     On a alors des \'equivalences de cat\'egories 
        
     \medskip\leftline{
      $   \{\sA^{\s a}$-alg\`ebres compl\`etes, \'etales  finies modulo toute puissance de $p $ } 
       \centerline{et \'etales finies apr\`es inversion de $g$ 
   $\}\stackrel{\sim}{\to} $ } 
     
        \medskip\centerline{$\{$ $\sA^{\hat a}$-alg\`ebres presque perfecto\"{\i}des, \'etales finies apr\`es inversion de $g\}  \stackrel{\sim}{\to}$}

        \medskip\rightline{ $\{ \sA[\frac{1}{g}]$-alg\`ebres \'etales finies$\}  $ }    
     \medskip \noindent   induites par inversion de $\varpi$ et de $ g$ respectivement.
     \end{thm}
     
    \begin{qn} {\it Ces cat\'egories sont-elles aussi \'equivalentes \`a celles de \eqref{ep10}?} \end{qn}
    \cf rem. \ref{fr}.

       \begin{proof} Il est loisible de remplacer $\sA$ par toute alg\`ebre isomorphe dans  $\sK\langle T^{\e}\rangle^{\hat a}\hbox{-}{\bf{pPerf}}$, notamment par $\sA^{\hat a}_{\hat \ast} = g^{\f}\sA^{\s}[\frac{1}{\varpi}]$. Cela permet de supposer d'embl\'ee que $g$ est non-diviseur de z\'ero et, en vertu du th. \ref{T5}, que 
  \begin{equation}\label{e41} \sA  \cong {\rm{ulim}}\, \sA^{[j]}  .  \end{equation} 

     Notons $ \sA^{\s a}\hbox{-}{\bf{Alg}}^{(/p^\bullet, \frac{1}{g})et.fin} $ la premi\`ere cat\'egorie du th\'eor\`eme, et par $\sA^{\hat a}\hbox{-}{\bf{pPerf}}^{\frac{1}{g} et.fin} $ la seconde, et 
       \[\sA^{\s a}\hbox{-}{\bf{Alg}}^{(/p^\bullet, \frac{1}{g})et.fin}  \stackrel{\xi'}{\to} \sA^{\hat a}\hbox{-}{\bf{pPerf}}^{\frac{1}{g} et.fin} \stackrel{\varrho'}{\to}  \sA[\frac{1}{g}]\hbox{-}{\bf{Alg}}^{et.fin}  \] les foncteurs en jeu, dont il s'agit de montrer que ce sont des \'equivalences.

  \smallskip Commen\c cons par prouver que {\it $\varrho'$ est une \'equivalence}. 
   
 \smallskip Notons $2\hbox{-}{\rm{lim}}\,\sA^{[j]}\hbox{-}{\bf{Alg}} $ la cat\'egorie des syst\`emes projectifs de $\sA^{[j]}$-alg\`ebres $\sB^{j}$  dont les morphismes de transition $\sB^{i}\to \sB^{j}$ se factorisent \`a travers des isomorphismes 
 \begin{equation}\label{e42} \sB^{i}\{ \frac{\varpi^j}{g}\}   \cong \sB^{j}.\end{equation}

   Observons que  $( \sA^{[j]})_j$ est un syst\`eme projectif de $ \sA[\frac{1}{g}]$-alg\`ebres, et consid\'erons le diagramme de foncteurs
  \begin{equation}\label{e43}  \begin{CD}  \sA^{\hat a}\hbox{-}{\bf{pPerf}}    @> \varrho >>  \sA[\frac{1}{g}]\hbox{-}{\bf{Alg}}  \\@VV\varsigma V     \\\ 2\hbox{-}{\rm{lim}}\,\sA^{[j]}\hbox{-}{\bf{Perf}}   @> \upsilon >>  2\hbox{-}{\rm{lim}}\,\sA^{[j]}\hbox{-}{\bf{Alg}}     \end{CD}\end{equation}
o\`u $\varrho$ est le foncteur oubli de la norme et inversion de $g$,
 $\upsilon$ est le foncteur oubli de la norme,   
 et $\varsigma$ est le foncteur consid\'er\'e au th. \ref{T6}, qui est pleinement fid\`ele, de quasi-inverse \`a gauche $\lambda$ donn\'e par la limite uniforme. 
  
   Apr\`es restriction aux alg\`ebres \'etales finies, ce diagramme se compl\`ete d'apr\`es la prop. \ref{P17} $(1), (2)$ en un diagramme essentiellement commutatif 
    \begin{equation} \begin{CD}\label{e44}  \sA^{\hat a}\hbox{-}{\bf{pPerf}}^{\frac{1}{g}et.fin}  @> \varrho' >>  \sA[\frac{1}{g}]\hbox{-}{\bf{Alg}}^{et.fin}  \\@VV\varsigma' V @V\tau' VV     \\\ 2\hbox{-}{\rm{lim}}\,\sA^{[j]}\hbox{-}{\bf{Perf}}^{et.fin}   @> \upsilon' >>  2\hbox{-}{\rm{lim}}\,\sA^{[j]}\hbox{-}{\bf{Alg}}^{et.fin}    \end{CD} \end{equation} o\`u $\tau$ est donn\'e par les produits tensoriels $(- \otimes_{\sA[\frac{1}{g}]}  \sA^{[j]}  )_j $.
  D'apr\`es le th. \ref{T4}, $\upsilon'$ est une \'equivalence: on obtient un quasi-inverse canonique $\upsilon'^{-1}$  en munissant chaque $\sB^{j}$ de sa norme spectrale canonique. Posons \[\sigma := \lambda \nu'^{-1}\tau': \sA[\frac{1}{g}]\hbox{-}{\bf{Alg}}^{et.fin} \to  \sA^{\hat a}\hbox{-}{\bf{pPerf}},\; \sB' \mapsto {} \sB := {\rm{ulim}}\, \sB'\otimes_{\sA[\frac{1}{g}]} \sA^{[j]}.\]  
  On a $\,\sigma \rho' = \lambda \nu'^{-1} \tau' \rho' =  \lambda \nu'^{-1}\nu'\varsigma'= \lambda\varsigma'= id_{\sA^{\hat a}\hbox{-}{\bf{pPerf}}^{\frac{1}{g}et.fin} }.$
  
 \noindent D'autre part, d'apr\`es la prop. \ref{P17} $(3)$, on a un isomorphisme canonique $\sB' \stackrel{\sim}{\to} {} \sB[\frac{1}{g}] = \rho'\sigma(\sB')$, ce qui montre que $\sigma$ est \`a valeurs dans la $\varrho$-pr\'eimage $\sA^{\hat a}\hbox{-}{\bf{pPerf}}^{\frac{1}{g}et.fin} $ de $\sA[\frac{1}{g}]\hbox{-}{\bf{Alg}}^{et.fin}$ dans $\sA^{\hat a}\hbox{-}{\bf{pPerf}} $, puis que $\,\rho' \sigma = id_{ \sA[\frac{1}{g}]\hbox{-}{\bf{Alg}}^{et.fin}}$. 
 
 En d\'efinitive, (\ref{e44}) est un {\it diagramme d'\'equivalences de cat\'egories}.  
  
 \medskip 
 Prouvons ensuite que {\it $\xi'$ est une \'equivalence}. 
  Comme il est clair que $\xi'$ est pleinement fid\`ele, et compte tenu de ce que $\rho'$ est essentiellement surjectif, il suffit de montrer la 
 
 \begin{prop}\label{L20} Soit $\sB$ une $\sA$-alg\`ebre perfecto\"{\i}de dans laquelle $g$ n'est pas diviseur de z\'ero, et telle que  $\sB[\frac{1}{g}] $ soit \'etale finie sur $\sA[\frac{1}{g}] $. Alors 
 pour tout $m\in \N$,  $\sB^{\s a}/p^m$ est \'etale finie sur $\sA^{\s a}/p^m$ dans le cadre  $( \sK^{\s}[T^{\e}], \,    (\varpi T)^{\e} \sK^{\s}[T^{\e}])$, 
 et $\sB^{\s}[\frac{1}{g}]$ est presque \'etale finie sur $\sA^{\s}[\frac{1}{g}]$ dans le cadre $(\sK^{\s}, \sK^{\ss})$.     \end{prop} 
 
  Si $\car \sK = p$, cela est prouv\'e dans \cite[3.5.28]{GR1} (en prenant $V = \sA^{\s}$ et $\varepsilon = \varpi T$ dans \loccit). 
 Supposons donc $\car \sK =0$. 
   
 $a)$ Pla\c cons-nous encore pour commencer dans la situation o\`u
  $\sB[\frac{1}{g}]$ est extension galoisienne de galoisienne de  $\sA[\frac{1}{g}]$  de groupe $G$.  
    Posons pour all\'eger \begin{equation}\sB_2 := \sB\otimes_\sA \sB, \, \sB_2^{[j]} := \sB^{[j]}\otimes_{\sA^{[j]}} \sB^{[j]}. \end{equation} 
  D'apr\`es la prop. \ref{P7},  les compl\'et\'es $\widehat{\sB_2} = \sB \hat\otimes_\sA \sB$ et $  \hat\sB_2^{[j]}$ sont des alg\`ebres perfecto\"{\i}des et $\widehat{\sB_2}^{\s a}=  \sB^{\s a}\hat\otimes_{\sA^{\s a}}\sB^{\s a} $ (c'est vrai dans le cadre $(\sK^{\s}, \sK^{\ss})$ et a fortiori dans $(\mathfrak V =  \sK^{\s}[T^{\e}], \, \mathfrak m =   (\varpi T)^{\e}\mathfrak V )$). En outre, $  \widehat{\sB_2}^{[j]} =  \sB_2^{[j]} $ car $ \sB^{[j]}$ est \'etale fini sur $ \sA^{[j]}$. 
    D'apr\`es le th. \ref{T5}, on a 
  \begin{equation}\label{e45}  \lim {\sA^{[j]\s}} = \sA_\ast^{\s} ,\;\;\;\; \lim {\sB^{[j]\s} } = \sB_\ast^{\s} , \;\; \;\;\lim {\sB_2^{[j]\s} } = (\widehat{\sB_2}^{\s}) _\ast  = g^{\f}\widehat{\sB_2}^{\s}.\end{equation}
  Comme $\sB[\frac{1}{g}]$ est galoisien de groupe $G$ sur $\sA[\frac{1}{g}]$, $\sB^{[j]\s a} $ est galoisien de groupe $G$ sur ${\sA^{[j]\s a} }$ d'apr\`es la prop. \ref{P8'}. On en tire que $G$ agit par automorphismes de $\sB^{\s}_\ast$ et $G\times G$ par automorphismes de $(\widehat{\sB_2}^{\s})_\ast$, et que 
  \begin{equation}\label{e46}   \sB_{\ast}^{\s G }= (\lim  \sB^{[j] \s})^{G }  = \lim \sA^{[j]\s} =  \sA^{\s}_\ast \end{equation} 
    \begin{equation}\label{e47} (\widehat{\sB_2}^{\s})_\ast^{  G\times G}= (\lim  \sB_2^{[j] \s})^{G\times G}  = \lim \sA^{[j]\s} =  \sA^{\s}_\ast.
  \end{equation}
En particulier $ \sB_{\ast}^{\s  }$ et $(\widehat{\sB_2}^{\s})_\ast$ 
sont entiers sur $\sA_\ast^{\s}$.    
   Les isomorphismes 
    \begin{equation}\label{e48}   \displaystyle  \sB_{2\ast}^{[j]\s  }   \stackrel{\sim}{\to}  \prod_{G} \sB_\ast^{[j]\s  }  ,\;\; b\otimes b' \mapsto (\gamma(b)b')_{\gamma\in G} \end{equation} 
    induisent \`a la limite un isomorphisme 
      \begin{equation}\label{e49}  \displaystyle (\widehat{\sB_2}^{\s})_\ast \stackrel{\sim}{\to}  \prod_{G} \sB_\ast^{\s}. \end{equation} 
        Modulo $p^m$ cela donne aussi   \begin{equation}\label{e50}  \displaystyle  (\sB^{\s a}/p^m) \otimes_{\sA^{\s a}/p^m} (\sB^{\s a}/p^m)  \stackrel{\sim}{\to}  \prod_{G} (\sB^{\s a}/p^m), \end{equation} 
   \ie $\sB^{\s a}/p^m$ est galoisienne de groupe $G$ sur $\sA^{\s a}/p^m$ dans le cadre $( \sK^{\s}[T^{\e}], \,    (\varpi T)^{\e} \sK^{\s}[T^{\e}])$.
 
 \smallskip En particulier, $(\sB^{\s}/p)[\frac{1}{g}]$ est presque plate sur $(\sA^{\s}/p)[\frac{1}{g}]$ dans le cadre   $( \sK^{\s}[T^{\e}], \,    (\varpi T)^{\e} \sK^{\s}[T^{\e}])$, donc aussi dans le cadre $(\sK^{\s}, \sK^{\ss})$ puisqu'on a invers\'e $g$. Comme $\sB[\frac{1}{g}]$ est \'etale, donc plate, sur $\sA[\frac{1}{g}]$ et que $\sB^{\s}[\frac{1}{g}]$ est sans $\varpi$-torsion, on conclut d'apr\`es \cite[5.2.1]{GR1} que $\sB^{\s}[\frac{1}{g}]$ est presque plate sur $\sA^{\s}[\frac{1}{g}]$ dans le cadre $(\sK^{\s}, \sK^{\ss})$, cadre dans lequel on se place dans les lignes qui suivent.  Donc $\sB^{\s}_2[\frac{1}{g}]= (\sB^{\s}_{2})_\ast[\frac{1}{g}]$ est presque plate sur $\sB^{\s}[\frac{1}{g}]$, en particulier presque sans $\varpi$-torsion, de sorte que la torsion $\varpi$-primaire de $\sB_2^{\s}$ (\resp de $(\sB^{\s}_{2})_\ast = g^{\f}\sB_2^{\s}$) est presque \'egale \`a la torsion $g$-primaire.
 
La prop. \ref{P17} (jointe \`a \eqref{e45}) montre que  $(\sB^{\s}_{2})_{\ast} \to  (\widehat{\sB_2}^{\s})_\ast$ induit un isomorphisme apr\`es inversion de $pg$. D'apr\`es ce qui pr\'ec\`ede, le morphisme compos\'e $ {\sB_{2{\hat\ast}}}Ê:=  (\sB^{\s}_{2})_\ast[\frac{1}{p}] \to  \widehat{\sB_{2{\hat\ast}}}Ê \to (\widehat{\sB_2})_{\hat\ast} = (\widehat{\sB_2}^{\s})_\ast[\frac{1}{p}]$ est donc injectif, de sorte que $ {\sB_{2{\hat\ast}}}$ est s\'epar\'e. 

D'autre part, $\widehat{\sB_{2{\hat\ast}}}Ê\to (\widehat{\sB_2})_{\hat\ast}$ est isom\'etrique: en effet, on d\'eduit du point $(3)$ du lemme \ref{L7} que $(\widehat{\sB_{2{ \ast}}^{\s}})/pÊ\to (\widehat{\sB_2}^{\s})_\ast/p$ est injectif. Donc la norme de  $ {\sB_{2{\hat\ast}}}Ê$ est spectrale tout comme celle de $(\widehat{\sB_2})_{\hat\ast}$. Il suit que $ {\sB_{2{\hat\ast}}}^{\s} =  {\sB_{2{\hat\ast}}}_{\leq 1}Ê$ est presque \'egal \`a $(\sB^{\s}_{2})_{\ast}/(\varpi^\infty$-tors).  On en conclut que $\sB^{\s}_2[\frac{1}{g}]= (\sB^{\s}_{2})_\ast[\frac{1}{g}]$ est presque \'egal \`a $\sB_2[\frac{1}{g}] \cap (\widehat{\sB_2}^{\s})_\ast[\frac{1}{g}]$, et que les idempotents $e_\gamma$ d\'efinissant la d\'ecomposition \eqref{e49} sont dans $\varpi^{\f}\sB_2^{\s}[\frac{1}{g}]$, \ie que $\sB^{\s}[\frac{1}{g}]$ est presque galoisienne de groupe $G$ sur $\sA^{\s}[\frac{1}{g}]$ dans le cadre $(\sK^{\s}, \sK^{\ss})$.

 $b)$ La r\'eduction au cas galoisien se fait comme dans la prop. \ref{P8'} $(2 b)$. 
 
 Ceci termine la preuve de la prop. \ref{L20}, ainsi que celle du th. \ref{T7}.
  \end{proof}
     
     \subsubsection{Remarque.}\label{fr} Il ne semble pas qu'on puisse tirer de l\`a que  $\sB^{\s a}$ est \'etale finie sur $\sA^{\s a}$ (comme $\sA_\ast^{\s}$ est $p$-adiquement hens\'elien, de m\^eme que son extension enti\`ere $\sB_\ast^{\s}$, il suffirait de voir que $\widehat{\sB^{\s}_{2\ast}} $ contient les $e_\gamma \in (\widehat{\sB^{\s}_{2}})_\ast$).  On ne peut invoquer  l'{``\'equivalence remarquable de Grothendieck"} comme en \ref{eqrem} car le r\'esultat de \cite[5.3.27]{GR1} ne s'applique plus: l'id\'eal engendr\'e par $\varpi$ n'est pas {``tight"}\footnote{comme m'en a averti O. Gabber.}. En fait, l'information qui manque ici est la presque finitude de $\sB^{\s }$ sur $\sA^{\s }$: si l'on en disposait, on conclurait par \cite[5.3.20]{GR1} que $\sB^{\s a}$ est projectif fini sur $\sA^{\s a}$, et il serait facile de passer de l\`a \`a {``\'etale fini"}.

      Par ailleurs, tenter de d\'eduire de la prop. \ref{L10} la presque-puret\'e (au sens de presque-injectivit\'e universelle) de $\sB^{\s }$ sur $\sA^{\s }$ se heurte au fait qu'il existe des modules de pr\'esentation finie non s\'epar\'es sur $\sA^{\s}$ pour certaines alg\`ebres perfecto\"{\i}des $\sA$, et m\^eme des id\'eaux principaux non ferm\'es: en $\car\, p$, un exemple est donn\'e par l'id\'eal $(U)$ dans $\sK^{\flat \s} \langle T_i^{\e}, U^{\e} , (\frac{\varpi^{\flat i}T_i}{U})^{\e}\rangle_{i\in \N}$.

  \subsection{Preuve du th\'eor\`eme \ref{T1}.}\label{pT1} Soit $\sA$ une $\sK\langle T^{\e}\rangle$-alg\`ebre perfecto\"{\i}de, et posons $g:= T.1\in \sA^{\s}$.
Soit $\sB'$ une $\sA[\frac{1}{g}]$-alg\`ebre \'etale finie. Notons $\sB^j := \sB' \otimes_{\sA[\frac{1}{g}] } \sA^{[j] } $, $  \sB $ la $\sA$-alg\`ebre presque perfecto\"{\i}de ${\rm{ulim}}\, \sB^{j}$, et $\tilde\sB^{\s} $ (\resp $\tilde\sB$) la fermeture int\'egrale de $g^{\f}\sA^{\s}$ (\resp $g^{\f}\sA^{\s}[\frac{1}{\varpi}]$) dans $\sB'$. 
 Comme la fermeture int\'egrale commute \`a la localisation, on a  $ \, \tilde\sB[\frac{1}{g}]  =  \sA[\frac{1}{g}]^{  +}_{\sB'}= \sB'$, qui n'est autre que $ \sB[\frac{1}{g}]$ d'apr\`es la prop. \ref{P17}.  
   On a $  \sB^{\s} = g^{\f}   \sB^{\s} = (\varpi g)^{\f}   \sB^{\s}$, et $  \sB^{\s}$ est $\varpi$-adiquement compl\`ete.

   \smallskip Le th. \ref{T1} r\'esulte alors de la proposition suivante (l'alg\`ebre not\'ee  $\sB$ dans  \loccit est ici $ \tilde\sB^\natural$).

 \begin{prop}\label{L19} \begin{enumerate} \item On a  $ \, \sB = \tilde \sB$, et $\tilde\sB^{\s}=  (g^{\f}\sA^{\s})^+_{\sB'} = (\varpi g)^{\f} \tilde\sB^{\s} = \hat{\tilde\sB}^{\s}$. 
\item $\tilde\sB^\natural$ est la plus grande $\sA$-alg\`ebre perfecto\"{\i}de contenue dans $\sB'$. 
\item Si $\sB'$ est un $\sA[\frac{1}{g}]$-module fid\`ele, ${\rm{Tr}}_{\sB'/ \sA[\frac{1}{g}]} (\tilde\sB^{\s} )\inj g^{\f}\sA^{\s}$ est un presque-isomorphisme.
\end{enumerate}
\end{prop} 

\begin{proof} $(1)$ On peut remplacer $\sA^{\s}$ par $g^{\f}\sA^{\s}$. Comme $(\sA^{\s})^+_{\sB'} \otimes_{\sA^{\s}} \sA^{[j]{\s}}$ est contenu dans la fermeture int\'egrale de $\sA^{[j]{\s}}$ dans $\sB^j$,  elle-m\^eme contenue dans  $\sB^{j \s}\, $ (lemme \ref{L6} $(1)$), on a un morphisme canonique $\, (\sA^{\s})^+_{\sB'} \to  \sB^{j \s}$. 
 Il est compatible avec les fl\`eches de transition en $j$, d'o\`u, en passant \`a la limite, un morphisme canonique $ (\sA^{\s})^+_{\sB'}  {\to}  \sB^{\s},$ et il suffit de montrer que c'est un isomorphisme. Comme il s'agit de sous-anneaux de $\sB'$, il est injectif, ce qui ram\`ene \`a montrer qu'il est surjectif, ou encore, que $  \sB^{\s}$ est entier sur $\sA^{\s}$, ce qui d\'ecoule de \eqref{e46}.

\smallskip $(2)$ suit de $(1)$ et du cor. \ref{C3}. 

\smallskip $(3)$ Si $\sB'$ est un $\sA[\frac{1}{g}]$-module fid\`ele, $\tilde\sB^{\s}$ est un $\sA^{\s}$-module fid\`ele, et comme le plongement $\sA \to \tilde \sB$ est isom\'etrique, $\tilde\sB^{\s}/p$ est un $\sA^{\s}/p$-module fid\`ele. Alors $\tilde\sB^{\s  a}/p$ est une extension \'etale finie de $\sA^{\s  a}/p$ (prop. \ref{L20}), et la trace fournit un homomorphisme de $(\tilde\sB^{\s  a})_\ast $ sur $(\sA^{\s  a})_\ast = g^{\f}\sA^{\s} $ qui est presque surjectif modulo $p$, donc presque surjectif d'apr\`es la version   \ref{MLN} du lemme de Nakayama.   \end{proof}

 \subsection{Normalisation de Noether et {``enveloppes"} presque perfecto\"{\i}des d'alg\`ebres affino\"{\i}des.}\label{nN}
{ Les constructions d'alg\`ebres perfecto\"{\i}des se bornent habituellement \`a des situations {``toriques"} o\`u l'on met \`a profit la version logarithmique du th\'eor\`eme de {``presque-puret\'e"} de Faltings. 

 Gr\^ace au lemme d'Abhyankar perfecto\"{\i}de, on peut en fait attacher des alg\`ebres perfecto\"{\i}des \`a toute alg\`ebre affino\"{\i}de r\'eduite $ B$ sur $\sK$ de dimension $n$, en consid\'erant $ B$ comme une extension finie de $  A := \sK\langle T_{\leq n}\rangle$ par le lemme de normalisation de Noether \cite[cor. 6.1.2/2]{BGR}. 
  Soit $g\in   A^{\s}  $ tel que $  B[\frac{1}{g}]$ soit \'etale sur $  A[\frac{1}{g}]$. Soient $\sA$ l'alg\`ebre presque perfecto\"{\i}de $  g^{\f}\sK\langle T^{\e}_{\leq n}, g^{\e}\rangle  $, et {\it $\sB$ la fermeture int\'egrale de $ B\otimes_{  A} \sA$ dans $ B\otimes_{  A} \sA[\frac{1}{g}]$. Alors $\sB$ est presque perfecto\"{\i}de} (prop. \ref{L19}), et $\sB^{\s a}$ est \'etale fini sur $\sA^{\s a}\,$ modulo toute puissance de $p$ dans le cadre $\,(  \sK^{\s}[T^{\e}], (\varpi T)^{\e}\sK^{\s}[T^{\e}])$, (via $T^{\frac{1}{p^k}}\mapsto g^{\frac{1}{p^k}}$).} 
  
  Si la question \ref{q} a une r\'eponse positive, $\sA$ et $\sB$ sont perfecto\"{\i}des.

    \bigskip 
   
         \section*{Index des symboles}

 \medskip
 
  $(\; )^\dagger,\; (\; )^+ , \;(\; )^\ast$  \dotfill \ref{ferm}, 
 
\smallskip $(\; )^{a},\; (\; )_\ast , \;(\; )_!, \; (\;)_{!!}$  \dotfill \ref{Mod}, \; \ref{Alg}, 
 
\smallskip $(\; )^{\hat a},\; (\; )_{\hat \ast} ,  \; (\;)_{\hat{!!}}$  \dotfill \ref{recc}, \; \ref{234}, 

\smallskip $(\;)^u, \; \hat\otimes^u$  \dotfill \ref{abu},  
 
 \smallskip  $ {\rm{ulim}} , {\rm{ucolim}}$  \dotfill \ref{ulim}, \; \ref{{ucolim}},   

 \smallskip $\flat, \;\#,\; \sharp$   \dotfill \ref{basc}, \; \ref{bascu},

\smallskip $\natural$   \dotfill \ref{nat}, \; \ref{app},  

  \smallskip $g^{\f}(\; )$   \dotfill \ref{E6}, \; \ref{prp},
 
   \smallskip $(\; )\langle g^{\e}\rangle$ \dotfill \ref{E9}, \; \ref{r8},
 
 \smallskip $(\; )\{  \frac{\lambda}{g}\}$   \dotfill \ref{loc},\;  \ref{locp},  
 
    \smallskip $(\; )^{[j]}$   \dotfill \ref{lu}, \; \ref{ec}.

      \end{sloppypar}
            \end{document}